\documentclass[11pt,a4paper]{article}

\usepackage{amsfonts,amsgen,amsmath,amstext,amsbsy,amsopn,amsthm,cases,listings,
}
\usepackage{multirow}
\usepackage{epstopdf}
\usepackage{graphicx}   
\usepackage{pgf,tikz}
\usepackage{mathrsfs}
\usepackage[marginal]{footmisc}
\usepackage{enumerate}
\usepackage[titletoc]{appendix}
\usepackage{booktabs}
\usepackage{url}
\usepackage{mathtools}

\usepackage{pgfplots}
\usepackage{authblk}
\usepackage{amssymb}

\usepackage{wasysym}

\usepackage{empheq}

\usepackage{dsfont}

\usepackage{tikz}
\usepackage{longtable}
\usepackage{float}
\usepackage{subcaption}
\pgfplotsset{compat=1.18}
\usepackage{mathrsfs}
\usepackage{wasysym} 
\usetikzlibrary{arrows}
\usepackage{aligned-overset}
\usepackage{bm}
\usepackage{bbm}
\usepackage[T1]{fontenc}
\usepackage[backref=page]{hyperref}

\usepackage[thinc]{esdiff}

\allowdisplaybreaks[1]

\definecolor{uuuuuu}{rgb}{0.27,0.27,0.27}
\definecolor{sqsqsq}{rgb}{0.1255,0.1255,0.1255}

\usepackage[a4paper,margin=2.5cm]{geometry}

\newtheorem{definition}{Definition} [section]
\newtheorem{theorem}[definition]{Theorem}
\newtheorem{lemma}[definition]{Lemma}
\newtheorem{proposition}[definition]{Proposition}

\newtheorem{claim}[definition]{Claim}
\newtheorem{problem}[definition]{Problem}

\newtheorem{fact}[definition]{Fact}

\newcommand{\norm}[1]{\left\lVert#1\right\rVert}

\begin{document}
\title{\bf\Large Exact Tur\'{a}n number of the Fano plane in the $\ell_2$-norm}
\date{\today}
\author[1]{Jianfeng Hou\thanks{Email: \texttt{jfhou@fzu.edu.cn}}}
\author[2]{Xizhi Liu\thanks{Email: \texttt{liuxizhi@ustc.edu.cn}}}
\author[1]{Yixiao Zhang\thanks{Email: \texttt{fzuzyx@gmail.com}}}
\affil[1]{\small Center for Discrete Mathematics, Fuzhou University, Fuzhou, China}
\affil[2]{\small School of Mathematical Sciences, University of Science and Technology of China, Hefei, China}
\maketitle
\begin{abstract}
    A classical object in hypergraph Tur\'{a}n theory is the Fano plane $\mathbb{F}$, the unique linear $3$-graph on seven vertices with seven edges. The Tur\'{a}n density and exact Tur\'{a}n number of $\mathbb{F}$, first proposed as a problem by S\'{o}s~\cite{Sos76} in the 1970s, were determined through a sequence of works by De Caen--F\"{u}redi~\cite{DCF00}, F\"{u}redi--Simonovits~\cite{FS05}, Keevash--Sudakov~\cite{KS05}, and Bellmann--Reiher~\cite{BR19}. 
    
    Confirming a conjecture of Balogh--Clemen--Lidick\'{y}~{\cite[Conjecture~3.1]{BCL22a}} in a stronger form, we establish an Andr\'{a}sfai--Erd\H{o}s--S\'{o}s--type stability theorem for $\mathbb{F}$ in the $\ell_2$-norm: there exists a positive constant $\varepsilon$ such that for large $n$, every $\mathbb{F}$-free $3$-graph on $n$ vertices with minimum $\ell_2$-norm degree at least $(5/4 - \varepsilon)n^3$ must be bipartite. 
    As a consequence, for large $n$, the balanced complete bipartite $3$-graph is the unique extremal construction for the $\ell_{2}$-norm Tur\'{a}n problem of $\mathbb{F}$. 

    Our proof includes a refinement of a classical result by Ahlswede--Katona~\cite{AK78} on counting stars, and the establishment of an Andr\'{a}sfai--Erd\H{o}s--S\'{o}s--type theorem for a multigraph Tur\'{a}n problem studied by Bellmann--Reiher~\cite{BR19}, both of which are of independent interest. 
\end{abstract}

\section{Introduction}\label{SEC:Introduction}
Given an integer $r\ge 2$, an \emph{$r$-uniform hypergraph} (henceforth an \emph{$r$-graph}) $\mathcal{H}$ is a collection of $r$-subsets of some set $V$. We call $V$ the \emph{vertex set} of $\mathcal{H}$ and denote it by $V(\mathcal{H})$. 
The size of $V(\mathcal{H})$ is denoted by $v(\mathcal{H})$. 
We identify a hypergraph $\mathcal{H}$ with its set of edges, and hence, $|\mathcal{H}|$ represents the number of edges in $\mathcal{H}$. 

Given an $r$-graph $\mathcal{H}$, the \emph{shadow} of $\mathcal{H}$ is defined as 
\begin{align*}
    \partial\mathcal{H}
    \coloneqq \left\{e\in \tbinom{V(\mathcal{H})}{r-1} \colon \text{there exists an edge $E\in \mathcal{H}$ containing $e$} \right\}. 
\end{align*}
The \emph{link} of a vertex $v \in V(\mathcal{H})$ is given by 
\begin{align*}
    L_{\mathcal{H}}(v)
    \coloneqq \left\{e \in \partial\mathcal{H} \colon \{v\} \cup e \in \mathcal{H}\right\}, 
\end{align*}
and the \emph{degree} of $v$ is $d_{\mathcal{H}}(v) \coloneqq |L_{\mathcal{H}}(v)|$. 

For an $(r-1)$-set $T \subseteq V(\mathcal{H})$, the \emph{neighborhood} of $T$ in $\mathcal{H}$ is defined as 
\begin{align*}
    N_{\mathcal{H}}(T) 
    \coloneqq \left\{v\in V(\mathcal{H}) \colon T \cup \{v\} \in \mathcal{H}\right\}, 
\end{align*}
and the \emph{codegree} of $T$ is $d_{\mathcal{H}}(T) \coloneqq |N_{\mathcal{H}}(T)|$. 

Following the definitions in~\cite{BCL22b,CILLP24}, for every real number $p \ge 1$, the \emph{$\ell_p$-norm} of an $r$-graph $\mathcal{H}$ is defined as 
\begin{align*}
    \norm{\mathcal{H}}_{p}
    \coloneqq \sum_{e\in \partial\mathcal{H}} d^{p}_{\mathcal{H}}(e), 
\end{align*}
where $d^{p}_{\mathcal{H}}(e)$ is an abbreviation for $\left(d_{\mathcal{H}}(e)\right)^{p}$. 
Note that for $p=1$, we have $\norm{\mathcal{H}}_{1} = r \cdot |\mathcal{H}|$.

Extending the definition of degree, the \emph{$\ell_p$-norm degree} of a vertex $v\in V(\mathcal{H})$ in $\mathcal{H}$, denoted by $d_{p,\mathcal{H}}(v)$, is defined as 
\begin{align*}
    d_{p,\mathcal{H}}(v)
    \coloneqq \norm{\mathcal{H}}_{p} - \norm{\mathcal{H} - v}_{p},
\end{align*}
where $\mathcal{H}-v$ denotes the $r$-graph obtained from $\mathcal{H}$ by removing the vertex $v$ and all edges containing $v$. 
Denote by $\delta_{\ell_p}(\mathcal{H})$ the \emph{minimum $\ell_p$-norm degree} of $\mathcal{H}$.

Given a family $\mathcal{F}$ of $r$-graphs, an $r$-graph $\mathcal{H}$ is \emph{$\mathcal{F}$-free}
if it does not contain any member of $\mathcal{F}$ as a subgraph.
The \emph{$\ell_p$-norm Tur\'{a}n number} $\mathrm{ex}_{\ell_p}(n, \mathcal{F})$ of $\mathcal{F}$ is the maximum $\ell_p$-norm of an $\mathcal{F}$-free $r$-graph $\mathcal{H}$ on $n$ vertices. 
The \emph{$\ell_p$-norm Tur\'{a}n density} of $\mathcal{F}$ is defined as 
\begin{align*}
    \pi_{\ell_p}(\mathcal{F})
    \coloneqq \lim_{n \to \infty} \tfrac{\mathrm{ex}_{\ell_p}(n, \mathcal{F})}{n^{r-1+p}}. 
\end{align*}
The existence of this limit follows from a simple averaging argument which can be found in such as~\cite{KNS64},~{\cite[Proposition~1.8]{BCL22b}}, and~{\cite[Proposition~2.2]{CILLP24}}. 
%

Recall that the ordinary Tur\'{a}n number $\mathrm{ex}(n, \mathcal{F})$ of $\mathcal{F}$ is the maximum number of edges in an $n$-vertex $\mathcal{F}$-free $r$-graph, and the ordinary Tur\'{a}n density of $\mathcal{F}$ is $\pi(\mathcal{F}) \coloneqq \lim_{n\to \infty}\mathrm{ex}(n, \mathcal{F})/\tbinom{n}{r}$. 
Note that the $\ell_1$-norm Tur\'{a}n number of $\mathcal{F}$ is simply $r$ times its ordinary Tur\'{a}n number, and $\pi(\mathcal{F}) = (r-1)! \cdot \pi_{\ell_1}(\mathcal{F})$.  

For $r=2$ (i.e., graphs), the value of $\pi(\mathcal{F})$ is well understood thanks to the celebrated general theorem of Erd\H{o}s--Stone~\cite{ES46} (see also~\cite{ES66}), which extends Tur\'{a}n's seminal theorem on complete graphs~\cite{Tur41} to arbitrary graph families.
The study of $\pi_{\ell_p}(\mathcal{F})$ for $p > 1$ was initiated by Caro--Yuster~\cite{CY00,CY04}, and Erd\H{o}s--Stone-type results in this setting were later obtained by Bollob\'{a}s--Nikiforov~\cite{BN12}. 

For $r \ge 3$, determining $\pi(\mathcal{F})$ is already notoriously difficult in general, despite significant effort devoted to this area. 
For results up to~2011, we refer the reader to the excellent survey by Keevash~\cite{Kee11}. 
Very recently, Balogh--Clemen--Lidick\'{y}~\cite{BCL22a,BCL22b} initiated the study of $\pi_{\ell_p}(\mathcal{F})$ for hypergraph families. 
Among their many results, they determined the values of $\pi_{\ell_2}(K_{4}^{3})$ and $\pi_{\ell_2}(K_{5}^{3})$, utilizing computer-assisted flag algebra computations, a powerful tool first introduced by~\cite{Raz07}. 
These results are particularly interesting given the notorious difficulty of determining $\pi(K_{\ell}^{3})$ for any $\ell \ge 4$, a problem originally posed by Tur\'{a}n~\cite{Tur41}. 
Subsequently, Bodn\'{a}r et al.~\cite{BCDHLSYZ25} and Chen et al.~\cite{CDHLZ26} determined the exact $\ell_2$-norm Tur\'{a}n numbers of $K_4^3$ and $K_5^3$, respectively, for all sufficiently large $n$. 
The $\ell_p$-norm Tur\'{a}n problems for hypergraphs have been explored more systematically in recent works such as~\cite{CL24,CILLP24,GLMP25}.

\begin{figure}[H]
    \centering
    \tikzset{every picture/.style={line width=1pt}} 
    \begin{tikzpicture}[x=0.75pt,y=0.75pt,yscale=-1,xscale=1,line join=round]
    \draw   (389.01,216.91) .. controls (369.98,216.81) and (354.71,184.29) .. (354.92,144.28) .. controls (355.12,104.27) and (370.71,71.91) .. (389.74,72.01) .. controls (408.77,72.11) and (424.03,104.62) .. (423.83,144.63) .. controls (423.63,184.64) and (408.04,217) .. (389.01,216.91) -- cycle ;
    \draw   (496.07,218.55) .. controls (477.04,218.56) and (461.59,186.13) .. (461.57,146.12) .. controls (461.55,106.11) and (476.96,73.66) .. (495.99,73.65) .. controls (515.02,73.64) and (530.47,106.07) .. (530.49,146.08) .. controls (530.51,186.1) and (515.1,218.54) .. (496.07,218.55) -- cycle ;
    \draw [fill=uuuuuu, fill opacity=0.5]  (393,118) .. controls (431,116) and (479,115) .. (497,98)  .. controls (480,110) and (484,127) .. (495,138) .. controls (478,120) and (434,122) .. (393,118) -- cycle;
    \draw  [fill=uuuuuu, fill opacity=0.5]  (389,157) .. controls (425,172) and (472,165) .. (491,165) .. controls (460,171) and (427,172) .. (391,191) .. controls (404,175) and (402,169) .. (389,157) --cycle;
    \draw   (134,167) .. controls (134,142.15) and (154.15,122) .. (179,122) .. controls (203.85,122) and (224,142.15) .. (224,167) .. controls (224,191.85) and (203.85,212) .. (179,212) .. controls (154.15,212) and (134,191.85) .. (134,167) -- cycle ;
    \draw   (179,76) -- (259,212) -- (99,212) -- cycle ;
    \draw    (179,76) -- (179,212) ;
    \draw    (99,212) -- (218,144) ;
    \draw    (141,145) -- (259,212) ;
    \draw [fill=uuuuuu] (393,118) circle (1.2pt);
    \draw [fill=uuuuuu]  (389,157) circle (1.2pt);
    \draw [fill=uuuuuu] (491,165) circle (1.2pt);
    \draw [fill=uuuuuu] (391,191) circle (1.2pt);
    \draw [fill=uuuuuu] (497,98) circle (1.2pt);
    \draw [fill=uuuuuu] (495,138) circle (1.2pt);
    \draw (380,226) node [anchor=north west][inner sep=0.75pt]   [align=left] {$V_1$};
    \draw (486,226) node [anchor=north west][inner sep=0.75pt]   [align=left] {$V_2$};
    \draw [fill=uuuuuu] (179,76) circle (2pt);
    \draw (173,58) node [anchor=north west][inner sep=0.75pt]   [align=left] {$1$};
    \draw [fill=uuuuuu] (218,144) circle (2pt);
    \draw (225,134) node [anchor=north west][inner sep=0.75pt]   [align=left] {$2$};
    \draw [fill=uuuuuu] (259,212) circle (2pt);
    \draw (253,216) node [anchor=north west][inner sep=0.75pt]   [align=left] {$3$};
    \draw [fill=uuuuuu] (179,212) circle (2pt);
    \draw (174,216) node [anchor=north west][inner sep=0.75pt]   [align=left] {$4$};
    \draw [fill=uuuuuu] (99,212) circle (2pt);
    \draw (93,216) node [anchor=north west][inner sep=0.75pt]   [align=left] {$5$};
    \draw [fill=uuuuuu] (140,145) circle (2pt);
    \draw (123,134) node [anchor=north west][inner sep=0.75pt]   [align=left] {$6$};
    \draw [fill=uuuuuu] (179,167) circle (2pt);
    \draw (181,173) node [anchor=north west][inner sep=0.75pt]   [align=left] {$7$};
    \end{tikzpicture}
    \caption{The Fano plane and the balanced complete bipartite $3$-graph $\mathbb{B}_{n}$.}
    \label{fig:Fano}
\end{figure}
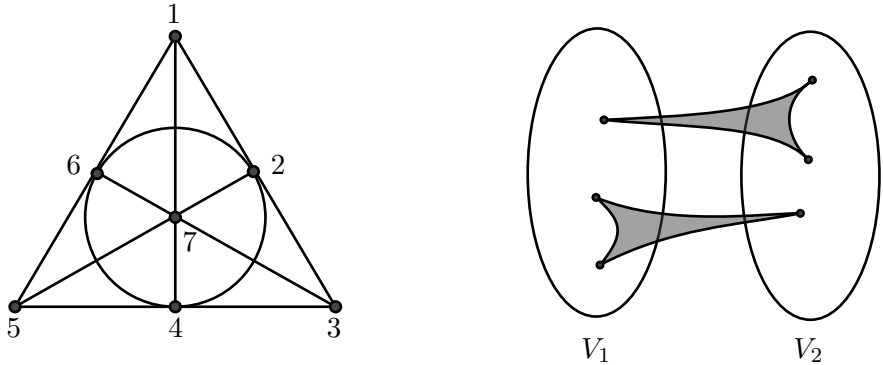

In this work, we focus on the following classical object in hypergraph Tur\'{a}n theory. 
Let $\mathbb{F}$ denote the $3$-graph (the Fano plane, see Figure~\ref{fig:Fano}) on $\{1,\ldots, 7\}$ with $7$ edges: 
\begin{align*}
    \big\{ \{1,2,3\},~\{3,4,5\},~\{5,6,1\},~\{1,7,4\},~\{2,7,5\},~\{3,7,6\},~\{2,4,6\} \big\}.
\end{align*}

A $3$-graph $\mathcal{H}$ is \emph{bipartite} if there exists a partition $V_1 \cup V_2 = V(\mathcal{H})$ such that every edge in $\mathcal{H}$ intersects both $V_1$ and $V_2$. 
Let $\mathbb{B}_{n}$ denote the balanced complete bipartite $3$-graph on $n$ vertices. 
Simple calculations show that $\delta_{\ell_2}(\mathbb{B}_{n}) = \left( 5/4 + o_{n}(1) \right)n^3$. 

Observe that every bipartite $3$-graph is $\mathbb{F}$-free, and hence, $\pi(\mathbb{F}) \ge \lim_{n\to \infty} |\mathbb{B}_{n}|/\tbinom{n}{3} = 3/4$. 
It was conjectured by S\'{o}s~\cite{Sos76} in the 1970s and famously resolved by De Caen--F\"{u}redi in an elegant paper~\cite{DCF00} that $\pi(\mathbb{F}) = 3/4$. 
Later, using Simonovits' stability method, F\"{u}redi--Simonovits~\cite{FS05}, and independently, Keevash--Sudakov~\cite{KS05} proved that $\mathrm{ex}(n,\mathbb{F}) = |\mathbb{B}_{n}|$ for all sufficiently large $n$.
The determination of $\mathrm{ex}(n,\mathbb{F})$ for all $n$ was achieved recently by Bellmann--Reiher~\cite{BR19}. 

In~{\cite[Conjecture~3.1]{BCL22a}}, Balogh--Clemen--Lidick\'{y} conjectured that for sufficiently large $n$, the unique extremal construction for the $\ell_2$-norm Tur\'{a}n problem of $\mathbb{F}$ remains $\mathbb{B}_{n}$,  which, if true, would imply that $\pi_{\ell_2}(\mathbb{F}) = 5/16$.  
They observed the upper bound $\pi_{\ell_2}(\mathbb{F}) \le 3/8$, which follows from the simple inequality:
\begin{align*}
    \norm{\mathcal{H}}_{2}
    = \sum_{e\in \partial\mathcal{H}}d^{2}_{\mathcal{H}}(e)
    \le (n-2) \cdot \sum_{e\in \partial\mathcal{H}}d_{\mathcal{H}}(e)
    = 3 (n-2) |\mathcal{H}|. 
\end{align*}
This upper bound was later improved by Brooks--Linz~{\cite[Proposition~5.1]{BL23}} to $\pi_{\ell_2}(\mathbb{F}) \le 11/32$. 
Our main result in this work is a confirmation of~{\cite[Conjecture~3.1]{BCL22a}}. 

\begin{theorem}\label{THM:AES-Fano-L2norm}
    There exist constants $\varepsilon > 0$ and $N_0$ such that the following holds for all $n \ge N_0$. 
    Suppose that $\mathcal{H}$ is an $\mathbb{F}$-free $3$-graph on $n$ vertices with $\delta_{\ell_{2}}(\mathcal{H}) \ge \left(5/4 - \varepsilon\right) n^3$. Then $\mathcal{H}$ is bipartite. 
    Consequently, for all sufficiently large $n$, we have $\mathrm{ex}_{\ell_2}(n,\mathbb{F}) = \norm{\mathbb{B}_{n}}_{2}$, and moreover, $\mathbb{B}_{n}$ is the unique extremal construction. 
\end{theorem}

Recall the classical Andr\'{a}sfai--Erd\H{o}s--S\'{o}s Theorem, which states that for $\ell \ge 2$, every $K_{\ell+1}$-free graph on $n$ vertices with minimum degree greater than $\tfrac{3\ell-4}{3\ell-1}n$ is $\ell$-partite. 
Theorem~\ref{THM:AES-Fano-L2norm} can be viewed as a generalized Andr\'{a}sfai--Erd\H{o}s--S\'{o}s theorem:  instead of considering the minimum degree, we consider the minimum $\ell_{2}$-norm degree. 
Analogous results (using different types of minimum degrees) have also been obtained in recent works such as~\cite{CL24,CILLP24,CHHL24}. 

There are several key ingredients in our proof of Theorem~\ref{THM:AES-Fano-L2norm}. 
First, we establish an Andr\'{a}sfai--Erd\H{o}s--S\'{o}s--type stability theorem (Theorem~\ref{THM:AES-5-multigraph}) for a multigraph Tur\'{a}n problem, previously considered by Bellmann-Reiher~\cite{BR19} in determining the Tur\'{a}n number of $\mathbb{F}$. 
Second, we derive several upper bounds on the density of two-edge stars in a graph with given edge density, independence number, and vertex degrees, refining the classical theorem of Ahlswede--Katona~\cite{AK78}. 
Combining these two results, we establish an Erd\H{o}s--Simonovits~\cite{Sim68}-type stability theorem for $\mathbb{F}$-free $3$-graphs in the $\ell_2$-norm. 
Finally, we prove the vertex-extendability of $\mathbb{F}$ with respect to the class of bipartite $3$-graphs in $\ell_{2}$-norm. 
Together with the Erd\H{o}s--Simonovits~\cite{Sim68}-type stability theorem and the general framework developed in~\cite{CL24}, this leads to the Andr\'{a}sfai--Erd\H{o}s--S\'{o}s--type stability theorem for $\mathbb{F}$-free $3$-graphs in the $\ell_2$-norm as stated in Theorem~\ref{THM:AES-Fano-L2norm}.  

In the next section, we present several definitions and preliminary results related to graphs, multigraphs, and hypergraphs. 
In Section~\ref{SEC:density-edge-stability}, we develop several technical tools concerning the $\ell_{2}$-norm Tur\'{a}n problem of $\mathbb{F}$ and prove an Erd\H{o}s--Simonovits~\cite{Sim68}-type stability theorem for $\mathbb{F}$-free $3$-graphs in the $\ell_2$-norm. 
In Section~\ref{SEC:exact-degree-stability}, we establish the vertex-extendability of $\mathbb{F}$ in $\ell_2$-norm and complete the proof of Theorem~\ref{THM:AES-Fano-L2norm}. 
In Section~\ref{SEC:Multigraphs}, we prove an Andr\'{a}sfai--Erd\H{o}s--S\'{o}s--type stability theorem (Theorem~\ref{THM:AES-5-multigraph}) for the multigraph Tur\'{a}n problem studied by Bellmann-Reiher~\cite{BR19}, which also serves as a key ingredient in the proof of Theorem~\ref{THM:AES-Fano-L2norm}.  
Section~\ref{SEC:Remarks} concludes with some remarks. 
Throughout the remainder of the paper, we omit floor and ceiling functions when they are not essential to the arguments.  
\section{Preliminaries}\label{SEC:Preliminaries}
Given an $r$-graph $\mathcal{H}$ and a vertex set $S\subseteq V(\mathcal{H})$, we use $\mathcal{H}[S]$ to denote the \emph{induced subgraph} of $\mathcal{H}$ on $S$, and $\mathcal{H} - S$ to denote the induced subgraph of $\mathcal{H}$ on $V(\mathcal{H})\setminus S$.

Let $r\ge 2$ and $k \ge 1$ be integers. 
Denote by $S_{k}^{r}$ the $r$-uniform $k$-edge star whose center has size $r-1$, i.e., the $r$-graph on $r-1+k$ vertices with edge set 
\begin{align*}
    \big\{ \{1,\ldots, r-1, r\},~\ldots,~\{1,\ldots, r-1, r-1+k\} \big\}. 
\end{align*}
For convenience, we omit the superscript $r$ when $r=2$, and set $\mathbb{S}_{2} \coloneqq S_{2}^{3}$. 

Given two $r$-graphs $Q$ and $\mathcal{H}$, a map $\phi \colon V(Q) \to V(\mathcal{H})$ is a \emph{homomorphism} from $Q$ to $\mathcal{H}$ if $\phi(e) \in \mathcal{H}$ for all $e \in Q$. 
If, in addition, $\phi$ is injective, then we say that $\phi$ is an \emph{embedding} of $Q$ into $\mathcal{H}$. 
We say that $Q$ is \emph{$\mathcal{H}$-colorable} if there exists a homomorphism from $Q$ to $\mathcal{H}$.

Denote by $\mathrm{Hom}(Q, \mathcal{H})$ and $\mathrm{Emb}(Q, \mathcal{H})$ the sets of all homomorphisms and embeddings from $Q$ to $\mathcal{H}$, respectively (for clarity, let us assume that both $Q$ and $\mathcal{H}$ are vertex-labeled, with distinct vertices assigned distinct labels). 
The \emph{automorphism group} of $Q$, denoted by $\mathrm{Aut}(Q)$, is simply $\mathrm{Emb}(Q, Q)$. 
Let $\mathrm{hom}(Q, \mathcal{H}) \coloneqq |\mathrm{Hom}(Q, \mathcal{H})|$ and $\mathrm{emb}(Q, \mathcal{H}) \coloneqq |\mathrm{Emb}(Q, \mathcal{H})|$.
The number of (copies of) $Q$ in $\mathcal{H}$ is defined as $\mathrm{N}(Q, \mathcal{H}) \coloneqq \mathrm{emb}(Q, \mathcal{H})/|\mathrm{Aut}(Q)|$. 

For a vertex $v\in V(\mathcal{H})$, the \emph{$Q$-degree} of $v$ in $\mathcal{H}$, denoted by $d_{Q,\mathcal{H}}(v)$, is defined as 
\begin{align*}
    d_{Q,\mathcal{H}}(v)
    \coloneqq \mathrm{N}(Q, \mathcal{H}) - \mathrm{N}(Q, \mathcal{H}-v),
\end{align*}
where recall that $\mathcal{H}-v$ is the $r$-graph obtained from $\mathcal{H}$ by removing the vertex $v$ and all edges containing $v$.
Note that the combinatorial meaning of $d_{Q,\mathcal{H}}(v)$ is the number of copies of $Q$ in $\mathcal{H}$ that contain $v$. 
Thus, a simple double-counting argument yields
\begin{align}\label{equ:sum-Q-degree}
    \sum_{v\in V(\mathcal{H})} d_{Q,\mathcal{H}}(v)
    = v(Q) \cdot \mathrm{N}(Q, \mathcal{H}). 
\end{align}

Recall that for every real number $p \ge 1$, the \emph{$\ell_p$-norm degree} of a vertex $v\in V(\mathcal{H})$ in $\mathcal{H}$, denoted by $d_{p,\mathcal{H}}(v)$, is defined as 
\begin{align}\label{equ:def-p-norm-degree}
    d_{p,\mathcal{H}}(v)
    \coloneqq \norm{\mathcal{H}}_{p} - \norm{\mathcal{H} - v}_{p}.
\end{align}

Let $p \ge 2$ be an integer. 
The number of $S_{p}^{r}$ in an $r$-graph $\mathcal{H}$ can be expressed as 
\begin{align*}
    \mathrm{N}(S_{p}^{r}, \mathcal{H}) 
    = \sum_{e\in \partial\mathcal{H}} \tbinom{d_{\mathcal{H}}(e)}{p}
    & = \tfrac{1}{p!} \sum_{e\in \partial\mathcal{H}} d_{\mathcal{H}}(e) \left(d_{\mathcal{H}}(e) - 1\right) \cdots \left(d_{\mathcal{H}}(e) - p + 1\right) \\
    & = \tfrac{1}{p!} \sum_{e\in \partial\mathcal{H}} \sum_{i=0}^{p}s(p,i) \cdot d^{i}_{\mathcal{H}}(e)
    = \tfrac{1}{p!} \sum_{i=1}^{p} s(p,i) \cdot \norm{\mathcal{H}}_{i}, 
\end{align*}
where $s(p,i)$ is \emph{the signed Stirling number of the first kind}, and in particular, $s(p,0) = 0$.

Conversely, $\norm{\mathcal{H}}_{p}$ can be expressed as 
\begin{align*}
    \norm{\mathcal{H}}_{p}
    = \sum_{e\in \partial\mathcal{H}} d^{p}_{\mathcal{H}}(e)
    = \sum_{e\in \partial\mathcal{H}} \sum_{i=0}^{p} S(p,i) \cdot i! \cdot \tbinom{d_{\mathcal{H}}(e)}{i}
    = r|\mathcal{H}| + \sum_{i=2}^{p} S(p,i) \cdot i! \cdot \mathrm{N}(S_{i}^{r}, \mathcal{H}), 
\end{align*}
where $S(p,i)$ is \emph{the Stirling number of the second kind}, and in particular, $S(p,0) = 0$. 

In particular, for $p = 2$ and $r = 2$, we have 
\begin{align}\label{equ:2-norm-in-S2-edge-gp}
    \mathrm{N}(S_2, G) 
    = \tfrac{\norm{G}_{2} - 2 |G|}{2} 
    \quad\text{and}\quad 
    \norm{G}_{2}
    = 2 \mathrm{N}(S_2, G) + 2 |G|,  
\end{align}
and for $p = 2$ and $r = 3$, we have 
\begin{align}\label{equ:2-norm-in-S2-edge-3gp}
    \mathrm{N}(\mathbb{S}_{2}, \mathcal{H}) 
    = \tfrac{\norm{\mathcal{H}}_{2} - 3 |\mathcal{H}|}{2} 
    \quad\text{and}\quad 
    \norm{\mathcal{H}}_{2}
    = 2 \mathrm{N}(\mathbb{S}_{2}, \mathcal{H}) + 3 |\mathcal{H}|. 
\end{align}

\subsection{Counting stars in graphs}\label{SUBSEC:graphs}
We will use the following special case of the classical Andr\'{a}sfai--Erd\H{o}s--S\'{o}s Theorem in Section~\ref{SEC:Multigraphs}. 

\begin{theorem}[Andr\'{a}sfai--Erd\H{o}s--S\'{o}s~\cite{AES74}]\label{THM:AES-Theorem}
    Every $K_3$-free graph $G$ on $n$ vertices with $\delta(G) > 2n/5$ is bipartite.
\end{theorem}

Given a graph $G$ on $n$ vertices with $m$ edges, the classical theorem of Ahlswede--Katona~\cite{AK78} provides a tight upper bound for the maximum number of $S_2$ in $G$. 
For our purposes, we will use the following asymptotic version of their result. 

\begin{theorem}[Ahlswede--Katona~\cite{AK78}]\label{THM:max-S2-density}
    Suppose that $G$ is a graph on $n$ vertices with $x n^2$ edges. 
    Then 
    \begin{align*}
        \tfrac{\mathrm{N}(S_{2}, G)}{n^3}
        \le \max\left\{ \tfrac{(1-2 x)^{3/2}+4 x-1}{2},~\sqrt{2} x^{3/2} \right\} + o_{n}(1).
    \end{align*}
    Consequently, by~\eqref{equ:2-norm-in-S2-edge-gp},
    \begin{align*}
        \tfrac{\norm{G}_{2}}{n^3}
        \le \max\left\{ (1-2 x)^{3/2}+4 x-1,~2\sqrt{2} x^{3/2} \right\} + o_{n}(1).
    \end{align*}
\end{theorem}

Before introducing the next result, which refines Theorem~\ref{THM:max-S2-density}, let us briefly describe the constructions that achieve asymptotically the bounds in Theorem~\ref{THM:max-S2-density}. 

Given positive integers $n \ge k \ge 0$, define the following graphs. 
\begin{itemize}
    \item Let $C(n,k)$ denote the $n$-vertex graph consisting of a copy of $K_{k}$ and $n-k$ isolated vertices. 
    \item Let $S(n,k)$ denote the $n$-vertex graph that is the complement of $C(n,k)$. 
\end{itemize} 

Fix $x \in [0,1/2]$ and let 
\begin{align*}
    \beta_{1} \coloneqq \sqrt{1-2x} \quad\text{and}\quad
    \beta_2 \coloneqq \sqrt{2x}. 
\end{align*}
The upper bounds in Theorem~\ref{THM:max-S2-density} are asymptotically attained by the constructions $S\left(n, \beta_1 n \right)$ for $x \in [0, 1/4]$, and $C\left(n, \beta_2 n \right)$ for $x \in [1/4, 1/2]$. 

Unfortunately, in certain cases (e.g. Lemma~\ref{LEMMA:main-lemma-sum-important}), the bounds in Theorem~\ref{THM:max-S2-density} are insufficient to provide the desired lower bound on the number of edges in $G$. 
Therefore, a refinement of the theorem is needed. 
Note that for $x \in [1/4, 1/2]$, the extremal construction $C\left(n, \beta_2 n \right)$ has the property that in every independent set, all but at most one vertex have degree zero. 
This observation suggests that if we additionally assume the graph $G$ contains a large independent set in which every vertex has large degree, then the bound in Theorem~\ref{THM:max-S2-density} can be substantially improved. 
The following result confirms that this is indeed the case for $x \in [17/50, 7/20]$ (although we believe it should be the case for all $x > 1/4$, we are currently unable to prove this more general statement). 
The proof of the following result is rather tedious and involves extensive calculations; see~\cite{HLZ25star}. 

\begin{proposition}[{\cite[Theorems~1.2 and~1.3]{HLZ25star}}]\label{PROP:graph-max-S2} 
    Suppose that $G$ is an $n$-vertex graph with $x n^2$ edges, where $x$ is a real number in $\left[\tfrac{17}{50}, \tfrac{7}{20}\right]$, and suppose that $I \subseteq V(G)$ is an independent set of size $\alpha n$.
    If either of the following conditions holds$\colon$
    \begin{enumerate}[(i)]
        \item\label{PROP:graph-max-S2-07} $\alpha \in \left[\tfrac{17}{100}, \tfrac{23}{100}\right]$ and $d_{G}(v) \ge \alpha n$ for every $v\in I$; 
        \item\label{PROP:graph-max-S2-022-033} $\alpha \in \left[\tfrac{1}{3}, \tfrac{2}{5}\right]$ and $d_{G}(v) \ge \tfrac{n}{5}$ for every $v\in I$,  
    \end{enumerate}
    then 
    \begin{align*}
        \tfrac{\mathrm{N}(S_{2}, G)}{n^3}
        \le \max\left\{\tfrac{\alpha^3 + \left(2x-\alpha^2\right)\left(2x+\alpha^2\right)^{1/2}}{2},~\tfrac{\left(1-2x\right)^{3/2}+4x-1}{2}\right\} + o_{n}(1). 
    \end{align*}
    Consequently, by~\eqref{equ:2-norm-in-S2-edge-gp},
    \begin{align*}
        \tfrac{\norm{G}_{2}}{n^3}
        \le \max\left\{ \alpha^3 + \left(2x-\alpha^2\right)\left(2x+\alpha^2\right)^{1/2},~\left(1-2x\right)^{3/2}+4x-1\right\} + o_{n}(1). 
    \end{align*}
\end{proposition} 

Let us briefly explain the constructions corresponding to the bounds in Proposition~\ref{PROP:graph-max-S2}. 

Given positive integers $n, k, \ell$, let $\hat{S}(n, k, \ell)$ denote the $n$-vertex graph formed by the union of $S(k+\ell,k)$ and $n-k-\ell$ additional isolated vertices.

Fix $x \in [0,1/2]$ and $\alpha \in [0,1]$. 
Let 
\begin{align*} 
    \beta_3 \coloneqq \sqrt{\alpha ^2+2 x}-\alpha. 
\end{align*}
The upper bounds in Proposition~\ref{PROP:graph-max-S2} are achieved asymptotically by the constructions $\hat{S}\left(n,\alpha n, \beta_3 n \right)$ and $S\left(n, \beta_1 n \right)$, respectively.

\subsection{Tur\'an problem in multigraphs}\label{SUBSEC:Multigraphs}
For a positive integer $m$, an $m$-tuple $\vec{\mathcal{G}}= (G_1, \ldots, G_m)$ of graphs on a common vertex set $V$ will be referred to as an {\emph{$m$-multigraph}}. 
We use $V(\vec{\mathcal{G}})$ to denote the common vertex set of $\vec{\mathcal{G}}$. 
The size (i.e., the number of edges) of $\vec{\mathcal{G}}$ is defined as $|\vec{\mathcal{G}}| \coloneqq \sum_{i=1}^{m} |G_i|$.

For a vertex $v \in V(\vec{\mathcal{G}})$, the degree of $v$ in $\vec{\mathcal{G}}$ is given by $d_{\vec{\mathcal{G}}}(v) \coloneqq \sum_{i=1}^{m} d_{G_i}(v)$. 
Denote by $\delta(\vec{\mathcal{G}})$ the minimum degree of $\vec{\mathcal{G}}$. 

For every pair $\{u,v\} \subseteq V(\vec{\mathcal{G}})$, define the \emph{color set} of $\{u,v\}$ as 
\begin{align*}
    C_{\vec{\mathcal{G}}}(uv)
    \coloneqq 
    \left\{i \in [m] \colon uv \in G_i\right\}.
\end{align*}
The {\emph{multiplicity}} of $\{u,v\}$ is defined as  $\mu_{\vec{\mathcal{G}}}(uv) \coloneqq |C_{\vec{\mathcal{G}}}(uv)|$, which is the number of graphs among $G_1, \ldots,G_m$ in which $\{u,v\}$ appears as an edge. 
We will omit the subscript $\vec{\mathcal{G}}$ when it is clear from  context.

Denote by $\mathbb{K}_{4}$ the $3$-multigraph $(H_1, H_2, H_3)$ on vertex set $\{1,2,3,4\}$, where 
\begin{align*}
    H_1 = \big\{\{1,2\},~\{3,4\}\big\}, \quad 
    H_2 = \big\{\{1,3\},~\{2,4\}\big\}, \quad\text{and}\quad 
    H_3 = \big\{\{1,4\},~\{2,3\}\big\}. 
\end{align*}
Given an $m$-multigraph $\vec{\mathcal{G}}= (G_1, \ldots, G_m)$ on $V$, we say that $\vec{\mathcal{G}}$ contains a copy of $\mathbb{K}_4$ if there exists a triple $\{i,j,k\} \subseteq [m]$ and an injective map $\psi \colon \{1,2,3,4\} \to V$ such that 
\begin{itemize}
    \item $\big\{\{\psi(1), \psi(2)\},~\{\psi(3), \psi(4)\}\big\} \subseteq G_{i}$, 
    \item $\big\{\{\psi(1), \psi(3)\},~\{\psi(2), \psi(4)\}\big\} \subseteq G_{j}$, and  
    \item $\big\{\{\psi(1), \psi(4)\},~\{\psi(2), \psi(3)\}\big\} \subseteq G_{k}$. 
\end{itemize}
We say that $\vec{\mathcal{G}}$ is \emph{$\mathbb{K}_{4}$-free} if it does not contain any copy of $\mathbb{K}_4$. 
Denote by $\mathrm{ex}_{m}(n,\mathbb{K}_{4})$ the maximum size of a $\mathbb{K}_4$-free $m$-multigraph on $n$ vertices. 

The multigraph Tur\'{a}n problem plays a crucial role in the study of the Tur\'{a}n number of Fano plane and has been used in all the aforementioned works~\cite{DCF00,FS05,KS05,BR19}.  
The exact value of $\mathrm{ex}_{4}(n,\mathbb{K}_{4})$ was determined by F{\" u}redi--Simonovits in~\cite{FS05}, where they proved that for every $n \ge 4$, 
\begin{align*}
    \mathrm{ex}_{4}(n,\mathbb{K}_{4}) = 2\tbinom{n}{2} + 2 \left\lfloor n^2/4 \right\rfloor.
\end{align*}
In their study of $\mathrm{ex}(n,\mathbb{F})$, Bellmann--Reiher~{\cite[Proposition~3.4]{BR19}} investigated the value of $\mathrm{ex}_{5}(n,\mathbb{K}_{4})$ and proved that for every $n \ge 4$,  
\begin{align}\label{equ:Bellmann-Reiher-5-multigraph-upper-bound}
    \mathrm{ex}_{5}(n,\mathbb{K}_{4}) \le \tfrac{7n^2-n}{4}. 
\end{align}
They also stated the following exact bound, though without providing a proof (due to its tediousness):
\begin{align}\label{equ:Bellmann-Reiher-5-multigraph}
    \mathrm{ex}_{5}(n,\mathbb{K}_{4})
    = 
    \begin{cases}
        5\left\lfloor n^2/3 \right\rfloor, & \mathrm{if}\quad n \le 12, \\[0.5em]
        2\tbinom{n}{2} + 3\left\lfloor n^2/4 \right\rfloor, & \mathrm{if}\quad n \ge 13. 
    \end{cases}
\end{align}
Here, the first bound arises from the construction $\vec{\mathcal{G}}= (G_1, \ldots, G_5)$, where each $G_i$ is a $K_{4}$-free graph on $n$ vertices with exactly $\lfloor n^2/3 \rfloor$ edges.
The second bound is obtained from the following construction of $\vec{\mathcal{G}}= (G_1, \ldots, G_5)$:

Let $V_1 \cup V_2 = [n]$ be a balanced bipartition. 
Define 
\begin{align*}
    G_1 = G_2 = \tbinom{V_1}{2} \cup K[V_1, V_2], \quad 
    G_3 = G_4 = \tbinom{V_2}{2} \cup K[V_1, V_2], \quad\text{and}\quad 
    G_5 = K[V_1, V_2], 
\end{align*}
where $K[V_1, V_2]$ denotes the complete bipartite graph with parts $V_1$ and $V_2$.

Motivated by the second construction, we introduce the following definitions. 

Let $\vec{\mathcal{G}}=(G_1,\ldots, G_5)$ be a $5$-multigraph and $V_1 \cup V_2 = V(\vec{\mathcal{G}})$ be a partition of its vertex set. 
\begin{itemize}
    \item We say the partition $V_1 \cup V_2 = V(\vec{\mathcal{G}})$ is \emph{nice} if, up to a permutation of the indices, the following conditions hold: 
    \begin{enumerate}[(i)]
        \item $G_{1}[V_2] = G_{2}[V_2] = G_{3}[V_1] = G_{4}[V_1] = G_{5}[V_1] = \emptyset$, and 
        \item $\mu(uv) \le 2$ for every pair $\{u,v\} \subseteq V_2$. 
    \end{enumerate}
    \item We say the partition $V_1 \cup V_2 = V(\vec{\mathcal{G}})$ is \emph{good} if, up to a permutation of the indices, the following holds: 
    \begin{align*}
        G_{1}[V_2] = G_{2}[V_2] = G_{3}[V_1] = G_{4}[V_1] = G_{5}[V_1] = G_{5}[V_2]= \emptyset. 
    \end{align*}
\end{itemize}

Observe that every good partition is also a nice partition, but the converse does not necessarily hold. 

We establish the following Andr\'{a}sfai--Erd\H{o}s--S\'{o}s--type stability theorems for $\mathbb{K}_{4}$-free $5$-multigraphs, which refine the result of Bellmann--Reiher. 

\begin{theorem}\label{THM:AES-5-multigraph}
    Let $n \in \mathbb{N}$ and $\vec{\mathcal{G}}=(G_1,\ldots, G_5)$ be an $n$-vertex $\mathbb{K}_{4}$-free $5$-multigraph. 
    \begin{enumerate}[(i)]
        \item\label{THM:AES-5-multigraph-a} If $\delta(\vec{\mathcal{G}}) \ge 44n/13$, then $\vec{\mathcal{G}}$ admits a nice partition. 
        \item\label{THM:AES-5-multigraph-b} If $\delta(\vec{\mathcal{G}}) \ge 58n/17$, then $\vec{\mathcal{G}}$ admits a good partition. 
    \end{enumerate}
\end{theorem}

The proof of Theorem~\ref{THM:AES-5-multigraph} is presented in Section~\ref{SEC:Multigraphs}. 

\subsection{$3$-graphs}\label{SUBSEC:3-graphs}
Let $\mathcal{H}$ be a $3$-graph.  
For every vertex $v\in V(\mathcal{H})$, the \emph{neighborhood} $N_{\mathcal{H}}(v)$ of $v$ in $\mathcal{H}$ is 
\begin{align*}
    N_{\mathcal{H}}(v)
    \coloneqq \left\{u \in V(\mathcal{H})\setminus \{v\} \colon \text{there exists $E \in \mathcal{H}$ such that $\{u,v\} \subseteq E$}\right\}.
\end{align*}
It follows from the definition that $u \in N_{\mathcal{H}}(v)$ if and only if $u \neq v$ and $d_{\mathcal{H}}(uv) \neq 0$. 

Recall from Definition~\eqref{equ:def-p-norm-degree} that the $\ell_2$-norm degree of $v$ in $\mathcal{H}$ is 
\begin{align}\label{equ:def-2norm-degree}
    d_{2,\mathcal{H}}(v)
    = \norm{\mathcal{H}}_{2} - \norm{\mathcal{H} - v}_{2}
    & = \sum_{u \in N_{\mathcal{H}}(v)}d^{2}_{\mathcal{H}}(uv) + \sum_{e\in L_{\mathcal{H}}(v)}\left(2\cdot d_{\mathcal{H}}(e) - 1\right) \notag \\
    & = \sum_{u \in N_{\mathcal{H}}(v)}d^{2}_{\mathcal{H}}(uv) + 2 \sum_{e\in L_{\mathcal{H}}(v)} d_{\mathcal{H}}(e) - d_{\mathcal{H}}(v). 
\end{align}
where the straightforward calculations leading to the second equality can be found in~{\cite[Lemma~3.1]{CILLP24}}. 

Observe that $d_{\mathcal{H}}(uv)$ is the same as the degree of $u$ in the link graph $L_{\mathcal{H}}(v)$.
Therefore, $\sum_{u \in N_{\mathcal{H}}(v)}d^{2}_{\mathcal{H}}(uv)$ is simply the $\ell_{2}$-norm of $L_{\mathcal{H}}(v)$. 
Combining it with the trivial upper bound 
\begin{align*}
    \sum_{e\in L_{\mathcal{H}}(v)} d_{\mathcal{H}}(e) 
    \le \sum_{e\in L_{\mathcal{H}}(v)} (n-2) 
    \le |L_{\mathcal{H}}(v)| n, 
\end{align*}
we obtain the following crude but useful upper bound for $d_{2,\mathcal{H}}(v)$:
\begin{align}\label{equ:raw-upper-bound-2norm-degree}
    d_{2,\mathcal{H}}(v) 
    \le \norm{L_{\mathcal{H}}(v)}_{2} + 2 |L_{\mathcal{H}}(v)| n. 
\end{align}

Another useful expression of $d_{2,\mathcal{H}}(v)$ that is derived from~\eqref{equ:2-norm-in-S2-edge-3gp} is as follows:
\begin{align}\label{equ:defb-2norm-degree}
    d_{2,\mathcal{H}}(v)
    & = \norm{\mathcal{H}}_{2} - \norm{\mathcal{H} - v}_{2} \notag \\[0.3em]
    & = \big(2 \mathrm{N}(\mathbb{S}_{2}, \mathcal{H}) + 3|\mathcal{H}|\big) - \big(2 \mathrm{N}(\mathbb{S}_{2}, \mathcal{H}-v) + 3|\mathcal{H}-v|\big) 
    = 2 d_{\mathbb{S}_{2}, \mathcal{H}}(v) + 3 d_{\mathcal{H}}(v). 
\end{align}

Using this identity, one can easily derive the following results, which are special cases of~{\cite[Lemma~3.5]{CILLP24}} and~{\cite[Lemma~3.7]{CILLP24}}.   

\begin{lemma}\label{LEMMA:CILLP24-35-37}
    Let $\mathcal{H}$ be an $n$-vertex $3$-graph on vertex set $V$. 
    The following statements hold. 
    \begin{enumerate}[(i)]
        \item\label{LEMMA:CILLP24-35-37-a} We have $\sum\limits_{v\in V}d_{2,\mathcal{H}}(v) =  4\norm{\mathcal{H}}_{2} -3|\mathcal{H}|$.
        In particular, $\Big| \sum\limits_{v\in V}d_{2,\mathcal{H}}(v) - 4\norm{\mathcal{H}}_{2} \Big| \le n^3$.  
        \item\label{LEMMA:CILLP24-35-37-b} For every subgraph $\mathcal{H}' \subseteq \mathcal{H}$, it holds that $\norm{\mathcal{H}}_{2} - \norm{\mathcal{H}'}_{2} \le 6 n \cdot |\mathcal{H}\setminus\mathcal{H}'|$.
    \end{enumerate}
\end{lemma}
\begin{proof}[Proof of Lemma~\ref{LEMMA:CILLP24-35-37}]
    It follows from~\eqref{equ:defb-2norm-degree},~\eqref{equ:sum-Q-degree}, and~\eqref{equ:2-norm-in-S2-edge-3gp} that 
    \begin{align*}
        \sum_{v\in V}d_{2,\mathcal{H}}(v)
        = \sum_{v\in V} \big(2 d_{\mathbb{S}_{2}, \mathcal{H}}(v) + 3 d_{\mathcal{H}}(v)\big)
        = 8 \mathrm{N}(\mathbb{S}_{2}, \mathcal{H}) + 9 |\mathcal{H}|
        = 4 \norm{\mathcal{H}}_{2} - 3|\mathcal{H}|. 
    \end{align*}
    Consequently, $\left| \sum_{v\in V}d_{2,\mathcal{H}}(v) - 4\norm{\mathcal{H}}_{2} \right| = 3|\mathcal{H}| \le 3\tbinom{n}{3} \le n^3$, which proves~\eqref{LEMMA:CILLP24-35-37-a}. 

    Fix a subgraph $\mathcal{H}' \subseteq \mathcal{H}$. Since every edge in $\mathcal{H}$ is contained in at most $3(n-3)$ copies of $\mathbb{S}_{2}$ in $\mathcal{H}$, we have 
    \begin{align*}
        \mathrm{N}(\mathbb{S}_{2}, \mathcal{H}) - \mathrm{N}(\mathbb{S}_{2}, \mathcal{H}')
        \le 3(n-3) |\mathcal{H} \setminus \mathcal{H}'|. 
    \end{align*}
    It follows from~\eqref{equ:2-norm-in-S2-edge-3gp} that 
    \begin{align*}
        \norm{\mathcal{H}}_{2} - \norm{\mathcal{H}'}_{2}
        & = 2 \big(\mathrm{N}(\mathbb{S}_{2}, \mathcal{H}) - \mathrm{N}(\mathbb{S}_{2}, \mathcal{H}') \big) +3 \big(|\mathcal{H}|  - |\mathcal{H}'|\big) \\
        & \le 6 (n-3) |\mathcal{H} \setminus \mathcal{H}'| + 3 |\mathcal{H} \setminus \mathcal{H}'|
        \le 6n |\mathcal{H} \setminus \mathcal{H}'|,
    \end{align*}
    which proves~\eqref{LEMMA:CILLP24-35-37-b}. 
\end{proof}

The next lemma shows that, in the $\ell_2$-norm, every extremal $\mathbb{F}$-free $3$-graph must be near-regular in terms of its $\ell_2$-norm degree.
For convenience, define the \emph{average $\ell_2$-norm degree} of a $3$-graph $\mathcal{H}$ by 
\begin{align*}
    d_{\ell_{2}}(\mathcal{H})
    \coloneqq \tfrac{1}{v(\mathcal{H})} \sum_{v\in V(\mathcal{H})} d_{2,\mathcal{H}}(v). 
\end{align*}

\begin{lemma}\label{LEMMA:2norm-degree-regular}
    Suppose that $\mathcal{H}$ is an $\mathbb{F}$-free $n$-vertex $3$-graph with $\norm{\mathcal{H}}_{2} = \mathrm{ex}_{\ell_2}(n,\mathbb{F})$.
    Then for every pair of vertices $\{u,v\} \subseteq V(\mathcal{H})$, we have $\left|d_{2,\mathcal{H}}(u) - d_{2,\mathcal{H}}(v)\right| \le 60 n^2$. 
    %
    In particular, for every vertex $v \in V(\mathcal{H})$, 
    \begin{align}\label{equ:LEMMA-near-regular-extremal-graph}
        \left|d_{2,\mathcal{H}}(v) - d_{\ell_{2}}(\mathcal{H})\right| 
        \le 60 n^2.
    \end{align}
\end{lemma}
\begin{proof}[Proof of Lemma~\ref{LEMMA:2norm-degree-regular}]
    Let $\mathcal{H}$ be an $\mathbb{F}$-free $n$-vertex $3$-graph with $\norm{\mathcal{H}}_{2} = \mathrm{ex}_{\ell_2}(n,\mathbb{F})$.
    Suppose to the contrary that there exist two vertices $u, v \in V(\mathcal{H})$ with $d_{2,\mathcal{H}}(u) - d_{2,\mathcal{H}}(v) > 60 n^2$. 
    Let $x$ and $y$ denote the number of $\mathbb{S}_{2}$ in $\mathcal{H}$ containing $u$ and $v$, respectively, and let $z$ denote the number of $\mathbb{S}_{2}$ in $\mathcal{H}$ containing both $u$ and $v$.
    
    It is clear that $z \le 4! \tbinom{n-2}{2} \le 12 n^2$. By~\eqref{equ:defb-2norm-degree}, we have 
    \begin{align}\label{equ:x-y-lower-bound}
        x - y 
        & = \tfrac{d_{2,\mathcal{H}}(u) - 3d_{\mathcal{H}}(u)}{2} - \tfrac{d_{2,\mathcal{H}}(v) - 3d_{\mathcal{H}}(v)}{2} \notag \\
        & = \tfrac{d_{2,\mathcal{H}}(u) - d_{2,\mathcal{H}}(v)}{2} - \tfrac{3 \left(d_{\mathcal{H}}(u) - d_{\mathcal{H}}(v)\right)}{2}
        > \tfrac{60 n^2}{2} - \tfrac{3}{2} \tbinom{n}{2}
        \ge 29 n^2.  
    \end{align}
    Let $\mathcal{H}'$ denote the $3$-graph obtained from $\mathcal{H}$ by removing all edges containing both $u$ and $v$, and let 
    \begin{align*}
        \mathcal{G}
        \coloneqq \left(\mathcal{H}' \setminus \left\{E\in \mathcal{H}' \colon v\in E \right\}\right) \cup \left\{e \cup \{v\} \colon e\in L_{\mathcal{H}'}(u) \right\}. 
    \end{align*}
    Since $\partial \mathbb{F}$ is complete, the new $3$-graph $\mathcal{G}$ remains $\mathbb{F}$-free. 
    Moreover, 
    \begin{align}\label{equ:edge-change-G-H}
        |\mathcal{H}| - |\mathcal{G}|
        & = \big(d_{\mathcal{H}}(u) + d_{\mathcal{H}}(v) -  d_{\mathcal{H}}(uv) \big) - 2\big(d_{\mathcal{H}}(u) - d_{\mathcal{H}}(uv)\big) \notag \\
        & = d_{\mathcal{H}}(v) + d_{\mathcal{H}}(uv) - d_{\mathcal{H}}(u) 
        \le \tbinom{n-1}{2} + n-2
        \le n^2. 
    \end{align}
    On the other hand, by~\eqref{equ:x-y-lower-bound}, we have  
    \begin{align}\label{equ:S2-change-G-H}
        \mathrm{N}(\mathbb{S}_{2}, \mathcal{G}) - \mathrm{N}(\mathbb{S}_{2}, \mathcal{H})
        & = \mathrm{N}(\mathbb{S}_{2}, \mathcal{G}) - \mathrm{N}(\mathbb{S}_{2}, \mathcal{H}') - \left( \mathrm{N}(\mathbb{S}_{2}, \mathcal{H}) - \mathrm{N}(\mathbb{S}_{2}, \mathcal{H}') \right) \notag \\
        & \ge x-z-y -z
        =x-y-2z
        \ge 29 n^2 - 24 n^2 = 5n^2. 
    \end{align}
    Combining~\eqref{equ:edge-change-G-H},~\eqref{equ:S2-change-G-H}, and~\eqref{equ:2-norm-in-S2-edge-3gp}, we obtain 
    \begin{align*}
        \norm{\mathcal{G}}_{2} - \norm{\mathcal{H}}_{2}
        = 2 \big(\mathrm{N}(\mathbb{S}_{2}, \mathcal{G}) -\mathrm{N}(\mathbb{S}_{2}, \mathcal{H}) \big) - 3 \big(|\mathcal{H}| - |\mathcal{G}|\big) 
        \ge 10n^2 - 3n^2 
        > 0,
    \end{align*}
    which contradicts the maximality of $\mathcal{H}$. 
\end{proof}

Below, we list some basic properties of $\mathbb{F}$-free $3$-graphs. 
\begin{fact}\label{FACT:find-Fano}
    Let $\mathcal{H}$ be an $\mathbb{F}$-free $3$-graph. 
    The following statements hold.
    \begin{enumerate}[(i)]
        \item\label{FACT:Find-Fano-rainbow-K4} For every edge $uvw \in \mathcal{H}$, the $3$-multigraph $\big( L_{\mathcal{H}}(u),L_{\mathcal{H}}(v),L_{\mathcal{H}}(w) \big)$ is $\mathbb{K}_{4}$-free.
        \item\label{FACT:Find-Fano-K222} For every vertex $v \in V(\mathcal{H})$, if $u_1u_2, v_1v_2, w_1w_2$ are three pairwise disjoint edges in the link graph $L_{\mathcal{H}}(v)$, then at least one triple in $\left\{u_{i}v_{j}w_{k} \colon (i,j,k) \in [2]^3\right\}$ is not an edge of $\mathcal{H}$. 
    \end{enumerate}
\end{fact}

The following result by Balogh--Clemen--Lidick\'{y}~\cite{BCL22b} on the $\ell_2$-norm Tur\'{a}n problem for $K_{5}^{3}$ will also be crucial in the proof of Theorem~\ref{THM:AES-Fano-L2norm}.  
\begin{theorem}[{\cite[Theorem~1.6]{BCL22b}}]\label{THM:2norm-K53} 
    It holds that $\pi_{\ell_2}(K_5^3) = 5/16$.
    Moreover, for every $\varepsilon > 0$, there exist $\delta_{\ref{THM:2norm-K53}} \coloneqq \delta_{\ref{THM:2norm-K53}}(\varepsilon) > 0$ and $N_{\ref{THM:2norm-K53}} \coloneqq N_{\ref{THM:2norm-K53}}(\varepsilon)$ such that every $K_5^3$-free $3$-graph $\mathcal{H}$ on $n \ge N_{\ref{THM:2norm-K53}}$ vertices with $\norm{\mathcal{H}}_{2} \ge \left(5/16 - \delta_{\ref{THM:2norm-K53}} \right) n^4$ is bipartite after removing at most $\varepsilon n^3$ edges.
\end{theorem}

We will also use the following simple fact concerning the number of  $\mathbb{S}_{2}$ in a $3$-graph. 
\begin{fact}\label{FACT:number-of-S2}
    Let $\mathcal{H}$ be a $3$-graph on $n$ vertices. 
    Then the following statements hold:
    \begin{enumerate}[(i)]
        \item\label{FACT:number-of-S2-a} Every edge of $\mathcal{H}$ is contained in at most $\tbinom{3}{2} (n-3) \le 3n$ copies of $\mathbb{S}_{2}$ in $\mathcal{H}$.
        \item\label{FACT:number-of-S2-b} Every vertex of $\mathcal{H}$ is contained in at most $4! \tbinom{n-1}{3} \le 4n^3$ copies of $\mathbb{S}_{2}$ in $\mathcal{H}$.
        \item\label{FACT:number-of-S2-c} Every pair of vertices in $\mathcal{H}$ is contained in at most $4! \tbinom{n-2}{2} \le 12n^2$ copies of $\mathbb{S}_{2}$ in $\mathcal{H}$.
    \end{enumerate}
\end{fact}

\subsection{Generalized vertex-extendability and degree-stability}\label{SUBSEC:vtx-extend}
In this subsection, we present the framework developed in~\cite{CL24, CILLP24} for proving the Andr\'{a}sfai--Erd\H{o}s--S\'{o}s--type stability theorem for the $\ell_p$-norm Tur\'{a}n problem. 

A family $\mathfrak{H}$ of $r$-graphs is \emph{hereditary} if, for every $\mathcal{H} \in \mathfrak{H}$, all subgraphs of $\mathcal{H}$ also belong to $\mathfrak{H}$. 

Let $p \ge 1$ be a real number. 
Let $\mathcal{F}$ be a family of $r$-graphs and $\mathfrak{H}$ be a hereditary family of $\mathcal{F}$-free $r$-graphs. 
\begin{enumerate}[(i)]
    \item We say that $\mathcal{F}$ is \emph{$\ell_p$-edge-stable} with respect to $\mathfrak{H}$ if for every $\varepsilon > 0$, there exist $\delta > 0$ and $N_0$ such that every $\mathcal{F}$-free $r$-graph $\mathcal{H}$ on $n \ge N_0$ vertices with $\norm{\mathcal{H}}_{p} \ge (1-\delta) \cdot \mathrm{ex}_{\ell_p}(n,\mathcal{F})$ is contained in $\mathfrak{H}$ after removing at most $\varepsilon n^r$ edges. 
    \item We say that $\mathcal{F}$ is \emph{$\ell_p$-degree-stable} with respect to $\mathfrak{H}$ if there exist $\delta>0$ and $N_0$ such that every $\mathcal{F}$-free $r$-graph $\mathcal{H}$ on $n \ge N_0$ vertices with $\delta_{\ell_p}(\mathcal{H}) \ge (1-\delta) \cdot \mathrm{exdeg}_{p}(n,\mathcal{F})$ is contained in $\mathfrak{H}$, where 
    \begin{align*}
        \mathrm{exdeg}_{p}(n,\mathcal{F}) 
        \coloneqq \tfrac{(r-1+p) \cdot \mathrm{ex}_{\ell_p}(n,\mathcal{F})}{n} 
        = (r-1+p+o_{n}(1)) \cdot \pi_{\ell_p}(\mathcal{F}) \cdot n^{r-2+p}.
    \end{align*}
    \item We say that $\mathcal{F}$ is \emph{$\ell_p$-vertex-extendable} with respect to $\mathfrak{H}$ if there exist $\delta>0$ and $N_0$ such that the following holds for every $\mathcal{F}$-free $r$-graph $\mathcal{H}$ on $n \ge N_0$ vertices with $\delta_{\ell_p}(\mathcal{H}) \ge (1-\delta) \cdot \mathrm{exdeg}_{p}(n,\mathcal{F}) \colon$ 
    if $\mathcal{H}-v \in \mathfrak{H}$ for some $v \in V(\mathcal{H})$, then $\mathcal{H} \in \mathfrak{H}$. 
\end{enumerate}

The following theorem, which extends~{\cite[Theorem~1.1]{HLZ24}} on ordinary Tur\'{a}n problems, is a consequence of~{\cite[Theorem~1.7]{CL24}} (see also~{\cite[Theorem~4.10]{CL24}} and~{\cite[Theorem~3.10]{CILLP24}}). 
\begin{theorem}\label{THM:CL-Lp-general}
    Let $r \ge 2$ be an integer and $p \ge 1$ be a real number. Let $\mathcal{F}$ be a family of $r$-graphs satisfying $\pi(\mathcal{F}) > 0$ and let $\mathfrak{H}$ be a hereditary family of $\mathcal{F}$-free $r$-graphs. 
    Suppose that $\mathcal{F}$ is $\ell_p$-edge-stable and $\ell_p$-vertex-extendable with respect to $\mathfrak{H}$. 
    Then $\mathcal{F}$ is $\ell_p$-degree-stable with respect to $\mathfrak{H}$. 
\end{theorem}

\section{Tur\'{a}n density and edge-stability of Fano Plane in $\ell_2$-norm}\label{SEC:density-edge-stability}
In this section, we prove the following results. 

\begin{theorem}\label{THM:2-norm-density}
    We have $\pi_{\ell_2}(\mathbb{F}) = 5/16$. 
\end{theorem}

Let $\mathfrak{B}$ denote the family of all bipartite $3$-graphs. Observe that $\mathfrak{B}$ is hereditary. 
\begin{theorem}\label{THM:2-edge-stability}
    The $3$-graph $\mathbb{F}$ is $\ell_2$-edge-stable with respect to $\mathfrak{B}$. 
\end{theorem}

Let us very briefly explain the main idea behind the proof of Theorem~\ref{THM:2-norm-density} (the same idea also applies to the proof of Theorem~\ref{THM:2-edge-stability}).
Let $\mathcal{H}$ be an extremal $\mathbb{F}$-free $3$-graph on $n$ vertices that maximizes the $\ell_2$-norm. 
By Theorem~\ref{THM:2norm-K53}, we are done if $\mathcal{H}$ is $K_5^3$-free.
Thus, we may assume that $\mathcal{H}$ contains five vertices $v_1, \ldots, v_5$ that induce a copy of $K_5^3$ in $\mathcal{H}$. 
Without loss of generality, assume that $v_1$ has the smallest (vertex) degree among these five vertices, and denote this minimum degree by $d_{\min}$.

Using results (i.e.~\eqref{equ:Bellmann-Reiher-5-multigraph-upper-bound}) on the $5$-multigraph Tur\'{a}n problem introduced by Bellmann--Reiher~\cite{BR19}, we first show in Lemma~\ref{LEMMA:K53-deg-sum} that $d_{\min} \le (61/176 + o_{n}(1))n^2$, which is substantially smaller than the expected bound $\left( 3/8 - o_{n}(1) \right)n^2$ in the extremal construction $\mathbb{B}_{n}$. 

Next, by analyzing the $\ell_2$-norm degree, we derive a lower bound for $d_{\min}$, contradicting the upper bound above.
Since $\mathcal{H}$ is extremal, Lemma~\ref{LEMMA:2norm-degree-regular} implies that $\mathcal{H}$ is nearly $\ell_2$-norm degree regular.
In particular, we obtain a lower bound of $d_{2,\mathcal{H}}(v_1) \ge (5/4 - o_{n}(1))n^3$.
By the inequality~\eqref{equ:raw-upper-bound-2norm-degree}, this lower bound heuristically suggests a corresponding lower bound on $|L_{\mathcal{H}}(v_1)|$, i.e., on $d_{\mathcal{H}}(v_1)$.
Lemma~\ref{LEMMA:main-lemma-sum-important} confirms this intuition.
However, to obtain a lower bound exceeding $\left( \tfrac{61}{176} + o_{n}(1) \right)n^2$, the classical theorem of Ahlswede--Katona (Theorem~\ref{THM:max-S2-density}) is insufficient.
That is the reason we need the refined bound given in Proposition~\ref{PROP:graph-max-S2}.
To apply this refinement, we also require the stability version of the $5$-multigraph Tur\'{a}n problem, as established in Theorem~\ref{THM:AES-5-multigraph}.

\subsection{Technical lemmas}
In this subsection, we establish several technical lemmas needed for the proofs of Theorems~\ref{THM:2-norm-density} and~\ref{THM:2-edge-stability}. 

The following lemma provides a lower bound on the degree of a vertex $v$ in a $3$-graph in terms of its $\ell_2$-norm degree. 
\begin{lemma}\label{LEMMA:warmup-bound-degree}
    Let $\mathcal{H}$ be an $n$-vertex $3$-graph and $v \in V(\mathcal{H})$ be a vertex. 
    Let 
    \begin{align*}
        \alpha 
        \coloneqq \tfrac{d_{2,\mathcal{H}}(v)}{n^3}
        \quad\text{and}\quad 
        x
        \coloneqq \tfrac{d_{\mathcal{H}}(v)}{n^2}. 
    \end{align*}
    Then $x \ge f^{-1}(\alpha) - o_{n}(1)$,
    where $f \colon [0, 1/2] \to \mathbb{R}$ is the increasing function defined by 
    \begin{align*}
        f(y) \coloneqq \max\left\{(1-2y)^{3/2}+6y-1,~(2y)^{3/2}+2y\right\}.
    \end{align*}
\end{lemma}
\begin{proof}[Proof of Lemma~\ref{LEMMA:warmup-bound-degree}]
    It follows from~\eqref{equ:raw-upper-bound-2norm-degree} and Theorem~\ref{THM:max-S2-density} that 
    \begin{align*}
        \alpha 
        = \tfrac{d_{2,\mathcal{H}}(v)}{n^3}
        \le \tfrac{\norm{L_{\mathcal{H}}(v)}_{2} + 2 n |L_{\mathcal{H}}(v)|}{n^3} 
        \le \max\left\{(1-2x)^{3/2}+4x-1,~(2x)^{3/2}\right\} + 2x + o_{n}(1), 
    \end{align*}
    which implies that $x \ge f^{-1}(\alpha) - o_{n}(1)$. 
\end{proof}

The following lemma shows that a dense $\mathbb{K}_{4}$-free $5$-multigraph contains a large induced subgraph with large minimum degree. 
\begin{lemma}\label{LEMMA:large-mindeg-multigraph}
    Suppose that $\vec{\mathcal{G}}=(G_1,\ldots,G_5)$ is an $n$-vertex $\mathbb{K}_{4}$-free $5$-multigraph with $|\vec{\mathcal{G}}| \ge \beta\tbinom{n+1}{2}$, where $\beta \in [0, 7/2)$ is a real number. Then there exists a set $U \subseteq V(\vec{\mathcal{G}})$ such that   
    \begin{align*}
        |U|
        \ge \left(\tfrac{4|\vec{\mathcal{G}}|-2\beta n(n+1)}{7-2\beta}\right)^{1/2} 
    \end{align*}
    and the induced subgraph $\vec{\mathcal{G}}[U] = (G_1[U], \ldots, G_{5}[U])$ satisfies $\delta(\vec{\mathcal{G}}[U]) \ge \beta |U|$. 
\end{lemma}
\begin{proof}[Proof of Lemma~\ref{LEMMA:large-mindeg-multigraph}]
    Let $m \coloneqq \left\lfloor\left(\tfrac{4|\vec{\mathcal{G}}|-2\beta n(n+1)}{7-2\beta}\right)^{1/2}\right\rfloor$. 
    Suppose to the contrary that this lemma fails. 
    Then there exist a partition $S \cup U = V(\vec{\mathcal{G}})$ with $|U| = m$, and an ordering $(v_1, \ldots, v_{n-m})$ of vertices in $S$ (obtained by repeatedly deleting vertices of minimum degree), such that for every $i \in [n-m]$, the vertex $v_i$ has less than $\beta (n + 1 - i)$ neighbors in the set $\{v_{i+1}, \ldots, v_{n-m}\} \cup U$.
    Therefore, the number of edges in $\vec{\mathcal{G}}$ that contain at least one vertex from $S$ is less than 
    \begin{align*}
        \sum_{i=1}^{n-m} \beta (n+1-i) 
        = \tfrac{\beta}{2} \left(n^2 - m^2 + n-m\right)
        \le \beta \tbinom{n+1}{2} - \tfrac{\beta m^2}{2}. 
    \end{align*}
    On the other hand, it follows from the Bellmann--Reiher bound~\eqref{equ:Bellmann-Reiher-5-multigraph-upper-bound} that $|\vec{\mathcal{G}}[U]| \le 7m^2/4$. 
    Therefore, we have 
    \begin{align*}
        |\vec{\mathcal{G}}|
        < \beta \tbinom{n+1}{2} - \tfrac{\beta m^2}{2} + \tfrac{7 m^2}{4} 
        & = \beta \tbinom{n+1}{2} + \tfrac{(7 - 2\beta) m^2}{4} 
        \le \beta \tbinom{n+1}{2}+ \tfrac{(7 - 2\beta)}{4} \cdot \tfrac{4|\vec{\mathcal{G}}|-2\beta n(n+1)}{7-2\beta}
        = |\vec{\mathcal{G}}|,
    \end{align*}
    a contradiction. 
\end{proof}

Note that the extremal construction $\mathbb{B}_{n}$ has minimum degree $(3/8 -o_{n}(1))n^2$. 
The next lemma shows that if an $\mathbb{F}$-free $3$-graph contains a copy of $K_{5}^{3}$, then it must have a vertex whose degree is significantly smaller than $(3/8 - o_{n}(1))n^2$. 
\begin{lemma}\label{LEMMA:K53-deg-sum}
    Let $\mathcal{H}$ be an $\mathbb{F}$-free $3$-graph on $n$ vertices. 
    Suppose that $S \subseteq V(\mathcal{H})$ is a set of five vertices that induces a copy of $K_{5}^{3}$ in $\mathcal{H}$. 
    Then 
    \begin{align*}
        \sum_{v\in S} d_{\mathcal{H}}(v) + \tfrac{3}{17} d_{\mathrm{min}} 
        \le \tfrac{61 n(n+1)}{34},
    \end{align*}
    where $d_{\mathrm{min}}\coloneqq \min_{v\in S} d_{\mathcal{H}}(v)$. 
    Consequently,   
    \begin{align*}
        d_{\mathrm{min}} 
        \le \tfrac{61 n(n+1)}{176}.  
    \end{align*} 
\end{lemma}
\begin{proof}[Proof of Lemma~\ref{LEMMA:K53-deg-sum}]
    Suppose to the contrary that this lemma fails for some $5$-set $S=\{v_1,\ldots, v_5\} \subseteq V(\mathcal{H})$. 
     For each $i \in [5]$, let $G_i \coloneqq L_{\mathcal{H}}(v_i)$, and denote $5$-multigraph $\vec{\mathcal{G}} \coloneqq (G_1, \ldots, G_5)$. 
    Let $d_{\mathrm{min}} \coloneqq \min\{|G_1|, \ldots, |G_5|\}$. 
    By our assumption, and using the trivial upper bound $d_{\mathrm{min}} \le \tbinom{n}{2} \le \tbinom{n+1}{2}$, we have 
    \begin{align*}
        |\vec{\mathcal{G}}|
        = \sum_{v\in S} d_{\mathcal{H}}(v) 
        > \tfrac{61 n(n+1)}{34}-\tfrac{3}{17} d_{\mathrm{min}}  
        \ge \tfrac{61 n(n+1)}{34}-\tfrac{3}{17} \tbinom{n+1}{2}
        = \tfrac{58}{17}\tbinom{n+1}{2}. 
    \end{align*}
    Since $\mathcal{H}$ is $\mathbb{F}$-free, by Fact~\ref{FACT:find-Fano}~\eqref{FACT:Find-Fano-rainbow-K4}, $\vec{\mathcal{G}}$ is $\mathbb{K}_{4}$-free. 
    Applying Lemma~\ref{LEMMA:large-mindeg-multigraph} to $\vec{\mathcal{G}}$ with $\beta \coloneqq 58/17$, we obtain a subset $U \subseteq V(\mathcal{H})$ such that  
    \begin{align}\label{equ:LEMMA-U-lower-bound}
        |U|
        \ge \left(\tfrac{4|\vec{\mathcal{G}}|-2\beta n(n+1)}{7-2\beta}\right)^{1/2}
        & > \left(\tfrac{4\left(\tfrac{61 n(n+1)}{34}-\tfrac{3}{17} d_{\mathrm{min}}\right)-2 \cdot \tfrac{58}{17} n(n+1)}{7-2\cdot \tfrac{58}{17}}\right)^{1/2} 
        =  \left(2n(n+1) -4d_{\mathrm{min}} \right)^{\frac{1}{2}} 
    \end{align}
    and $\delta(\vec{\mathcal{G}}[U]) \ge \beta |U| = 58|U|/17$. 

    Applying Theorem~\ref{THM:AES-5-multigraph}~\eqref{THM:AES-5-multigraph-b} to the induced subgraph $\vec{\mathcal{G}}[U]$, we obtain a partition $X \cup Y = U$ such that, up to a permutation of the indices, $G_{5}[X] = G_{5}[Y] = \emptyset$. Consequently, $|G_{5}[U]| \le |X||Y| \le |U|^2/4$. 
    Combining it with~\eqref{equ:LEMMA-U-lower-bound}, we obtain  
    \begin{align*}
        |G_5|
        =|G_{5}\setminus G_{5}[U]| + |G_{5}[U]| 
        & \le \tbinom{n}{2} - \tbinom{|U|}{2} + \tfrac{|U|^{2}}{4} 
        \le \tfrac{n^2}{2} - \tfrac{|U|^{2}}{4} 
        \le \tfrac{n^2}{2} - \tfrac{2n(n+1)-4d_{\mathrm{min}}}{4}
        = d_{\mathrm{min}} - \tfrac{n}{2},
    \end{align*}
    contradicting the definition of $d_{\mathrm{min}}$. 
    Therefore, we have 
    \begin{align*}
        |G_1|+ \cdots + |G_5| 
        = |\vec{\mathcal{G}}| 
        \le \tfrac{61 n(n+1)}{34}-\tfrac{3}{17} d_{\mathrm{min}}, 
    \end{align*}
    and it follows that 
    \begin{align*}
        d_{\mathrm{min}} 
        \le \tfrac{61 n(n+1)}{34} \left(5+\tfrac{3}{17}\right)^{-1}
        = \tfrac{61 n(n+1)}{176}. 
    \end{align*}
    This completes the proof of Lemma~\ref{LEMMA:K53-deg-sum}. 
\end{proof}

The next lemma shows that if an $\mathbb{F}$-free $3$-graph contains a copy of $K_{5}^{3}$, then it must have a vertex whose $\ell_{2}$-norm degree is bounded away from the minimum $\ell_{2}$-norm degree of the extremal construction $\mathbb{B}_{n}$. 
\begin{lemma}\label{LEMMA:main-lemma-sum-important}
    There exist constants $\varepsilon_{\ref{LEMMA:main-lemma-sum-important}} > 0$ and $N_{\ref{LEMMA:main-lemma-sum-important}}$ such that the following holds for every $n \ge N_{\ref{LEMMA:main-lemma-sum-important}}$.
    Suppose that $\mathcal{H}$ is an $n$-vertex $3$-graph that contains a $5$-set $S = \{v_1, \ldots, v_{5}\} \subseteq V(\mathcal{H})$ which induces a copy of $K_{5}^{3}$ in $\mathcal{H}$, and suppose that $d_{2,\mathcal{H}}(v_i) \ge (5/4 - \varepsilon_{\ref{LEMMA:main-lemma-sum-important}})n^3$ for $i \in [5]$.
    Then $\mathbb{F} \subseteq \mathcal{H}$. 
\end{lemma}
\begin{proof}[Proof of Lemma~\ref{LEMMA:main-lemma-sum-important}]
    Let $\varepsilon_{\ref{LEMMA:main-lemma-sum-important}} > 0$ be sufficiently small, and $N_{\ref{LEMMA:main-lemma-sum-important}}$ be sufficiently large. 
    Let $\mathcal{H}$ be a $3$-graph on $n \ge N_{\ref{LEMMA:main-lemma-sum-important}}$ vertices satisfying the assumptions of the lemma. 
    Suppose that $S = \{v_1, \ldots, v_{5}\} \subseteq V(\mathcal{H})$ is a set of five vertices that induces a copy of $K_{5}^{3}$ in $\mathcal{H}$, and suppose that $d_{2,\mathcal{H}}(v_i) \ge (5/4 - \varepsilon_{\ref{LEMMA:main-lemma-sum-important}})n^3$ for $i \in [5]$.
    
    Let $G_i \coloneqq L_{\mathcal{H}}(v_i)$ for $i \in [5]$, and let $\vec{\mathcal{G}}=(G_1, \ldots, G_5)$. 
    Let $d_{\mathrm{min}} \coloneqq \min\{|G_1|, \ldots, |G_{5}|\}$. 
    Suppose to the contrary that $\mathcal{H}$ is $\mathbb{F}$-free. By Fact~\ref{FACT:find-Fano}~\eqref{FACT:Find-Fano-rainbow-K4}, the $5$-multigraph $\vec{\mathcal{G}}$ is $\mathbb{K}_{4}$-free. 

    For every $i \in [5]$, since $d_{2,\mathcal{H}}(v_i) \ge (5/4 - \varepsilon_{\ref{LEMMA:main-lemma-sum-important}})n^3$, it follows from Lemma~\ref{LEMMA:warmup-bound-degree} and the continuity of the function $f$ that 
    \begin{align}\label{equ:first-bound-dv-0342067}
        \tfrac{d_{\mathcal{H}}(v_i)}{n^2}
        \ge f^{-1}\left(\tfrac{5}{4} - \varepsilon_{\ref{LEMMA:main-lemma-sum-important}}\right) - o_{n}(1) 
        = 0.342067... - \varepsilon_1 - o_{n}(1)
        > \tfrac{235}{687},  
    \end{align}
    where $\varepsilon_1$ is some constant depending on $\varepsilon_{\ref{LEMMA:main-lemma-sum-important}}$, and $\varepsilon_1 \to 0$ as $\varepsilon_{\ref{LEMMA:main-lemma-sum-important}} \to 0$.

    \begin{claim}\label{CLAIM:indepedent-set-mindeg}
        For each $i \in [5]$, there exists an independent set $I_i$ in $G_i$ such that $d_{G_i}(v) \ge \alpha_1 n$ for every $v\in I_i$ and $|I_{i}| \ge \alpha_2 n$, where 
        \begin{align*}
            \alpha_1 
             \coloneqq \tfrac{4}{13 n}\left(\tfrac{260}{3}d_{\mathrm{min}}-\tfrac{88}{3}n(n+1)\right)^{1/2} \quad\text{and}\quad 
            \alpha_2 
             \coloneqq \tfrac{6}{13 n}\left(\tfrac{260}{3}d_{\mathrm{min}}-\tfrac{88}{3}n(n+1)\right)^{1/2}. 
        \end{align*}
    \end{claim}
    \begin{proof}[Proof of Claim~\ref{CLAIM:indepedent-set-mindeg}]
        It follows from~\eqref{equ:first-bound-dv-0342067} that $d_{\mathrm{min}} \ge \tfrac{235}{687}n^2$, and since $n$ is large, we have 
        \begin{align*}
            |\vec{\mathcal{G}}|
            = |G_1|+ \cdots + |G_5| 
            \ge 5 d_{\mathrm{min}} 
            \ge 5\cdot \tfrac{235}{687}n^2 
            > \tfrac{44}{13} \tbinom{n+1}{2}. 
        \end{align*}
        By Lemma~\ref{LEMMA:large-mindeg-multigraph}, there exists a subset $U \subseteq V(\vec{\mathcal{G}})$ such that 
        \begin{align}\label{equ:U-lower-bound}
            |U|
            \ge \left(\tfrac{4|\vec{\mathcal{G}}|- 2\cdot \tfrac{44}{13} n(n+1)}{7-2\cdot \tfrac{44}{13}}\right)^{1/2} 
            & \ge \left(\tfrac{20d_{\mathrm{min}}-\tfrac{88}{13} n(n+1)}{3/13} \right)^{1/2}  
            = \left( \tfrac{260}{3}d_{\mathrm{min}}-\tfrac{88}{3}n(n+1)\right)^{1/2}
        \end{align}
        and 
        \begin{align}\label{equ:minimum-degree-U}
            \delta(\vec{\mathcal{G}}[U])
            \ge 44 |U|/13. 
        \end{align}
        Applying Theorem~\ref{THM:AES-5-multigraph}~\eqref{THM:AES-5-multigraph-a} to the induced subgraph $\vec{\mathcal{G}}[U] = (G_1[U], \ldots, G_{5}[U])$, we obtain a partition $X \cup Y = U$ with $|Y| \le |X|$ such that, up to a permutation of the indices, 
        \begin{align}\label{equ:Lemma-nice-partition}
            G_{1}[Y] = G_{2}[Y] = G_{3}[X] = G_{4}[X] = G_{5}[X] = \emptyset, 
        \end{align}
        and every pair of vertices in $Y$ has multiplicity at most $2$ in $\vec{\mathcal{G}}[U]$. 
        It follows that 
        \begin{align*}
            d_{\vec{\mathcal{G}}[U]}(v)
            \le 
            \begin{cases}
                2|X| + 5|Y| = 2|U| + 3|Y|, &\quad \text{if}\quad v\in X, \\[0.3em]
                2|Y| + 5|X| = 2|U| + 3|X|, &\quad \text{if}\quad v\in Y.
            \end{cases}
        \end{align*}
        Since $|Y| \le |X|$, we have 
        \begin{align*}
            \delta(\vec{\mathcal{G}}[U]) 
            \le 2|X| + 5|Y| = 2|U| + 3|Y|. 
        \end{align*}
        Combining this with~\eqref{equ:minimum-degree-U}, we obtain 
        \begin{align}\label{equ:first-bound-Y-independent-set}
            |Y|
            \ge \tfrac{1}{3}\left(\delta(\vec{\mathcal{G}}[U]) - 2|U|\right)
            \ge \tfrac{1}{3}\left(\tfrac{44}{13}|U| - 2|U|\right)
            = \tfrac{6}{13}|U|. 
        \end{align}
        Let $I_1 = I_2 = Y$ and $I_3 = I_4 = I_5 = X$. 
        Note from~\eqref{equ:Lemma-nice-partition} that $I_i$ is independent in $G_i$ for every $i \in [5]$, and moreover, it follows from~\eqref{equ:first-bound-Y-independent-set} and~\eqref{equ:U-lower-bound} that 
        \begin{align*}
            \min\left\{|I_1|, \ldots, |I_{5}|\right\}
            = \min\{|X|,~|Y|\}
            = |Y|
            \ge \tfrac{6}{13}|U|
            \ge \alpha_{2} n. 
        \end{align*}
        So it remains to show that $d_{G_i}(v) \ge \alpha_1 n$ for every $v \in I_i$ and every $i \in [5]$. 
        
        Suppose that $i \in \{1,2\}$ and $v \in I_i = Y$.
        Then it follows from~\eqref{equ:minimum-degree-U},~\eqref{equ:first-bound-Y-independent-set}, and~\eqref{equ:U-lower-bound} that 
        \begin{align*}
            d_{G_{i}[U]}(v)
            & \ge \delta(\vec{\mathcal{G}}[U]) - \left( d_{G_{3-i}[U]}(v) + d_{G_{3}[U]}(v) + d_{G_{4}[U]}(v) + d_{G_{5}[U]}(v) \right)  \\[0.3em]
            & \ge \tfrac{44}{13}|U| - \left( |X| + 3|X| + 2|Y|\right)  \\[0.3em]
            & = \tfrac{44}{13}|U| - 4|U| + 2|Y| 
            \ge \tfrac{44}{13}|U| - 4|U| + 2 \cdot \tfrac{6}{13} |U|
            = \tfrac{4}{13}|U|
            \ge \alpha_1 n. 
        \end{align*}
        Here, the inequality $d_{G_{3}[U]}(v) + d_{G_{4}[U]}(v) + d_{G_{5}[U]}(v) \le 3|X| + 2|Y|$ follows from the fact that every pair of vertices in $Y$ has multiplicity at most $2$. 

        Suppose that $i \in \{3,4,5\}$ and $v\in I_i = X$. 
        By symmetry, we may assume that $i = 3$. 
        Then, similarly, it follows from~\eqref{equ:minimum-degree-U} and~\eqref{equ:U-lower-bound} that 
        \begin{align}\label{equ:second-bound-degree-of-X}
            d_{G_{3}[U]}(v)
            & \ge \delta(\vec{\mathcal{G}}[U]) - \left( d_{G_{1}[U]}(v) + d_{G_{2}[U]}(v) + d_{G_{4}[U]}(v) + d_{G_{5}[U]}(v) \right)  \notag \\[0.3em]
            & \ge \tfrac{44}{13}|U| - \left( |Y| + |Y| + |U| + |U|\right) 
            \ge  \tfrac{44}{13}|U| - 3|U| 
            = \tfrac{5}{13}|U|
            \ge \alpha_1 n. 
        \end{align}
        This completes the proof of Claim~\ref{CLAIM:indepedent-set-mindeg}. 
    \end{proof}

    \begin{claim}\label{CLAIM:d-min-first-lower-bound}
        We have $d_{\mathrm{min}} \ge \tfrac{61}{177} n^2$. 
    \end{claim}
    \begin{proof}[Proof of Claim~\ref{CLAIM:d-min-first-lower-bound}]
        Suppose to the contrary that $d_{\mathrm{min}} \le 61n^2/177$. 
        Let $i_0 \in [5]$ such that $|G_{i_0}| = d_{\mathrm{min}}$. 
        We aim to apply Proposition~\ref{PROP:graph-max-S2}~\eqref{PROP:graph-max-S2-07} to $G_{i_0}$. 
        By Claim~\ref{CLAIM:indepedent-set-mindeg}, it suffices to verify that $\alpha_1 \in [17/100,~23/100]$ and $\rho\coloneqq {|G_{i_0}|}/{n^2} \in [17/50,~7/20]$. 
        
        Since $d_{\mathrm{min}} \le 61n^2/177$ and $n$ is large, we have 
        \begin{align*}
            \alpha_1  
            & = \tfrac{4}{13 n} \left( \tfrac{260}{3}d_{\mathrm{min}} -\tfrac{88}{3}n(n+1) \right)^{1/2} \\
            & \le \tfrac{4}{13 n} \left( \tfrac{260}{3} \cdot \tfrac{61}{177}n^2-\tfrac{88}{3}n(n+1) \right)^{1/2} 
            = 0.225024...  + o_{n}(1) 
            \le \tfrac{23}{100}. 
        \end{align*}
        On the other hand, by~\eqref{equ:first-bound-dv-0342067}, we have $d_{\mathrm{min}} \ge {235 n^2}/{687}$, and hence, 
        \begin{align*}
            \alpha_{1} 
            & = \tfrac{4}{13 n} \left( \tfrac{260}{3}d_{\mathrm{min}} -\tfrac{88}{3}n(n+1) \right)^{1/2} \\
            & \ge \tfrac{4}{13 n} \left( \tfrac{260}{3} \cdot \tfrac{235}{687} n^2 -\tfrac{88}{3}n(n+1) \right)^{1/2}
            = 0.171997... -o_{n}(1)
            \ge \tfrac{17}{100}.
        \end{align*}
        This shows that $\alpha_1 \in [17/100,~23/100]$. 
        
        In addition, it follows from the assumption and~\eqref{equ:first-bound-dv-0342067} that   
        \begin{align*}
            \rho 
            \coloneqq \tfrac{|G_{i_0}|}{n^2}
            = \tfrac{d_{\mathrm{min}}}{n^2} 
            \in \left[\tfrac{235}{687},~\tfrac{61}{177}\right] 
            \subseteq \left[\tfrac{17}{50},~\tfrac{7}{20} \right]. 
        \end{align*}
        Therefore, by~\eqref{equ:raw-upper-bound-2norm-degree} and by applying Proposition~\ref{PROP:graph-max-S2}~\eqref{PROP:graph-max-S2-07} to $G_{i_0}$, we obtain  
        \begin{align*}
            \tfrac{d_{2,\mathcal{H}}(v_{i_0})}{n^3}
            & \le \tfrac{\norm{G_{i_{0}}}_{2}}{n^3} + 2 \cdot \tfrac{|G_{i_0}|}{n^2} \\
            & \le \max\left\{\alpha_{1}^3 + \left(2\rho-\alpha_{1}^2\right)\left(2\rho+\alpha_{1}^2\right)^{1/2},~4\rho-1+\left(1-2\rho\right)^{3/2}\right\} + 2\rho + o_{n}(1).  
        \end{align*} 
        Suppose that ${d_{2,\mathcal{H}}(v_{i_0})}/{n^3} \le 4\rho-1+\left(1-2\rho\right)^{3/2} + 2 \rho + o_{n}(1)$. 
        Then it follows from the assumption $d_{2, \mathcal{H}}(v_{i_0}) \ge (5/4 - \varepsilon_{\ref{LEMMA:main-lemma-sum-important}})n^3$ that 
        \begin{align*}
        \rho 
        \ge  0.346707... - \varepsilon_1 - o_{n}(1)
        > 61/176 + 10^{-4}, 
        \end{align*}
        where $\varepsilon_1$ is some constant depending on $\varepsilon_{\ref{LEMMA:main-lemma-sum-important}}$, and $\varepsilon_1 \to 0$ as $\varepsilon_{\ref{LEMMA:main-lemma-sum-important}} \to 0$. 
        This is a contradiction to Lemma~\ref{LEMMA:K53-deg-sum}. 
        So we may assume that 
        \begin{align}\label{equ:2norm-deg-v1-upper-bound}
            \tfrac{d_{2,\mathcal{H}}(v_{i_0})}{n^3} \le \alpha_{1}^3 + \left(2\rho-\alpha_{1}^2\right)\left(2\rho+\alpha_{1}^2\right)^{1/2} + 2 \rho + o_{n}(1). 
        \end{align}
        Recall that 
        \begin{align*}
            \alpha_1 
            = \tfrac{4}{13 n}\left(\tfrac{260}{3}d_{\mathrm{min}}-\tfrac{88}{3}n(n+1)\right)^{1/2} 
            = \tfrac{4}{13}\left(\tfrac{260}{3} \rho-\tfrac{88}{3}\right)^{1/2} + o_{n}(1).
        \end{align*}
        Combining it with~\eqref{equ:2norm-deg-v1-upper-bound} and the assumption $d_{2, \mathcal{H}}(v_{i_0}) \ge (5/4 - \varepsilon_{\ref{LEMMA:main-lemma-sum-important}})n^3$, we obtain 
        \begin{align*}
            \tfrac{5}{4} - \varepsilon_{\ref{LEMMA:main-lemma-sum-important}}
            \le \tfrac{64}{2197} \left(\tfrac{260 \rho }{3}-\tfrac{88}{3}\right)^{3/2}  -\tfrac{22 (143 \rho -64)}{6591} \sqrt{\tfrac{2 (2587 \rho -704)}{3}}  +2 \rho + o_{n}(1).
        \end{align*}
        Solving this inequality for $\rho$, we obtain 
        \begin{align*}
            \rho 
            \ge 0.344635... - \varepsilon_{1} - o_{n}(1)
            > \tfrac{61}{177}, 
        \end{align*}
        where $\varepsilon_1$ is some constant depending on $\varepsilon_{\ref{LEMMA:main-lemma-sum-important}}$, and $\varepsilon_1 \to 0$ as $\varepsilon_{\ref{LEMMA:main-lemma-sum-important}} \to 0$. 
        This completes the proof of Claim~\ref{CLAIM:d-min-first-lower-bound}. 
    \end{proof}

    \begin{claim}\label{CLAIM:d-min-second-lower-bound}
        We have $d_{\mathrm{min}} \ge \tfrac{253}{730} n^2$. 
    \end{claim}
    \begin{proof}[Proof of Claim~\ref{CLAIM:d-min-second-lower-bound}]
        We aim to apply Proposition~\ref{PROP:graph-max-S2}~\eqref{PROP:graph-max-S2-022-033} to $G_{i_0}$. 
        By Claim~\ref{CLAIM:indepedent-set-mindeg}, it suffices to verify that $\rho \coloneqq |G_{i_0}|/n^2 = d_{\mathrm{min}}/n^2 \in [17/50, 7/20]$, $\alpha_1 \ge 1/5$, and $\alpha_2 \in [1/3, 2/5]$. 

        It follows from Claim~\ref{CLAIM:d-min-first-lower-bound} that 
        \begin{align*}
            \tfrac{d_{\mathrm{min}}}{n^2}
            \ge \tfrac{61}{177}
            \ge \tfrac{17}{50}, 
        \end{align*}
        and it follows from Lemma~\ref{LEMMA:K53-deg-sum} that 
        \begin{align*}
            \tfrac{d_{\mathrm{min}}}{n^2}
            \le \tfrac{61}{176} \cdot  \tfrac{n(n+1)}{n^2}
            = \tfrac{61}{176} + o_{n}(1)
            \le \tfrac{7}{20}. 
        \end{align*}
        This proves that $\rho = d_{\mathrm{min}}/n^2 \in [17/50, 7/20]$. 
        
        Using the lower bound $d_{\mathrm{min}} \ge 61n^2/177$, we obtain 
        \begin{align*}
            \alpha_{1}  
            & = \tfrac{4}{13 n}\left(\tfrac{260}{3}d_{\mathrm{min}}-\tfrac{88}{3}n(n+1)\right)^{1/2} \\
            & \ge \tfrac{4}{13 n} \left( \tfrac{260}{3} \cdot \tfrac{61}{177} n^2 -\tfrac{88}{3}n(n+1) \right)^{1/2}
            = 0.225024... -o_{n}(1) 
            \ge \tfrac{1}{5}, 
        \end{align*}
        and 
        \begin{align*}
            \alpha_{2} 
            & = \tfrac{6}{13 n}\left(\tfrac{260}{3}d_{\mathrm{min}}-\tfrac{88}{3}n(n+1)\right)^{1/2} \\
            & \ge \tfrac{6}{13 n} \left( \tfrac{260}{3} \cdot \tfrac{61}{177} n^2 -\tfrac{88}{3}n(n+1) \right)^{1/2}
            = 0.337536... -o_{n}(1) 
            \ge \tfrac{1}{3}.
        \end{align*}
        Using the upper bound $d_{\mathrm{min}} \le 61n(n+1)/176$ (by Lemma~\ref{LEMMA:K53-deg-sum}), we obtain 
        \begin{align*}
            \alpha_{2}  
            & \le \tfrac{6}{13n} \left( \tfrac{260}{3} \cdot \tfrac{61}{176}n(n+1) -\tfrac{88}{3}n(n+1) \right)^{1/2}
            = 0.387402... + o_{n}(1)
            \le \tfrac{2}{5}.
        \end{align*}
        This proves that $\alpha_1 \ge 1/5$ and $\alpha_2 \in [1/3, 2/5]$. 
        
        Therefore, by~\eqref{equ:raw-upper-bound-2norm-degree} and by applying Proposition~\ref{PROP:graph-max-S2}~\eqref{PROP:graph-max-S2-022-033} to $G_{i_0}$, we obtain  
        \begin{align*}
            \tfrac{d_{2,\mathcal{H}}(v_{i_0})}{n^3}
            & \le \tfrac{\norm{G_{i_0}}_{2}}{n^3} + 2 \cdot \tfrac{|G_{i_0}|}{n^2} \\
            & \le \max\left\{\alpha_{2}^3 + \left(2\rho-\alpha_{2}^2\right)\left(2\rho+\alpha_{2}^2\right)^{1/2},~4\rho-1+\left(1-2\rho\right)^{3/2}\right\} + 2\rho + o_{n}(1).  
        \end{align*}
        Recall from the proof of Claim~\ref{CLAIM:d-min-first-lower-bound} that ${d_{2,\mathcal{H}}(v_{i_0})}/{n^3} \le 4\rho-1+\left(1-2\rho\right)^{3/2} + 2 \rho + o_{n}(1)$ is impossible. 
        Hence, we have 
        \begin{align}\label{equ:2norm-deg-v1-upper-bound-b}
            \tfrac{d_{2,\mathcal{H}}(v_{i_0})}{n^3}
            \le \alpha_{2}^3 + \left(2\rho-\alpha_{2}^2\right)\left(2\rho+\alpha_{2}^2\right)^{1/2} + 2\rho +o_{n}(1). 
        \end{align}
        Recall that 
        \begin{align*}
            \alpha_2 
            = \tfrac{6}{13 n}\left(\tfrac{260}{3}d_{\mathrm{min}}-\tfrac{88}{3}n(n+1)\right)^{1/2} 
            = \tfrac{6}{13}\left(\tfrac{260}{3} \rho-\tfrac{88}{3}\right)^{1/2} + o_{n}(1).
        \end{align*}
        Combining it with~\eqref{equ:2norm-deg-v1-upper-bound-b} and $d_{2, \mathcal{H}}(v_{i_0}) \ge (5/4 - \varepsilon_{\ref{LEMMA:main-lemma-sum-important}})n^3$ we obtain 
        \begin{align*}
            \tfrac{5}{4} - \varepsilon_{\ref{LEMMA:main-lemma-sum-important}} 
            \le \tfrac{192 \sqrt{3} (65 \rho -22)^{3/2}}{2197} +\tfrac{2 (528-1391 \rho ) \sqrt{3458 \rho -1056}}{2197} +2 \rho + o_{n}(1).
        \end{align*}
        Solving this inequality for $\rho$, we obtain 
        \begin{align*}
            \rho 
            \ge 0.346577... - \varepsilon_{1} - o_{n}(1)
            > \tfrac{253}{730}, 
        \end{align*}
        where $\varepsilon_1$ is some constant depending on $\varepsilon_{\ref{LEMMA:main-lemma-sum-important}}$, and $\varepsilon_1 \to 0$ as $\varepsilon_{\ref{LEMMA:main-lemma-sum-important}} \to 0$. 
        This completes the proof of Claim~\ref{CLAIM:d-min-second-lower-bound}. 
    \end{proof}

    \begin{claim}\label{CLAIM:G345-lower-bound}
        We have $\min\left\{|G_{3}|, |G_{4}|, |G_{5}|\right\} \ge \tfrac{321}{926} n^2$. 
    \end{claim}
    \begin{proof}[Proof of Claim~\ref{CLAIM:G345-lower-bound}]
        By symmetry, we may assume that $|G_3| = \min\left\{|G_{3}|, |G_{4}|, |G_{5}|\right\}$. 
        Let $\rho \coloneqq |G_3|/n^2$. 
        It follows from Claim~\ref{CLAIM:d-min-second-lower-bound} that $\rho \ge d_{\mathrm{min}}/n^2 \ge 253/730 > 17/50$. 
        Suppose to the contrary that $\rho < {321}/{926} < {7}/{20}$.
        Then $\rho \in [17/50,~7/20]$.

        It follows from~\eqref{equ:U-lower-bound} and Claim~\ref{CLAIM:d-min-second-lower-bound} that 
        \begin{align*}
            |X| 
            \ge \tfrac{|U|}{2}
            \ge \tfrac{1}{2}\left( \tfrac{260}{3}\cdot \tfrac{253}{730} n^2-\tfrac{88}{3}n(n+1)\right)^{1/2}
            = (0.419284... - o_{n}(1))n
            > \tfrac{2}{5}n, 
        \end{align*}
        Fix a subset $I \subseteq X$ of size $2n/5$. 
        Since $G_{3}[X] = \emptyset$, the set $I$ is independent in $G_3$. 

        For every vertex $v \in I \subseteq X$, it follows from~\eqref{equ:second-bound-degree-of-X} and Claim~\ref{CLAIM:d-min-second-lower-bound} that 
        \begin{align*}
            d_{G_{3}}(v) 
            \ge \tfrac{5}{13} \left( \tfrac{260}{3}\cdot \tfrac{253}{730} n^2-\tfrac{88}{3}n(n+1)\right)^{1/2}
            = (0.322526... - o_{n}(1)) n
            > \tfrac{1}{5}n.
        \end{align*}
        By~\eqref{equ:raw-upper-bound-2norm-degree} and by applying Proposition~\ref{PROP:graph-max-S2}~\eqref{PROP:graph-max-S2-022-033} to $G_3$, with the independent set $I$ there corresponding to $I$ here and the parameter $\alpha$ there corresponding to $2/5$, we obtain 
        \begin{align*}
            \tfrac{d_{2,\mathcal{H}}(v_3)}{n^3}
            & \le \tfrac{\norm{G_{3}}_{2}}{n^3} + 2 \cdot \tfrac{|G_3|}{n^2} \\
            & \le \max\left\{\alpha^3 + \left(2\rho-\alpha^2\right)\left(2\rho+\alpha^2\right)^{1/2},~4\rho-1+\left(1-2\rho\right)^{3/2}\right\} + 2\rho + o_{n}(1) \\
            & = \tfrac{8}{125} + \left(2 \rho -\tfrac{4}{25}\right) \left(2 \rho +\tfrac{4}{25}\right)^{1/2}+ 2 \rho + o_{n}(1). 
        \end{align*}
        Combining it with the assumption $d_{2,\mathcal{H}}(v_3) \ge (5/4 - \varepsilon_{\ref{LEMMA:main-lemma-sum-important}})n^3$ and solving for $\rho$, we obtain 
        \begin{align*}
            \rho 
            \ge 0.346665... -\varepsilon_1 - o_{n}(1)
            > \tfrac{321}{926}, 
        \end{align*}
        where $\varepsilon_1$ is some constant depending on $\varepsilon_{\ref{LEMMA:main-lemma-sum-important}}$, and $\varepsilon_1 \to 0$ as $\varepsilon_{\ref{LEMMA:main-lemma-sum-important}} \to 0$. 
        This contradicts our assumption, and completes the proof of Claim~\ref{CLAIM:G345-lower-bound}. 
    \end{proof}
 
    It follows from Claims~\ref{CLAIM:d-min-second-lower-bound} and~\ref{CLAIM:G345-lower-bound} that 
    \begin{align*}
        \sum_{i=1}^{5}|G_i|  + \tfrac{3}{17}d_{\mathrm{min}}
        \ge 2\cdot \tfrac{253}{730} n^2 + 3\cdot \tfrac{321}{926}n^2 + \tfrac{3}{17} \cdot \tfrac{253}{730} n^2
        = \tfrac{5154779}{2872915} n^2 
        > \tfrac{61}{34}n(n+1), 
    \end{align*}
    which is a contradiction to Lemma~\ref{LEMMA:K53-deg-sum}.
    This completes the proof of Lemma~\ref{LEMMA:main-lemma-sum-important}. 
\end{proof}

\subsection{Proofs of Theorems~\ref{THM:2-norm-density} and~\ref{THM:2-edge-stability}}
In this subsection, we present the proofs of Theorems~\ref{THM:2-norm-density} and~\ref{THM:2-edge-stability}.

\begin{proof}[Proof of Theorem~\ref{THM:2-norm-density}]
    Let $\beta \coloneqq \pi_{\ell_2}(\mathbb{F})$.
    It follows from the construction $\mathbb{B}_{n}$ that $\beta \ge 5/16$. 
    So it suffices to show that $\beta \le 5/16$. 
    Fix an arbitrary small constant $\varepsilon > 0$ with $\varepsilon < \varepsilon_{\ref{LEMMA:main-lemma-sum-important}}$, and let $n$ be sufficiently large such that, in particular, $n \ge \max\left\{N_{\ref{THM:2norm-K53}}, N_{\ref{LEMMA:main-lemma-sum-important}}\right\}$, where $\varepsilon_{\ref{LEMMA:main-lemma-sum-important}}$ and $N_{\ref{LEMMA:main-lemma-sum-important}}$ are the constants given by Lemma~\ref{LEMMA:main-lemma-sum-important}, and $N_{\ref{THM:2norm-K53}}$ is the constant given by Theorem~\ref{THM:2norm-K53}. 
    Let $\mathcal{H}$ be an $n$-vertex $\mathbb{F}$-free $3$-graph with $\norm{\mathcal{H}}_{2} = \mathrm{ex}_{\ell_2}(n,\mathbb{F})$, that is, of the maximum $\ell_{2}$-norm. 

    Since $n$ is sufficiently large, we have $\mathrm{ex}_{\ell_2}(n,\mathbb{F}) \ge \beta n^4 - \varepsilon n^4/8$ and $61 n^2 \le \varepsilon n^3/2$. 
    It follows from Lemma~\ref{LEMMA:2norm-degree-regular} (Inequality~\eqref{equ:LEMMA-near-regular-extremal-graph}) and Lemma~\ref{LEMMA:CILLP24-35-37}~\eqref{LEMMA:CILLP24-35-37-a} that for every $v\in V(\mathcal{H})$, 
    \begin{align}\label{equ:THM-min-ell2-degree-extremal}
        d_{2,\mathcal{H}}(v)
        & \ge \tfrac{1}{n} \sum_{v\in V(\mathcal{H})} d_{2,\mathcal{H}}(v) - 60 n^2 
        \ge \tfrac{4\norm{\mathcal{H}}_{2} - n^3}{n} - 60 n^2   
        = \tfrac{4 \cdot \mathrm{ex}_{\ell_2}(n,\mathbb{F})}{n} - 61 n^2 
        \ge 4\beta n^3 - \varepsilon n^3.
    \end{align}
    By Theorem~\ref{THM:2norm-K53}, we are done if $K_{5}^{3} \not\subseteq \mathcal{H}$. 
    So we may assume that there exists a $5$-set $S = \{v_1, \ldots, v_{5}\} \subseteq V(\mathcal{H})$ that induces a copy of $K_{5}^{3}$ in $\mathcal{H}$.
    Note from the inequality above that for every $i \in [5]$, we have
    \begin{align*}
        d_{2,\mathcal{H}}(v_{i}) 
        \ge 4\beta n^3 - \varepsilon n^3
        \ge \left(\tfrac{5}{4} - \varepsilon_{\ref{LEMMA:main-lemma-sum-important}}\right) n^3. 
    \end{align*}
    So it follows from Lemma~\ref{LEMMA:main-lemma-sum-important} that $\mathbb{F} \subseteq \mathcal{H}$, a contradiction. Therefore, $\pi_{\ell_2}(\mathbb{F}) \le 5/16$, and this completes the proof of Theorem~\ref{THM:2-norm-density}. 
\end{proof}

Next, we present the proof of Theorem~\ref{THM:2-edge-stability}.
It will be more convenient to work with $\mathrm{N}(\mathbb{S}_{2}, \mathcal{H})$ instead of the $\ell_2$-norm of $\mathcal{H}$, as the former is more transparent from a combinatorial perspective.
Therefore, define 
\begin{align*}
    \mathrm{ex}(n,\mathbb{S}_{2}, \mathbb{F})
     \coloneqq \max\big\{ \mathrm{N}(\mathbb{S}_2, \mathcal{H}) \colon \text{$v(\mathcal{H}) = n$ and $\mathcal{H}$ is $\mathbb{F}$-free} \big\},
     \quad
    \pi(\mathbb{S}_2, \mathbb{F})
     \coloneqq \lim_{n\to \infty} \tfrac{\mathrm{ex}(n,\mathbb{S}_{2}, \mathbb{F})}{n^{4}}. 
\end{align*}
It follows from~\eqref{equ:2-norm-in-S2-edge-3gp} that $\big|\norm{\mathcal{H}}_{2} - 2\mathrm{N}(\mathbb{S}_2, \mathcal{H}) \big| = 3|\mathcal{H}| \le n^3$ and hence, $\pi_{\ell_2}(\mathbb{F}) = 2 \cdot \pi(\mathbb{S}_2, \mathbb{F})$. 
%
%
\begin{proof}[Proof of Theorem~\ref{THM:2-edge-stability}]
    It suffices to show that for every $\varepsilon > 0$ there exist $\delta > 0$ and $N_0$ such that every $\mathbb{F}$-free $3$-graph $\mathcal{H}$ on $n \ge N_0$ vertices with $\mathrm{N}(\mathbb{S}_2, \mathcal{H}) \ge \mathrm{ex}(n,\mathbb{S}_{2}, \mathbb{F}) - \delta n^4$ is bipartite after removing at most $\varepsilon n^3$ edges. 
    
    Fix $\varepsilon > 0$. We may assume that $\varepsilon$ is sufficiently small. 
    Let $\delta > 0$ be sufficiently small and $N_0$ be sufficiently large such that, in particular, $32 \delta^{1/2} < \varepsilon_{\ref{LEMMA:main-lemma-sum-important}}$, $9\delta^{1/2} < \delta_{\ref{THM:2norm-K53}}(\varepsilon/2)$, and $N_0 > \max\left\{2 N_{\ref{THM:2norm-K53}}(\varepsilon/2), 2 N_{\ref{LEMMA:main-lemma-sum-important}}\right\}$, where $\varepsilon_{\ref{LEMMA:main-lemma-sum-important}}, N_{\ref{LEMMA:main-lemma-sum-important}}$ are the constants given by Lemma~\ref{LEMMA:main-lemma-sum-important}, and $\delta_{\ref{THM:2norm-K53}}(\varepsilon/2), N_{\ref{THM:2norm-K53}}(\varepsilon/2)$ are the constants given by Theorem~\ref{THM:2norm-K53}. 
    Let $\mathcal{H}$ be an $\mathbb{F}$-free $3$-graph on $n \ge N_0$ vertices with 
    \begin{align}\label{equ:S2-H-lower-bound}
        \mathrm{N}(\mathbb{S}_{2}, \mathcal{H})
        \ge \mathrm{ex}(n,\mathbb{S}_{2}, \mathbb{F}) - \delta n^4 
        > \left( \pi(\mathbb{S}_2, \mathbb{F}) - \delta \right) n^4 - \delta n^4 
        = \left( \pi(\mathbb{S}_2, \mathbb{F}) - 2\delta \right) n^4.
    \end{align}
    It follows from Theorem~\ref{THM:2-norm-density} that $\pi(\mathbb{S}_2, \mathbb{F}) = \pi_{\ell_2}(\mathbb{F})/2 = 5/32$. 
    
    Define 
    \begin{align*}
        Z_{\delta} 
        \coloneqq \left\{v \in V(\mathcal{H}) \colon d_{\mathbb{S}_{2}, \mathcal{H}}(v) \le \left(5/8 - 4 \delta^{1/2}\right)n^3 \right\}.
    \end{align*}

    \begin{claim}\label{CLAIM:Z-delta-upper-bound}
        We have $|Z_{\delta}| \le \delta^{1/2} n$.
    \end{claim}
    \begin{proof}[Proof of Claim~\ref{CLAIM:Z-delta-upper-bound}]
        Suppose to the contrary that $|Z_{\delta}| > \delta^{1/2} n$.
        Fix a subset $Z \subseteq Z_{\delta}$ of size $\delta^{1/2} n$, and let $U \coloneqq V(\mathcal{H}) \setminus Z$. 
        Since $n$ is large, we have 
        \begin{align*}
            \mathrm{N}(\mathbb{S}_{2}, \mathcal{H}[U]) 
            \le \left(\pi(\mathbb{S}_2, \mathbb{F}) + \delta \right)|U|^4
            & = \left(\tfrac{5}{32} + \delta \right) \left(1-\delta^{1/2} \right)^4 n^4  
            \le \left(\tfrac{5}{32} \left(1-\delta^{1/2} \right)^4 + \delta \right) n^4.  
        \end{align*}
        Since $\delta$ is small, we have 
        \begin{align*}
            \left(1-\delta^{1/2} \right)^4 
            = 1 -4 \delta^{1/2} +6 \delta -4 \delta ^{3/2}+\delta ^2 \le 1 -4 \delta^{1/2} +6 \delta, 
        \end{align*}
        and hence, 
        \begin{align*}
            \mathrm{N}(\mathbb{S}_{2}, \mathcal{H}[U]) 
            \le \left(\tfrac{5}{32} \left(1 -4 \delta^{1/2} +6 \delta \right) + \delta \right) n^4
            \le \left(\tfrac{5}{32} - \tfrac{5}{8} \delta^{1/2} + 2 \delta \right) n^4
        \end{align*}
        Combining this with the definition of $Z_{\delta}$, we obtain that 
        \begin{align*}
            \mathrm{N}(\mathbb{S}_{2}, \mathcal{H})
            & \le \mathrm{N}(\mathbb{S}_{2}, \mathcal{H}[U]) + |Z| \cdot \left(\tfrac{5}{8} - 4 \delta^{1/2}\right)n^3  \\
            &  \le \left(\tfrac{5}{32} - \tfrac{5}{8} \delta^{1/2} + 2 \delta \right) n^4 + \delta^{1/2}n \cdot \left(\tfrac{5}{8} - 4 \delta^{1/2}\right)n^3 
            = \left(\tfrac{5}{32} - 2 \delta \right) n^4,  
        \end{align*}
        which is a contradiction to~\eqref{equ:S2-H-lower-bound}.  
        Therefore, we have $|Z_{\delta}| \le \delta^{1/2} n$. 
    \end{proof}

    Let $\mathcal{G}$ denote the induced subgraph of $\mathcal{H}$ on $W \coloneqq V(\mathcal{H}) \setminus Z_{\delta}$. 
    Let $\mathcal{E}_{1}$ denote the set of edges in $\mathcal{H}$ that have nonempty intersection with $Z_{\delta}$. 
    By Claim~\ref{CLAIM:Z-delta-upper-bound}, $|Z_{\delta}| \le \delta^{1/2} n$, so we have 
    \begin{align*}
        |\mathcal{E}_{1}| 
        \le |Z_{\delta}| \cdot \tbinom{n-1}{2} 
        \le \delta^{1/2} n^3 \le \varepsilon n^3/2. 
    \end{align*}
    Note that every vertex in $\mathcal{G}$ is not contained in $Z_{\delta}$, so by the definition of $Z_{\delta}$, we have 
    \begin{align*}
        \delta_{\mathbb{S}_{2}}(\mathcal{G})
        & \ge \left(\tfrac{5}{8} - 4 \delta^{1/2}\right)n^3 - 4! \cdot |Z_{\delta}| \cdot \tbinom{n-2}{2}  \\
        & \ge \left(\tfrac{5}{8} - 4 \delta^{1/2}\right)n^3 - 24 \delta^{1/2} n \cdot \tbinom{n-2}{2}
        \ge \left(\tfrac{5}{8} - 16 \delta^{1/2}\right)n^3.
    \end{align*}
    It follows from~\eqref{equ:defb-2norm-degree} that 
    \begin{align*}
        \delta_{\ell_2}(\mathcal{G})
        \ge 2 \delta_{\mathbb{S}_{2}}(\mathcal{G})
        \ge \left(5/4 - 32 \delta^{1/2}\right)n^3
        \ge \left(5/4 - \varepsilon_{\ref{LEMMA:main-lemma-sum-important}}\right)n^3. 
    \end{align*}
    Therefore, by Lemma~\ref{LEMMA:main-lemma-sum-important} and $\mathbb{F}$-freeness of $\mathcal{G}$, we conclude that $\mathcal{G}$ is $K_{5}^{3}$-free.
    
    Since $\mathcal{G}$ is $K_{5}^{3}$-free, and by Lemma~\ref{LEMMA:CILLP24-35-37}~\eqref{LEMMA:CILLP24-35-37-a} 
    \begin{align*}
        \norm{\mathcal{G}}_{2} 
        & = \tfrac{1}{4} \left( \sum_{v\in W} d_{2,\mathcal{G}}(v) + 3|\mathcal{G}| \right)
        \ge \tfrac{1}{4} \cdot |W| \cdot \delta_{\ell_2}(\mathcal{G}) \\
        & \ge \tfrac{1}{4} \left(n - \delta^{1/2} n\right)\cdot \left(\tfrac{5}{4} - 32 \delta^{1/2}\right)n^3
        \ge \left(\tfrac{5}{16}-9\delta^{1/2} \right)n^4
        \ge \left(\tfrac{5}{16}-\delta_{\ref{THM:2norm-K53}}(\varepsilon/2) \right)n^4. 
    \end{align*}
It follows from Theorem~\ref{THM:2norm-K53} that $\mathcal{G}$ is bipartite after removing a set $\mathcal{E}_{2}$ of at most $\varepsilon n^3/2$ edges. 
    Therefore, $\mathcal{H}$ is bipartite after removing a set of $|\mathcal{E}_{1}| + |\mathcal{E}_{2}| \le \varepsilon n^3$ edges. 
    This completes the proof of Theorem~\ref{THM:2-edge-stability}. 
\end{proof}

\section{Proof of Theorem~\ref{THM:AES-Fano-L2norm}}\label{SEC:exact-degree-stability}
In this section, we present the proof of Theorem~\ref{THM:AES-Fano-L2norm}.
We begin by establishing the $\ell_{2}$-degree-stability of $\mathbb{F}$ with respect to the family $\mathfrak{B}$.
Recall from Theorem~\ref{THM:2-edge-stability} that $\mathbb{F}$ has already been shown to be $\ell_{2}$-edge-stable with respect to $\mathfrak{B}$.
Therefore, by Theorem~\ref{THM:CL-Lp-general}, it remains to prove the following result.

\begin{theorem}\label{THM:Fano-vertex-extendable}
    The $3$-graph $\mathbb{F}$ is $\ell_{2}$-vertex-extendable with respect to $\mathfrak{B}$. 
\end{theorem}
\begin{proof}[Proof of Theorem~\ref{THM:Fano-vertex-extendable}]
    Fix a sufficiently small constant $\xi > 0$ and let $n$ be a sufficiently large integer. 
    Let $\mathcal{H}$ be an $(n+1)$-vertex $\mathbb{F}$-free $3$-graph with 
    \begin{align}\label{equ:lower-bound-delta_2H}
        \delta_{\ell_{2}}(\mathcal{H}) 
        \ge \left(4 \pi_{\ell_{2}}(\mathbb{F}) - \xi \right) n^3
        = \left(\tfrac{5}{4} - \xi \right)n^3.
    \end{align}
    Suppose that $v_{\ast} \in V(\mathcal{H})$ is a vertex such that $\mathcal{H} - v_{\ast} \in \mathfrak{B}$, that is, the $3$-graph $\mathcal{G} \coloneqq \mathcal{H} - v_{\ast}$ is bipartite. 
    Let $V_{1} \cup V_2 = V(\mathcal{G})$ be a bipartition such that $\mathcal{G}[V_1] = \mathcal{G}[V_2] = \emptyset$. 

    \begin{claim}\label{CLAIM:vtx-ext-l2-min-deg-G}
        We have $\delta_{\ell_{2}}(\mathcal{G}) \ge \left( 5/4 - 2 \xi \right)n^3$. 
    \end{claim}
    \begin{proof}[Proof of Claim~\ref{CLAIM:vtx-ext-l2-min-deg-G}]
        By~\eqref{equ:defb-2norm-degree}, it suffices to show that $\delta_{\mathbb{S}_{2}}(\mathcal{G}) \ge \left( 5/8 - \xi \right)n^3$. 
        Fix a vertex $v \in V(\mathcal{G})$ with $d_{\mathbb{S}_{2},\mathcal{G}}(v) = \delta_{\mathbb{S}_{2}}(\mathcal{G})$. 
        Combining~\eqref{equ:lower-bound-delta_2H} and~\eqref{equ:defb-2norm-degree}, we see that 
        \begin{align*}
            d_{\mathbb{S}_{2},\mathcal{H}}(v)
            \ge \tfrac{d_{2,\mathcal{H}}(v) - 3 d_{\mathcal{H}}(v)}{2} 
            \ge \tfrac{1}{2} \left(\left(\tfrac{5}{4} - \xi \right)n^3 - 3\tbinom{n}{2}\right)
            \ge \left(\tfrac{5}{8} - \tfrac{\xi}{2} \right)n^3 - n^2. 
        \end{align*}
        Since there are at most $4!\tbinom{n-1}{2} \le 12n^2$ copies of $\mathbb{S}_{2}$ in $\mathcal{H}$ containing both $v$ and $v_{\ast}$, we have 
        \begin{align*}
            d_{\mathbb{S}_{2},\mathcal{G}}(v)
            \ge d_{\mathbb{S}_{2},\mathcal{H}}(v) - 12n^2 
            \ge \left(\tfrac{5}{8} - \tfrac{\xi}{2} \right)n^3 - n^2 -12 n^2
            \ge \left(\tfrac{5}{8} - \xi \right)n^3,
        \end{align*}
        which proves Claim~\ref{CLAIM:vtx-ext-l2-min-deg-G}. 
    \end{proof}

    \begin{claim}\label{CLAIM:size-of-Vi}
        We have $\left| |V_i| - \tfrac{n}{2} \right| \le 2 \xi n$ for $i \in \{1,2\}$. 
    \end{claim}
    \begin{proof}[Proof of Claim~\ref{CLAIM:size-of-Vi}]
        Let  $\alpha \coloneqq |V_1|/n$. 
        For every $v \in V_2$, it follows from Claim~\ref{CLAIM:vtx-ext-l2-min-deg-G} and~\eqref{equ:def-2norm-degree} that 
        \begin{align*}
            \left(\tfrac{5}{4} -2\xi\right)n^3 
            \le d_{2,\mathcal{G}}(v)
            & = \sum_{u\in V_1 \cup V_2} d^{2}_{\mathcal{G}}(u,v) + 2\sum_{e\in L_{\mathcal{G}}(v)} d_{\mathcal{G}}(e) - d_{\mathcal{G}}(v) \\ 
            &\le |V_1|n^2 + |V_2||V_1|^2 +  2 \left(\tbinom{|V_1|}{2} |V_2| + |V_1| |V_2| n \right) \\
            & \le |V_1|n^2 + |V_2||V_1|^2 + |V_1|^2 |V_2| + 2 |V_1| |V_2| n
            =\left(3\alpha - 2\alpha^3\right) n^3. 
        \end{align*}
        It follows that $\alpha \ge 1/2 -2 \xi$, which means that $|V_1| \ge (1/2 -2 \xi) n$. 
        By symmetry, we also have $|V_2| \ge (1/2 -2 \xi) n$. This completes the proof of Claim~\ref{CLAIM:size-of-Vi}. 
    \end{proof}

    For convenience, for every $v \in V$, let $L_{v} \coloneqq L_{\mathcal{G}}(v)$. 
    
    \begin{claim}\label{CLAIM:near-full-degree}
        For every $i \in \{1,2\}$ and every $v \in V_i$, we have 
        \begin{align*}
            \left|L_v[V_{3-i}]\right| 
            \ge \tbinom{|V_{3-i}|}{2} - 15 \xi n^2 
            \quad\text{and}\quad  
            \left|L_v[V_1, V_2]\right| 
            \ge |V_1||V_2| - 5\xi n^2. 
        \end{align*}
    \end{claim}
    \begin{proof}[Proof of Claim~\ref{CLAIM:near-full-degree}]
        Fix $i \in \{1,2\}$ and  $v \in V_i$. 
        By symmetry, we may assume that $i = 2$. 
        Let $\alpha \coloneqq |V_1|/n$, $x \coloneqq \tbinom{|V_1|}{2} - |L_{v}[V_1]|$, and $y \coloneqq |V_1||V_2| - |L_{v}[V_1, V_2]|$. 
        Similar to the proof of Claim~\ref{CLAIM:size-of-Vi}, it follows from Claim~\ref{CLAIM:vtx-ext-l2-min-deg-G} and~\eqref{equ:def-2norm-degree} that 
        \begin{align*}
        \left(\tfrac{5}{4} -2\xi\right)n^3 
        \le d_{2, \mathcal{G}}(v) 
        & = \sum_{u\in V_1 \cup V_2} d^{2}_{\mathcal{G}}(u,v) + 2\sum_{e\in L_{\mathcal{G}}(v)} d_{\mathcal{G}}(e) - d_{\mathcal{G}}(v) \\
        &\le |V_1|n^2 + |V_2||V_1|^2 + 2 \left(\left|L_v[V_{1}]\right| \cdot \left|V_2\right| + \left|L_v[V_1, V_2]\right| \cdot n \right) \\
        & = |V_1|n^2 + |V_2||V_1|^2 +  2 \left(\tbinom{|V_1|}{2} |V_2| + |V_1| |V_2| n \right) - 2 \left( x |V_2| + y n \right) \\
        & \le \left(3\alpha - 2\alpha^3\right) n^3 - \left( x (1-\alpha)n + y n \right).
        \end{align*}
        Since, by Claim~\ref{CLAIM:size-of-Vi}, $\alpha \le 1/2 + 2\xi \le 2/3$, we have 
        \begin{align*}
            \left(\tfrac{5}{4} -2\xi\right)n^3 
            & \le \left(3\alpha - 2\alpha^3\right) n^3 - \left( x (1-\alpha)n + y n \right) 
            \le \left(\tfrac{5}{4} + 3\xi\right)n^3 - \left( \tfrac{xn}{3} + y n \right). 
        \end{align*}
        It follows that $x\le 15\xi n^2$ and $y \le 5 \xi n^2$,  
        which proves Claim~\ref{CLAIM:near-full-degree}. 
    \end{proof}

    Let $L_{v_{\ast}} \coloneqq L_{\mathcal{H}}(v_{\ast})$. 
    Recall from the assumption~\eqref{equ:lower-bound-delta_2H} that $d_{2,\mathcal{H}}(v_{\ast}) \ge (5/4 - \xi) n^3$. 
    Similar to~\eqref{equ:first-bound-dv-0342067}, it follows from Lemma~\ref{LEMMA:warmup-bound-degree} that 
    \begin{align}\label{equ:size-Lv-ast}
        \tfrac{|L_{v_{\ast}}|}{n^2}
        \ge f^{-1}\left(\tfrac{5}{4} - \xi\right) -o_{n}(1)
        = 0.342067... - \xi_1 - o_{n}(1)
        \ge \tfrac{17}{50},
    \end{align}
    where $\xi_1$ is some constant depending on $\xi$, and $\xi_1 \to 0$ as $\xi \to 0$.
        
    \begin{claim}\label{CLAIM:H-is-bipartite}
        We have $\mathcal{H} \in \mathfrak{B}$, that is, $\mathcal{H}$ is bipartite. 
    \end{claim}
    \begin{proof}[Proof of Claim~\ref{CLAIM:H-is-bipartite}]
        To show that $\mathcal{H}$ is bipartite, it suffices to show that $L_{v_{\ast}}[V_i] = \emptyset$ for some $i \in \{1,2\}$. 
        Suppose to the contrary that $L_{v_{\ast}}[V_1] \neq \emptyset$ and $L_{v_{\ast}}[V_2] \neq \emptyset$. 
        By symmetry, we may assume that $|L_{v_{\ast}}[V_1]| \ge |L_{v_{\ast}}[V_2]|$.
        
        Fix a link $v_1 v_2 \in L_{v_{\ast}}[V_2]$. 
        Choose uniformly at random three distinct vertices $u_1, u_2, w_1$ from $V_1$, and independently, choose uniformly at random a vertex $w_2$ from $V_2 \setminus \{v_{1}, v_2\}$. 
        Observe that we obtain a copy of $\mathbb{F}$ in $\mathcal{H}$ if the following events $A_1, A_2, A_3$ all occur:
        \begin{align*}
            A_1 \colon \{u_1 u_2, w_1 w_2\}  \subseteq L_{v_{\ast}}, \quad 
            A_2 \colon \{u_1 w_2, u_2 w_1\} \subseteq L_{v_1}, \quad 
            A_3 \colon \{u_1 w_1, u_2 w_2\} \subseteq L_{v_2}. 
        \end{align*}
        %
        Let us compute the probabilities of these events occurring.
        First, by Claim \ref{CLAIM:size-of-Vi}, 
        \begin{align*}
            \mathbb{P}\left[u_1u_2 \in L_{v_{\ast}} \right]
            & = \tfrac{|L_{v_{\ast}}[V_1]|}{\tbinom{|V_1|}{2}}
            \ge \tfrac{2 |L_{v_{\ast}}[V_1]|}{|V_1|^2} 
            \ge \tfrac{|L_{v_{\ast}}[V_1]| + |L_{v_{\ast}}[V_2]|}{\left(1/2+2\xi\right)^2 n^2},  \quad\text{and}\quad \\[0.5em]
            \mathbb{P}\left[w_1w_2 \in L_{v_{\ast}} \right]
            & = \tfrac{|L_{v_{\ast}}[V_1, V_2] - \{u_1, u_2, v_1, v_2\}|}{\left(|V_1|-2\right)\left(|V_2| - 2\right)}
            \ge \tfrac{|L_{v_{\ast}}[V_1, V_2]| - 2n}{\left(1/2+2\xi\right)^2 n^2}.
        \end{align*}
        It follows that 
        \begin{align*}
            \mathbb{P}\left[u_1u_2 \in L_{v_{\ast}} \right] + \mathbb{P}\left[w_1w_2 \in L_{v_{\ast}} \right] 
            & \ge \tfrac{|L_{v_{\ast}}[V_1]| + |L_{v_{\ast}}[V_2]| + |L_{v_{\ast}}[V_1, V_2]| - 2n}{\left(1/2+2\xi\right)^2 n^2} \\
            & \ge \tfrac{17n^2/50 - 2n}{\left(1/2+2\xi\right)^2 n^2}
            \ge \tfrac{34}{25} - 20\xi. 
        \end{align*}

\makeatletter
\newsavebox\myboxA
\newsavebox\myboxB
\newlength\mylenA
\newcommand*\xoverline[2][0.75]{%
    \sbox{\myboxA}{$\m@th#2$}%
    \setbox\myboxB\null
    \ht\myboxB=\ht\myboxA%
    \dp\myboxB=\dp\myboxA%
    \wd\myboxB=#1\wd\myboxA
    \sbox\myboxB{$\m@th\overline{\copy\myboxB}$}
    \setlength\mylenA{\the\wd\myboxA}
    \addtolength\mylenA{-\the\wd\myboxB}%
    \ifdim\wd\myboxB<\wd\myboxA%
       \rlap{\hskip 0.5\mylenA\usebox\myboxB}{\usebox\myboxA}%
    \else
        \hskip -0.5\mylenA\rlap{\usebox\myboxA}{\hskip 0.5\mylenA\usebox\myboxB}%
    \fi}
\makeatother

        Therefore, the probability that $A_1$ does not occur is 
        \begin{align}\label{EQ:pro-u_1u_2-w1w2-not-Gv0}
            \mathbb{P}\left[ \xoverline{A_{1}} \right]
            & \le \mathbb{P}\left[u_1u_2 \not\in L_{v_{\ast}} \right]  + \mathbb{P}\left[w_1w_2 \not\in L_{v_{\ast}} \right]   \notag \\[0.5em]
            & = 2- \left(\mathbb{P}\left[u_1u_2 \in L_{v_{\ast}} \right] + \mathbb{P}\left[w_1w_2 \in L_{v_{\ast}} \right] \right) 
            \le 2- \tfrac{34}{25} + 20\xi
            \le \tfrac{2}{3}. 
        \end{align} 
        Next, we consider $A_2$.
        Applying Claim~\ref{CLAIM:near-full-degree} to $v_1$, we obtain 
        \begin{align}\label{EQ:Pij-not-Gvj}
            \mathbb{P}\left[ \xoverline{A_{2}} \right]
            & \le \mathbb{P}\left[u_1w_2 \not\in L_{v_1}[V_1, V_2]\right] + \mathbb{P}\left[u_2w_1 \not\in L_{v_1}[V_1]\right] \notag \\
            & \le \tfrac{5\xi n^2}{\left(|V_1|-2\right)\left(|V_2|-2\right)} + \tfrac{15\xi n^2}{\tbinom{|V_1|-2}{2}} 
            \le \tfrac{5\xi n^2}{\left(\left(1/2 - 2\xi\right)n -2 \right)^2} + \tfrac{30 \xi n^2}{\left(\left(1/2 - 2\xi\right)n -3 \right)^2}
            \le \tfrac{1}{100}. 
        \end{align}

        Similarly, we have 
        \begin{align}\label{EQ:Qij-not-Gvj}
            \mathbb{P}\left[ \xoverline{A_{3}} \right]
            \le \tfrac{1}{100}.
        \end{align}
        Combining~\eqref{EQ:pro-u_1u_2-w1w2-not-Gv0},~\eqref{EQ:Pij-not-Gvj}, and~\eqref{EQ:Qij-not-Gvj}, we see that with probability at least $1 - \tfrac{2}{3} - \tfrac{1}{100} - \tfrac{1}{100} > 0$, all events $A_1, A_2, A_3$ occur simultaneously. 
        This implies that $\mathbb{F} \subseteq \mathcal{H}$, a contradiction. 
    \end{proof}

    It follows from Claim~\ref{CLAIM:H-is-bipartite} that $\mathcal{H} \in \mathfrak{B}$, which proves that $\mathbb{F}$ is $\ell_2$-vertex-extendable with respect to $\mathfrak{B}$. This completes the proof of Theorem~\ref{THM:Fano-vertex-extendable}.
\end{proof}

To complete the proof of Theorem~\ref{THM:AES-Fano-L2norm}, it remains to show that $\mathbb{B}_{n}$ is the unique extremal construction for $\mathbb{F}$ in the $\ell_{2}$-norm.

We begin by showing, in the following lemma, that among all $n$-vertex bipartite $3$-graphs, $\mathbb{B}_n$ uniquely attains the maximum $\ell_2$-norm. 
\begin{lemma}\label{LEMMA:Bn-max-2norm}
    Let $n \ge 3$ be an integer. 
    Suppose that $\mathcal{H}$ is an $n$-vertex bipartite $3$-graph. 
    Then $\norm{\mathcal{H}}_{2} \le \norm{\mathbb{B}_{n}}_{2}$, and equality holds if and only if $\mathcal{H} \cong \mathbb{B}_{n}$. 
\end{lemma}
\begin{proof}[Proof of Lemma~\ref{LEMMA:Bn-max-2norm}]
    It follows from the definition of $\mathbb{B}_{n}$ that  
    \begin{align*}
        \norm{\mathbb{B}_n}_{2} 
        & = \left\lceil \tfrac{n}{2} \right \rceil \left\lfloor \tfrac{n}{2} \right \rfloor (n-2)^2 + \tbinom{\left\lceil \tfrac{n}{2} \right \rceil}{2} {\left\lfloor \tfrac{n}{2} \right \rfloor}^2 + \tbinom{\left\lfloor \tfrac{n}{2} \right \rfloor}{2} {\left\lceil \tfrac{n}{2} \right \rceil}^2 
        = \left\lfloor\tfrac{n^2}{4}\right \rfloor(n-2)^2 + \left\lfloor\tfrac{n^2}{4}\right \rfloor ^2 -\tfrac{n}{2} \left\lfloor\tfrac{n^2}{4}\right \rfloor.
    \end{align*} 
    Suppose that $\mathcal{H}$ is a bipartite $3$-graph on $n$ vertices and $V_{1} \cup V_{2} = V(\mathcal{H})$ is the corresponding bipartition. 
    Then 
    \begin{align*}
        \norm{\mathcal{H}}_{2}
        &\le |V_1||V_2|(n-2)^2 + \tbinom{|V_1|}{2} |V_2|^2 + \tbinom{|V_2|}{2}|V_1|^2 \\ 
        & = |V_1||V_2|\left( (n-2)^2-\tfrac{n}{2}\right) + |V_1|^2|V_2|^2 
        \le \left\lfloor\tfrac{n^2}{4}\right \rfloor \left( (n-2)^2-\tfrac{n}{2}\right) + \left\lfloor\tfrac{n^2}{4}\right \rfloor^2 
        = \norm{\mathbb{B}_n}_{2}.
    \end{align*}
    Notice that the first inequality above holds only if $\mathcal{H}$ is complete bipartite, and the second inequality above holds only if it is also balanced. This completes the proof of Lemma~\ref{LEMMA:Bn-max-2norm}
\end{proof}

\begin{proof}[Proof of Theorem~\ref{THM:AES-Fano-L2norm}]
    Let $\varepsilon > 0$ be a sufficiently small constant and let $n$ be sufficiently large. 
    Let $\mathcal{H}$ be an $n$-vertex $\mathbb{F}$-free $3$-graph with $\norm{\mathcal{H}}_{2} = \mathrm{ex}_{\ell_2}(n,\mathbb{F})$. 
    Similar to the proof of Theorem~\ref{THM:2-norm-density}, it follows from~\eqref{equ:THM-min-ell2-degree-extremal} that 
    \begin{align*}
        \delta_{\ell_{2}}(\mathcal{H})
        \ge \left(4 \pi_{\ell_2}(\mathbb{F}) - \varepsilon \right)n^3
        = \left(\tfrac{5}{4} - \varepsilon\right)n^3. 
    \end{align*}
    Since $\mathbb{F}$ is $\ell_2$-degree-stable with respect to $\mathfrak{B}$, by choosing $\varepsilon$ sufficiently small and $n$ sufficiently large, it follows from the definition that $\mathcal{H} \in \mathfrak{B}$, that is, $\mathcal{H}$ is bipartite. 
    Since $\mathcal{H}$ is extremal, Lemma~\ref{LEMMA:Bn-max-2norm} implies that $\mathcal{H} \cong \mathbb{B}_{n}$. 
    This completes the proof of Theorem~\ref{THM:AES-Fano-L2norm}.
\end{proof}

\section{Proof of Theorem~\ref{THM:AES-5-multigraph}}\label{SEC:Multigraphs}
In this section, we prove Theorem~\ref{THM:AES-5-multigraph}, an Andr\'{a}sfai--Erd\H{o}s--S\'{o}s--type stability theorem for the multigraph Tur\'{a}n problem $\mathrm{ex}_{5}(n,\mathbb{K}_{4})$.

\subsection{Preparations}
Suppose that $\vec{\mathcal{G}}=(G_1,\ldots,G_5)$ is a $5$-multigraph on four vertices, say $w, x, y, z$. 
We say $\vec{\mathcal{G}}$ is \emph{saturated} if 
\begin{itemize}
    \item $C(wx) = C(yz) = C(wy) = C(xz) = [5]$, and 
    \item $C(wz) \cup C(xy) = [5]$ is a partition.
\end{itemize}
We call the pair $\{wz,xy\}$ the \emph{light pair} in $\vec{\mathcal{G}}$. 
Denote by $\mathcal{G}^{*}$ the family of all saturated $5$-multigraphs on four vertices. 

The following lemma establishes some basic properties of $\mathbb{K}_{4}$-free $5$-multigraphs on four vertices. 
We note that some of these properties were previously proven by Bellmann--Reiher in~\cite{BR19}; however, we include their proofs here for completeness.
\begin{lemma}\label{LEMMA:f5n-basic-property}
    Let $\vec{\mathcal{G}}=(G_1,\ldots,G_5)$ be a $5$-multigraph on the vertex set $\{w,x,y,z\}$. 
    Suppose that $\vec{\mathcal{G}}$ is $\mathbb{K}_{4}$-free and $\mu(wx)+\mu(yz) \ge \mu(wy)+\mu(xz)\ge \mu(wz)+\mu(xy)$. 
    Then the following statements hold: 
    \begin{enumerate}[(i)]
    \item\label{ITEM:no-876} Suppose that $\mu(wx)+\mu(yz) \ge 8$ and $\mu(wy)+\mu(xz) \ge 7$. 
    Then $C(wz) \cap C(xy) = \emptyset$. 
    In particular,  $\mu(wz)+\mu(xy)\le 5$. 
    \item\label{ITEM:f_5(4)=25} We have $|\vec{\mathcal{G}}| \le 25$. 
    \item\label{ITEM:23-deleing-from-Ex} Suppose that $|\vec{\mathcal{G}}| \ge 23$. Then $\vec{\mathcal{G}}$ is a subgraph of some member of $\mathcal{G}^{*}$. 
    \item\label{ITEM:22->xy=5} Suppose that $|\vec{\mathcal{G}}| \ge 22$. Then $\vec{\mathcal{G}}$ contains an edge with multiplicity $5$. 
    \item\label{ITEM:17->no-same-graph} Suppose that $\mu(wx)+\mu(yz)+\mu(wy)+\mu(xz) \ge 17$. 
    Then $C(wz) \cap C(xy) = \emptyset$.  
\end{enumerate}
\end{lemma}
\begin{proof}[Proof of Lemma~\ref{LEMMA:f5n-basic-property}]
    First, we prove~\eqref{ITEM:no-876}.
    Suppose to the contrary that $\mu(wx)+\mu(yz) \ge 8$ and $\mu(wy)+\mu(xz) \ge 7$, but $C(wz) \cap C(xy) \neq \emptyset$. 
    Fix an index $i \in C(wz) \cap C(xy)$, noting that $\{wz, xy\} \subseteq G_i$. 
    Since 
    \begin{align*}
        |C(wy) \cap C(xz)| 
        & \ge \mu(wy)+\mu(xz) -5 \ge 7 -5 =2 \quad\text{and} \\
        |C(wx) \cap C(yz)|
        & \ge \mu(wx)+\mu(yz) - 5 \ge 8-5 =3, 
    \end{align*}
    there exists  $j \in \left(C(wy) \cap C(xz)\right) \setminus \{i\}$ and $k \in \left(C(wx) \cap C(yz)\right) \setminus \{i, j\}$.
    However, this means that $\vec{\mathcal{G}}$ contains a copy of $\mathbb{K}_{4}$, a contradiction. 

    Next, we prove~\eqref{ITEM:f_5(4)=25}. 
    Suppose to the contrary that $|\vec{\mathcal{G}}| \ge 26$. 
    Then we have 
    \begin{align*}
        \mu(wx)+\mu(yz)
        & \ge \lceil 26/3 \rceil 
        = 9, \\
        \mu(wy)+\mu(xz)
        & \ge \lceil \left(26 - \mu(wx) - \mu(yz)\right)/2 \rceil
        \ge \lceil (26-5-5)/2 \rceil 
        = 8, \quad\text{and}\quad \\
        \mu(wz)+\mu(xy)
        & \ge 26 - \mu(wx) - \mu(yz) - \mu(wy) - \mu(xz)
        \ge 26 - 5-5-5-5
        = 6, 
    \end{align*}
    which contradicts~\eqref{ITEM:no-876}. 

    Next, we prove~\eqref{ITEM:23-deleing-from-Ex}.  
    Similar to the proof of~\eqref{ITEM:f_5(4)=25}, it follows from $|\vec{\mathcal{G}}| \ge 23$ that 
    \begin{align*}
        \mu(wx)+\mu(yz) \ge 8 \quad\text{and}\quad
        \mu(wy)+\mu(xz) \ge 7. 
    \end{align*}
    Therefore, by~\eqref{ITEM:no-876}, we have $C(wz) \cap C(xy) = \emptyset$. 
    This implies that $\vec{\mathcal{G}}$ is a subgraph of some member of $\mathcal{G}^{*}$.
    
    Next, we prove~\eqref{ITEM:22->xy=5}.  It follows from $|\vec{\mathcal{G}}| \ge 22$ that $\mu(wx)+\mu(yz) \ge \lceil 22/3 \rceil =8$. 
    By~\eqref{ITEM:no-876}, either $\mu(wy)+\mu(xz) \le 6$ or $\mu(wz)+\mu(xy)\le 5$. 
    It follows that $\mu(wx)+\mu(yz)$ is at least 
    \begin{align*}
        & \min\left\{22 - (\mu(wy)+\mu(xz)) - (\mu(wz)+\mu(xy)),~\left\lceil \tfrac{22-(\mu(wz)+\mu(xy))}{2} \right\rceil\right\} \\
        & \ge \min\{22 - 6 - 6,~\lceil (22-5)/2 \rceil \} 
        = 9. 
    \end{align*}
    Therefore, at least one of $wx, yz$ has multiplicity $5$. 

    Finally, we prove~\eqref{ITEM:17->no-same-graph}. 
    It follows from $\mu(wx)+\mu(yz)+\mu(wy)+\mu(xz) \ge 17$ that 
    \begin{align*}
        \mu(wx)+\mu(yz) \ge \lceil 17/2 \rceil = 9 \quad\text{and}\quad
        \mu(wy)+\mu(xz) \ge 17 - 10 =7.
    \end{align*}
    So, by~\eqref{ITEM:no-876}, we have $C(wz) \cap C(xy) = \emptyset$. 
    This proves~\eqref{ITEM:17->no-same-graph}, and hence, completes the proof of Lemma~\ref{LEMMA:f5n-basic-property}. 
\end{proof}

Let $a, b, c$ be integers. 
We say a triple $\{x,y,z\} \subseteq V(\vec{\mathcal{G}})$ is of \emph{$\{a,b,c\}$-type} in $\vec{\mathcal{G}}$ if 
\begin{align*}
    \mu(xy) = a, \quad 
    \mu(xz) = b, \quad\text{and}\quad 
    \mu(yz) = c. 
\end{align*}
\begin{lemma}\label{LEM:key-lemma-no-333-10-cases}
    Let $n \in \mathbb{N}$. 
    Suppose that $\vec{\mathcal{G}}=(G_1,\ldots,G_5)$ is a $\mathbb{K}_{4}$-free $5$-multigraph on $n$ vertices with $\delta(\vec{\mathcal{G}}) \ge 44n/13$.
    Then $\vec{\mathcal{G}}$ does not contain a triple of vertices of $\{i,j,k\}$-type with $\min\{i,j,k\} \ge 3$. 
\end{lemma}
\begin{proof}[Proof of Lemma~\ref{LEM:key-lemma-no-333-10-cases}]
    Suppose to the contrary that the lemma fails. 
    Fix an arbitrary triple $\{x,y,z\} \subseteq V(\vec{\mathcal{G}})$ of $\{i,j,k\}$-type with $\min\{i,j,k\} \ge 3$. 
    By symmetry, we may assume that 
    \begin{align*}
        \big(\mu(xy),~\mu(xz),~\mu(yz) \big) = (i,j,k)
        \quad\text{and}\quad 
        i \ge j \ge k. 
    \end{align*}
    Let $V \coloneqq V(\vec{\mathcal{G}})$. For every subset $U \subseteq V$, define $\overline{U} \coloneqq V\setminus U$. 
    It is clear that 
    \begin{align}\label{EQ:sum-le-2eU+}
        \sum_{u \in U}d_{\vec{\mathcal{G}}}(u) 
        = 2|\vec{\mathcal{G}}[U]|
            +\sum_{u \in U}\sum_{v \in \overline{U}} \mu(uv). 
    \end{align} 

    \begin{claim}\label{CLAIM:exist-w-sum-11}
        There exists a vertex $w\in V\setminus \{x,y,z\}$ such that $\mu(wx)+\mu(wy)+\mu(wz) \ge 11$. 
    \end{claim}
    \begin{proof}[Proof of Claim~\ref{CLAIM:exist-w-sum-11}]
        Applying~\eqref{EQ:sum-le-2eU+} to the set $\{x,y,z\}$, we obtain 
        \begin{align*}
            d_{\vec{\mathcal{G}}}(x)+d_{\vec{\mathcal{G}}}(y)+d_{\vec{\mathcal{G}}}(z)  
            & = 2 |\vec{\mathcal{G}}[\{x,y,z\}]| + \sum_{w\in V \setminus \{x,y,z\}}\left(\mu(wx)+\mu(wy)+\mu(wz) \right). 
        \end{align*}
        Since $|\vec{\mathcal{G}}[\{x,y,z\}]| \le 3 \cdot 5 = 15$ and  $\delta(\vec{\mathcal{G}}) \ge 44n/13$, it follows from the inequality above that  
        \begin{align*}
            \sum_{w\in V\setminus \{x,y,z\}}\left(\mu(wx)+\mu(wy)+\mu(wz) \right) 
            \ge 3\delta(\vec{\mathcal{G}}) - 30 
            \ge 3 \cdot \tfrac{44}{13}n -30 
            > 10(n-3). 
        \end{align*}
        By averaging, there exists a vertex $w\in V\setminus \{x,y,z\}$ such that 
        \begin{align*}
            \mu(wx)+\mu(wy)+\mu(wz) 
            \ge 11,
        \end{align*}
        which proves Claim~\ref{CLAIM:exist-w-sum-11}.
    \end{proof}

    Next, we consider the following $10$ possible cases for the triple $(i,j,k)$ and show that none of them can occur: 
    \begin{align*}
        (5,5,5),~(5,5,4),~(5,4,4),~(4,4,4),~(5,5,3),\\
        (5,4,3),~(5,3,3),~(4,4,3),~(4,3,3),~(3,3,3).
    \end{align*}
    It is worth noting that, since $\{x,y,z\}$ was chosen arbitrarily, whenever we show that $(i,j,k) \neq (a,b,c)$ for some $(a,b,c)$ in the list above, it implies that the set $V$ does not contain any triple of $\{a,b,c\}$-type. 

    \bigskip 
    
    \noindent\textbf{Case 1}: $(i,j,k) = (5,5,5)$. 
    
    Let $w$ be the vertex guaranteed by Claim~\ref{CLAIM:exist-w-sum-11}, and let $S \coloneqq \{w, x, y, z\}$. Then we have  
    \begin{align*}
        |\vec{\mathcal{G}}[S]|
        = \mu(xy) + \mu(xz) + \mu(yz) + \mu(wx)+\mu(wy)+\mu(wz)
        \ge 5+5+5+11 
        = 26,   
    \end{align*}
    which contradicts Lemma~\ref{LEMMA:f5n-basic-property}~\eqref{ITEM:f_5(4)=25}. 
    Therefore, $(i,j,k) \neq (5,5,5)$. 

    \bigskip 
    
    \noindent\textbf{Case 2}: $(i,j,k) = (5,5,4)$.
    
    For every vertex $w\in V \setminus \{x,y,z\}$, it follows from Lemma~\ref{LEMMA:f5n-basic-property}~\eqref{ITEM:f_5(4)=25} that 
    \begin{align}\label{equ:wx-wy-wz-upper-bound-11}
        \mu(wx)+\mu(wy)+\mu(wz)
        = |\vec{\mathcal{G}}[\{w,x,y,z\}]| - \mu(xy) - \mu(xz) - \mu(yz)
        \le 25 - 14 
        = 11.
    \end{align}
    Consider the following partition of $V\setminus\{x,y,z\}$:
    \begin{align*}
        U_1 
        & \coloneqq  \left\{w \in V \setminus \{x,y,z\} \colon \mu(wx)+\mu(wy)+\mu(wz) \le 9 \right\} \quad\text{and} \\
        U_2 
        & \coloneqq  \left\{w \in V \setminus \{x,y,z\} \colon \mu(wx)+\mu(wy)+\mu(wz) \ge 10 \right\}.
    \end{align*}
    Applying~\eqref{EQ:sum-le-2eU+} to the set $\{x,y,z\}$, and using~\eqref{equ:wx-wy-wz-upper-bound-11}, we obtain 
    \begin{align}\label{EQ:554-9U1+11U2}
        d_{\vec{\mathcal{G}}}(x)+d_{\vec{\mathcal{G}}}(y)+d_{\vec{\mathcal{G}}}(z) 
        & = 2 \left(\mu(xy) + \mu(xz) + \mu(yz)\right) + \sum_{w\in U_1 \cup U_2} \left(\mu(wx)+\mu(wy)+\mu(wz)\right) \notag \\ 
        & \le 28 + 9|U_1| + 11|U_2|. 
    \end{align}

    \begin{claim}\label{CLAIM:mu-wx-1}
        We have $\mu(wx) \le 1$ for every $w \in U_2$. 
    \end{claim}
    \begin{proof}[Proof of Claim~\ref{CLAIM:mu-wx-1}]
        Fix a vertex $w \in U_2$ and let $S \coloneqq \{w,x,y,z\}$. 
        Suppose to the contrary that $\mu(wx) \ge 2$. 
        It follows from the definition of $U_2$ that $|\vec{\mathcal{G}}[S]|\ge 5+5+4+10 = 24$. 
        So by Lemma~\ref{LEMMA:f5n-basic-property}~\eqref{ITEM:23-deleing-from-Ex}, $\vec{\mathcal{G}}[S]$ is a subgraph of some member $\vec{\mathcal{G}}'$ of $\mathcal{G}^{*}$. 
        
        Since $\mu(wx) + \mu(yz) \ge 2 + 4 = 6$, by the definition of saturated $5$-multigraphs, the pair $\{wx, yz\}$ cannot be the light pair of $\vec{\mathcal{G}}'$. 
        Thus, we have $\mu_{\vec{\mathcal{G}}'}(wx) = \mu_{\vec{\mathcal{G}}'}(yz) = 5$. 
        Since $\mu_{\vec{\mathcal{G}}'}(yz) - \mu_{\vec{\mathcal{G}}}(yz) = 1 = |\vec{\mathcal{G}}'| - |\vec{\mathcal{G}}[S]|$, we have $\mu_{\vec{\mathcal{G}}}(wx) = \mu_{\vec{\mathcal{G}}'}(wx) = 5$. 
        For similar reasons, we have $\mu(wy) \in \{0,5\}$ and $\mu(wz) \in \{0,5\}$. 
        Since $\mu(wy) + \mu(wz) \ge 10 - \mu(wx) =5$, either $\mu(wy) = 5$ or $\mu(wz) = 5$.
        By symmetry, we may assume that $\mu(wy) = 5$. 
        However, this means that $\{w,x,y\}$ is a triple of $\{5,5,5\}$-type, which is impossible by Case 1.
        This proves Claim~\ref{CLAIM:mu-wx-1}. 
    \end{proof}

    It follows from Claim~\ref{CLAIM:mu-wx-1} that 
    \begin{align}\label{EQ:554-5U1+U2}
        d_{\vec{\mathcal{G}}}(x) 
        = \mu(xy) + \mu(xz) + \sum_{w\in U_1 \cup U_2} \mu(xw)
        \le 10 + 5|U_1| + |U_2|. 
    \end{align}
    Adding twice Inequality~\eqref{EQ:554-9U1+11U2} to Inequality~\eqref{EQ:554-5U1+U2}  yields 
    \begin{align*}
        3d_{\vec{\mathcal{G}}}(x)+2d_{\vec{\mathcal{G}}}(y)+2d_{\vec{\mathcal{G}}}(z) 
        \le 66 + 23|U_1|+23|U_2|=23n-3, 
    \end{align*}
    which contradicts the lower bound 
    \begin{align*}
        3d_{\vec{\mathcal{G}}}(x)+2d_{\vec{\mathcal{G}}}(y)+2d_{\vec{\mathcal{G}}}(z) 
        \ge 7\delta(\vec{\mathcal{G}})
        \ge 7 \cdot \tfrac{44}{13}n
        >23n-3. 
    \end{align*}
    Therefore, $(i,j,k) \neq (5,5,4)$. 
    
    \bigskip

    \noindent\textbf{Case 3}: $(i,j,k) = (5,4,4)$.
 
    Let $w$ be the vertex guaranteed by Claim~\ref{CLAIM:exist-w-sum-11} and let $S\coloneqq \{w,x,y,z\}$. 
    Observe that 
    \begin{align*}
        |\vec{\mathcal{G}}[S]|
        = \mu(xy) + \mu(xz) + \mu(yz) + \mu(wx)+\mu(wy)+\mu(wz)
        \ge 5+4+4+11 
        = 24.   
    \end{align*}
    As in the proof of Case 2, Lemma~\ref{LEMMA:f5n-basic-property}~\eqref{ITEM:23-deleing-from-Ex} implies that $\mu(wz) = 5$, and either $\mu(wx) = 5$ or $\mu(wy) = 5$. 
    If $\mu(wx) = 5$, then the triple $\{w,x,z\}$ is of $\{5,5,4\}$-type, which is impossible by Case 2. 
    If $\mu(wy) = 5$, then the triple $\{w,y,z\}$ is of $\{5,5,4\}$-type, which is impossible by Case 2. 
    Therefore, $(i,j,k) \neq (5,4,4)$.

    \bigskip

    \noindent\textbf{Case 4}: $(i,j,k) = (4,4,4)$.

    Let $w$ be the vertex guaranteed by Claim~\ref{CLAIM:exist-w-sum-11} and let $S\coloneqq \{w,x,y,z\}$. 
    Observe that 
    \begin{align*}
        |\vec{\mathcal{G}}[S]|
        = \mu(xy) + \mu(xz) + \mu(yz) + \mu(wx)+\mu(wy)+\mu(wz)
        \ge 4+4+4+11 
        = 23.   
    \end{align*}
    As in the proof of Case 2, Lemma~\ref{LEMMA:f5n-basic-property}~\eqref{ITEM:23-deleing-from-Ex} implies that two edges in $\{wx, wy, wz\}$ have multiplicity $5$.
    By symmetry, we may assume that $\mu(wx) = \mu(wy) = 5$.
    Then the triple $\{w,x,y\}$ is of $\{5,5,4\}$-type, which is impossible by Case 2.
    Therefore, $(i,j,k) \neq (4,4,4)$.

    \bigskip

    \noindent\textbf{Case 5}: $(i,j,k) = (5,5,3)$.

    We follow a similar proof as in the case where $(i,j,k) = (5,5,4)$. 
    For every vertex $w\in V \setminus \{x,y,z\}$, it follows from Lemma~\ref{LEMMA:f5n-basic-property}~\eqref{ITEM:f_5(4)=25} that 
    \begin{align}\label{equ:wx-wy-wz-upper-bound-12}
        \mu(wx)+\mu(wy)+\mu(wz)
        = |\vec{\mathcal{G}}[\{w,x,y,z\}]| - \left(\mu(xy) + \mu(xz) + \mu(yz)\right)
        \le 25 - 13 
        = 12.
    \end{align}
    Consider the following partition of $V\setminus\{x,y,z\}$: 
    \begin{align*}
        U_1 
        & \coloneqq \left\{w \in V \setminus \{x,y,z\} \colon  \mu(wx)+\mu(wy)+\mu(wz) \le 8 \right\},\\
        U_2 
        & \coloneqq \left\{w \in V \setminus \{x,y,z\} \colon \mu(wx)+\mu(wy)+\mu(wz) = 9 \right\}, \quad\text{and} \\
        U_3 
        & \coloneqq \left\{w \in V \setminus \{x,y,z\} \colon \mu(wx)+\mu(wy)+\mu(wz) \ge 10 \right\}.
    \end{align*}
    Applying~\eqref{EQ:sum-le-2eU+} to the set $\{x,y,z\}$, and using~\eqref{equ:wx-wy-wz-upper-bound-12},  we obtain 
    \begin{align}\label{EQ:553-8U3+9U4+12U5}
        d_{\vec{\mathcal{G}}}(x)+d_{\vec{\mathcal{G}}}(y)+d_{\vec{\mathcal{G}}}(z) 
        & = 2 |\vec{\mathcal{G}}[\{x,y,z\}]| + \sum_{w\in U_1 \cup U_2 \cup U_3} \left(\mu(wx)+\mu(wy)+\mu(wz)\right) \notag \\ 
        & \le 26 + 8|U_1| + 9|U_2| + 12|U_3|. 
    \end{align}

    \begin{claim}\label{CLAIM:mu-wx-3-U4}
        We have $\mu(wx) \le 3$ for every $w \in U_2$. 
    \end{claim}
    \begin{proof}[Proof of Claim~\ref{CLAIM:mu-wx-3-U4}]
        Fix a vertex $w \in U_2$. 
        Suppose to the contrary that $\mu(wx) \ge 4$. 
        By symmetry, we may assume that $\mu(wy) \ge \mu(wz)$. 

        Suppose that $\mu(wx)=4$. 
        Then $\mu(wy)+\mu(wz) = 9 - \mu(wx) = 5$. 
        It implies that $\left(\mu(wy), \mu(wz)\right) = (3,2)$, since otherwise, the triple $\{y,x,w\}$ is either of $\{5,5,4\}$-type (if $\left(\mu(wy), \mu(wz)\right) = (5,0)$) or of $\{5,4,4\}$-type (if $\left(\mu(wy), \mu(wz)\right) = (4,1)$). 
        Both are impossible by Case 2 and Case 3, respectively. 

        Now we know the multiplicity of every pair in $\{w,x,y,z\}$.
        In particular, we have 
        \begin{align*}
            \mu(xz)+\mu(wy) = 5 + 3 = 8, \quad 
            \mu(xy)+\mu(wz) = 5 + 2 = 7, \quad 
            \mu(wx)+\mu(yz) = 4 + 3 =7,
        \end{align*}
        which contradicts Lemma~\ref{LEMMA:f5n-basic-property}~\eqref{ITEM:no-876}.
        Therefore, $\mu(wx) \neq 4$. 

        Now suppose that $\mu(wx) = 5$. 
        It is easy to see that $\mu(wy) \le 3$, since otherwise, the triple $\{x,y,w\}$ is of $\{5,5,4\}$-type, which is impossible by Case 2. 
        
        Suppose that $\left(\mu(wy), \mu(wz)\right) = (3,1)$. 
        Then 
        \begin{align*}
            \mu(xz)+\mu(wy) = 5 + 3 = 8, \quad 
            \mu(xy)+\mu(wz) = 5 + 1 = 6, \quad 
            \mu(wx)+\mu(yz) = 5 + 3 =8,
        \end{align*}
        which contradicts Lemma~\ref{LEMMA:f5n-basic-property}~\eqref{ITEM:no-876}.

        Suppose that $\left(\mu(wy), \mu(wz)\right) = (2,2)$. 
        Then 
        \begin{align*}
            \mu(xz)+\mu(wy) = 5 + 2 = 7, \quad 
            \mu(xy)+\mu(wz) = 5 + 2 = 7, \quad 
            \mu(wx)+\mu(yz) = 5 + 3 = 8,
        \end{align*}
        which also contradicts Lemma~\ref{LEMMA:f5n-basic-property}~\eqref{ITEM:no-876}.
        This completes the proof of Claim~\ref{CLAIM:mu-wx-3-U4}. 
    \end{proof}

    \begin{claim}\label{CLAIM:mu-wx-2-U5}
        We have $\mu(wx) \le 2$ for every $w \in U_3$. 
    \end{claim}
    \begin{proof}[Proof of Claim~\ref{CLAIM:mu-wx-2-U5}]
        Fix a vertex $w \in U_3$ and let $S \coloneqq \{w,x,y,z\}$. 
        Observe that 
        \begin{align*}
            |\vec{\mathcal{G}}[S]|
            = \mu(xy) + \mu(xz) + \mu(yz) + \mu(wx)+\mu(wy)+\mu(wz)
            \ge 5+5+3+10
            = 23.
        \end{align*}
        So by Lemma~\ref{LEMMA:f5n-basic-property}~\eqref{ITEM:23-deleing-from-Ex}, $\vec{\mathcal{G}}[S]$ is a subgraph of some member $\vec{\mathcal{G}}'$ of $\mathcal{G}^{\ast}$. 

        Suppose that $\{wy, xz\}$ is the light pair in $\vec{\mathcal{G}}'$. Then $\mu(wy) = 5 - \mu(xz) = 0$. 
        It follows that $\mu(wx) = \mu(wz) = 5$, which implies that the triple $\{w,x,z\}$ is of $\{5,5,5\}$-type, which is impossible by Case 1. 
        Therefore, $\{wy, xz\}$ cannot be the light pair. 
        By symmetry, $\{wz, xy\}$ cannot be the light pair either.
        Thus $\{wx, yz\}$ must be the light pair in $\vec{\mathcal{G}}'$. 
        It follows that $\mu(wx) \le 5 - \mu(yz) = 2$. 
        This completes the proof of Claim~\ref{CLAIM:mu-wx-2-U5}.
    \end{proof}

    It follows from Claims~\ref{CLAIM:mu-wx-3-U4} and~\ref{CLAIM:mu-wx-2-U5} that 
        \begin{align}\label{EQ:553-5U3+3U4+2U5}
            d_{\vec{\mathcal{G}}}(x) 
            = \mu(xy) + \mu(xz) + \sum_{w\in U_1\cup U_2\cup U_3} \mu(xw)
            \le 10 + 5|U_1| + 3|U_2| + 2|U_3|,
        \end{align}
    Taking $3\times\eqref{EQ:553-8U3+9U4+12U5} + 4\times \eqref{EQ:553-5U3+3U4+2U5}$, we obtain 
    \begin{align*}
        7d_{\vec{\mathcal{G}}}(x)+3d_{\vec{\mathcal{G}}}(y)+3d_{\vec{\mathcal{G}}}(z) 
        & \le 118+44|U_1|+39|U_2|+44|U_3| \\
        & \le 118+44(n-3)
        = 44n - 14,
    \end{align*}
    which contradicts the lower bound 
    \begin{align*}
        7d_{\vec{\mathcal{G}}}(x)+3d_{\vec{\mathcal{G}}}(y)+3d_{\vec{\mathcal{G}}}(z) 
        \ge 13 \delta(\vec{\mathcal{G}}) 
        \ge 13 \cdot \tfrac{44}{13}n
        > 44n-14. 
    \end{align*}
    Therefore, $(i,j,k) \neq (5,5,3)$. 

    \bigskip

    \noindent\textbf{Case 6}: $(i,j,k) = (5,4,3)$. 

    \begin{claim}\label{CLAIM:case-543-wx-wy-wz-11}
        We have $\mu(wx)+\mu(wy)+\mu(wz) \le 11$ for every $w \in V \setminus \{x,y,z\}$. 
    \end{claim}
    \begin{proof}[Proof of Claim~\ref{CLAIM:case-543-wx-wy-wz-11}]
        Suppose to the contrary that $\mu(wx)+\mu(wy)+\mu(wz) \ge 12$ for some $w \in V \setminus \{x,y,z\}$.
        Let $S \coloneqq \{w,x,y,z\}$. 
        Then 
        \begin{align*}
            |\vec{\mathcal{G}}[S]|
            = \mu(xy) + \mu(xz) + \mu(yz) + \mu(wx)+\mu(wy)+\mu(wz)
            \ge 5+4+3+12
            = 24. 
        \end{align*}
        So by Lemma~\ref{LEMMA:f5n-basic-property}~\eqref{ITEM:23-deleing-from-Ex}, $\vec{\mathcal{G}}[S]$ is a subgraph of some member $\vec{\mathcal{G}}'$ of $\mathcal{G}^{*}$.
        Since $|\vec{\mathcal{G}}[S]| \ge 24 = 25-1$, the light pair in $\vec{\mathcal{G}}'$ can only be $\{yz, wx\}$, and moreover, $\mu(wx) = 5-\mu(yz) = 2$. 
        Consequently, $\mu(wy) = \mu(wz) = 5$, and hence the triple $\{w,y,z\}$ is of $\{5,5,3\}$-type, which is impossible by Case 5.
        This completes the proof of Claim~\ref{CLAIM:case-543-wx-wy-wz-11}. 
    \end{proof}

    Consider the following partition of $V\setminus \{x,y,z\}$: 
    \begin{align*}
        U_1 
        & \coloneqq \left\{w \in V \setminus \{x,y,z\} \colon  \mu(wx)+\mu(wy)+\mu(wz) \le 9 \right\},\\
        U_2 
        & \coloneqq \left\{w \in V \setminus \{x,y,z\} \colon \mu(wx)+\mu(wy)+\mu(wz) = 10 \right\},\quad\text{and}\\
        U_3 
        & \coloneqq \left\{w \in V \setminus \{x,y,z\} \colon \mu(wx)+\mu(wy)+\mu(wz) =11 \right\}.
    \end{align*}
    Applying~\eqref{EQ:sum-le-2eU+} to the set $\{x,y,z\}$ yields
    \begin{align}\label{EQ:543-9U6+10U7+11U8}
        d_{\vec{\mathcal{G}}}(x)+d_{\vec{\mathcal{G}}}(y)+d_{\vec{\mathcal{G}}}(z)
        & = 2|\vec{\mathcal{G}}[\{x,y,z\}]| + \sum_{w\in U_1\cup U_2 \cup U_3} \left(\mu(wx)+\mu(wy)+\mu(wz)\right)  \notag \\
        & \le 24 + 9|U_1|+10|U_2|+11|U_3|. 
    \end{align}

    \begin{claim}\label{CLAIM:543-mu-wx-U7}
        We have $\mu(wx) \le 3$ for every $w \in U_2$. 
    \end{claim}
    \begin{proof}[Proof of Claim~\ref{CLAIM:543-mu-wx-U7}]
        Suppose to the contrary that $\mu(wx) \ge 4$ for some vertex $w \in U_2$. 

        Suppose that $\mu(wx) = 4$. 
        Since we have shown that $V$ contains no triples of type $\{5,5,4\}$, $\{5,4,4\}$, or $\{4,4,4\}$,  it follows that $\mu(wy) \le 3$ and $\mu(wz) \le 3$. 
        On the other hand, we have $\mu(wy) + \mu(wz) = 10 - \mu(wx) = 6$, so it must be that $\mu(wy) = \mu(wz) = 3$. 
        This implies that 
        \begin{align*}
            \mu(xy) + \mu(wz) = 8, \quad
            \mu(xz) + \mu(wy) = 7, \quad 
            \mu(yz) + \mu(wx) = 7,
        \end{align*}
        which contradicts Lemma~\ref{LEMMA:f5n-basic-property}~\eqref{ITEM:no-876}.

        Suppose that $\mu(wx) = 5$. 
        As before, we have $\mu(wy) \le 3$ and $\mu(wz) \le 3$. 
        Since $\mu(wy) + \mu(wz) = 10 - \mu(wx) = 5$, there are only two possibilities: $\left(\mu(wy), \mu(wz)\right) = (3,2)$ or $\left(\mu(wy), \mu(wz)\right) = (2,3)$. 
        As in the argument above, it is straightforward to verify that both cases are ruled out by Lemma~\ref{LEMMA:f5n-basic-property}~\eqref{ITEM:no-876}. 
        This completes the proof of Claim~\ref{CLAIM:543-mu-wx-U7}. 
    \end{proof}

    \begin{claim}\label{CLAIM:543-mu-wx-U8}
        We have $\mu(wx) \le 2$ for every $w \in U_3$. 
    \end{claim}
    \begin{proof}[Proof of Claim~\ref{CLAIM:543-mu-wx-U8}]
        Suppose to the contrary that $\mu(wx) \ge 3$ for some vertex $w \in U_3$. 
        Let $S \coloneqq \{w,x,y,z\}$. 
        Observe that 
        \begin{align*}
            |\vec{\mathcal{G}}[S]|
            = \mu(xy) + \mu(xz) + \mu(yz) + \mu(wx)+\mu(wy)+\mu(wz)
            = 5+4+3+11
            = 23.
        \end{align*}
        By Lemma~\ref{LEMMA:f5n-basic-property}~\eqref{ITEM:23-deleing-from-Ex}, $\vec{\mathcal{G}}[S]$ is a subgraph of some member $\vec{\mathcal{G}}'$ of $\mathcal{G}^{\ast}$. 
        Since $\mu(wx) + \mu(yz) \ge 3+3 = 6$, the pair $\{wx,yz\}$ cannot be the light pair in $\vec{\mathcal{G}}'$.
        It follows that $\mu(wx) = 5$, since otherwise, we would have $|\vec{\mathcal{G}}[S]| \le 4+3 + 5+5 + 5 = 22$, a contradiction. 
        Note that $\mu(wz) = 11 - \mu(wx) - \mu(wy) \ge 11 - 5 -5 \ge 1$. Thus $\mu(wz) + \mu(xy) \ge 1+ 5 = 6$, which implies that the pair $\{wz,xy\}$ cannot be the light pair in $\vec{\mathcal{G}}'$.
        It follows that $\mu(wz) = 5$, since otherwise, we would have $|\vec{\mathcal{G}}[S]| \le 5+4 + 5+3 + 5 = 22$, a contradiction. 
        However, this means that the triple $\{w,x,z\}$ is of $\{5,5,4\}$-type, which is impossible by Case 2.
        This completes the proof of Claim~\ref{CLAIM:543-mu-wx-U8}. 
    \end{proof}
    
    It follows from Claims~\ref{CLAIM:543-mu-wx-U7} and~\ref{CLAIM:543-mu-wx-U8} that 
    \begin{align}\label{EQ:543-5U6+3U7+2U8}
        d_{\vec{\mathcal{G}}}(x) 
        = \mu(xy) + \mu(xz) + \sum_{w\in U_1\cup U_2\cup U_3} \mu(wx)
        \le 9 + 5|U_1| + 3|U_2|+2|U_3|. 
    \end{align}
    However, $3 \times \eqref{EQ:543-9U6+10U7+11U8} + 2\times \eqref{EQ:543-5U6+3U7+2U8}$ yields  
    \begin{align*}
        5d_{\vec{\mathcal{G}}}(x)+3d_{\vec{\mathcal{G}}}(y)+3d_{\vec{\mathcal{G}}}(z) 
        \le 90 + 37|U_1|+36|U_2|+37|U_3| 
        \le 90 + 37(n-3) = 37n-21, 
    \end{align*}
    which contradicts the lower bound 
    \begin{align*}
        5d_{\vec{\mathcal{G}}}(x)+3d_{\vec{\mathcal{G}}}(y)+3d_{\vec{\mathcal{G}}}(z) 
        \ge 11\delta(\vec{\mathcal{G}}) 
        \ge 11 \cdot \tfrac{44}{13} n
        >37n. 
    \end{align*}
    Therefore, $(i,j,k) \neq (5,4,3)$. 

    \bigskip

    \noindent\textbf{Case 7}: $(i,j,k) = (5,3,3)$. 

    Let $w$ be the vertex guaranteed by Claim~\ref{CLAIM:exist-w-sum-11} and let $S \coloneqq \{w,x,y,z\}$. 

    \begin{claim}\label{CASE:533-wx-wy-wz-11}
        We have $\mu(wx)+\mu(wy)+\mu(wz) = 11$. 
    \end{claim}
    \begin{proof}[Proof of Claim~\ref{CASE:533-wx-wy-wz-11}]
        Suppose to the contrary that $\mu(wx)+\mu(wy)+\mu(wz) \ge 12$. 
        Then 
        \begin{align*}
            |\vec{\mathcal{G}}[S]|
            = \mu(xy) + \mu(xz) + \mu(yz) + \mu(wx)+\mu(wy)+\mu(wz)
            \ge 5+3+3+12
            = 23. 
        \end{align*}
        By Lemma~\ref{LEMMA:f5n-basic-property}~\eqref{ITEM:23-deleing-from-Ex}, $\vec{\mathcal{G}}[S]$ is a subgraph of some member $\vec{\mathcal{G}}'$ of $\mathcal{G}^{\ast}$. 
        Since $\mu(xy) = 5$ and 
        \begin{align*}
            \min\{\mu(wx),\mu(wy),\mu(wz)\} 
            \ge 12 - 5-5
            = 2, 
        \end{align*}
        the pair $\{xy, wz\}$ cannot be the light pair in $\vec{\mathcal{G}}'$. 
        By symmetry, we may assume that $\{xz,wy\}$ is the light pair in $\vec{\mathcal{G}}'$. 
        Since $|\vec{\mathcal{G}}[S]| \ge 23 = 25-2$, it must hold that $\mu(wy) = 2$ and $\mu(wz) = \mu(wx) = 5$. 
        However, it implies that the triple $\{w,x,z\}$ is of $\{5,5,3\}$-type, which is impossible by Case 5.
        This completes the proof of Claim~\ref{CASE:533-wx-wy-wz-11}. 
    \end{proof}

    \begin{claim}\label{CASE:533-wx-wy-3}
        We have $\max\{\mu(wx), \mu(wy)\} \le 3$. 
    \end{claim}
    \begin{proof}[Proof of Claim~\ref{CASE:533-wx-wy-3}]
        Suppose to the contrary that $\max\{\mu(wx), \mu(wy)\} \ge 4$.
        By symmetry, we may assume that $\mu(wx) \ge 4$. 

        Suppose that $\mu(wx) = 5$. 
        Then it follows from Claim~\ref{CASE:533-wx-wy-wz-11} that $\mu(wy) + \mu(wz) = 11-5 = 6$. 
        By Cases 1, 2, and 5, we have $\mu(wy) \le 2$, and hence, $\mu(wz) \ge 6 - 2 = 4$. 
        This means that the triple $\{w,x,z\}$ is either of $\{5,5,3\}$-type or of $\{5,4,3\}$-type, which is impossible by Cases 5 and 6, respectively.   

        Suppose that $\mu(wx) = 4$. 
        Then it follows from Claim~\ref{CASE:533-wx-wy-wz-11} that $\mu(wy) + \mu(wz) = 11-4 = 7$. 
        By Cases 2, 3, and 6, we have $\mu(wy) \le 2$, and hence, $\mu(wz) \ge 7 - 2 = 5$.
        However, this means that the triple $\{w,z,x\}$ is of $\{5,4,3\}$-type, which is impossible by Case 6. 
        This completes the proof of Claim~\ref{CASE:533-wx-wy-3}. 
    \end{proof}

    It follows from Claims~\ref{CASE:533-wx-wy-wz-11} and~\ref{CASE:533-wx-wy-3} that $\left(\mu(wx), \mu(wy), \mu(wz)\right) = (3,3,5)$. 
    Applying~\eqref{EQ:sum-le-2eU+} to the set $S = \{w,x,y,z\}$ yields 
    \begin{align*}
        d_{\vec{\mathcal{G}}}(w)+d_{\vec{\mathcal{G}}}(x)+d_{\vec{\mathcal{G}}}(y)+d_{\vec{\mathcal{G}}}(z) 
        & = 2 |\vec{\mathcal{G}}[S]| + \sum_{u\in V \setminus S} \left(\mu(uw)+\mu(ux)+\mu(uy)+\mu(uz) \right) \\
        & = 44 + \sum_{u\in V \setminus S} \left(\mu(uw)+\mu(ux)+\mu(uy)+\mu(uz) \right). 
    \end{align*}
    By averaging, there exists a vertex $u \in V\setminus S$ such that 
    \begin{align*}
        \mu(uw)+\mu(ux)+\mu(uy)+\mu(uz) 
        & \ge \left\lceil \tfrac{d_{\vec{\mathcal{G}}}(w)+d_{\vec{\mathcal{G}}}(x)+d_{\vec{\mathcal{G}}}(y)+d_{\vec{\mathcal{G}}}(z) - 44}{n-4}  \right\rceil 
        \ge \left\lceil \tfrac{4 \cdot 44n/13 - 44}{n-4} \right\rceil
        \ge 14. 
    \end{align*}
    By averaging again, there exists a triple $\{a,b,c\} \subseteq \{w,x,y,z\}$ such that 
    \begin{align*}
        \mu(ua) + \mu(ub) + \mu(uc)
        \ge \left\lceil \tfrac{3}{4} \cdot 14 \right\rceil 
        = 11. 
    \end{align*}
    Since every triple in $\{w,x,y,z\}$ is of $\{5,3,3\}$-type, by symmetry, we may assume that $\{a,b,c\} = \{x,y,z\}$.
    Repeating the above argument with $w$ replaced by $z$, we obtain  $\left(\mu(ux), \mu(uy), \mu(uz)\right) = (3,3,5)$. 
    It follows that $\mu(uw) \ge 14 - \left( \mu(ux)+\mu(uy)+\mu(uz) \right) = 3$. 
    However, this implies that the triple $\{z, u, w\}$ is of $\{5,5,\mu(uw)\}$-type with $\mu(uw) \ge 3$, which is impossible by Cases 1, 2, and 5. 
    Therefore, $(i,j,k) \neq (5,3,3)$. 

    \bigskip

    \noindent\textbf{Case 8}: $(i,j,k) = (4,4,3)$.

    Let $w$ be the vertex guaranteed by Claim~\ref{CLAIM:exist-w-sum-11} and let $S \coloneqq \{w,x,y,z\}$.
    Note that 
    \begin{align*}
        |\vec{\mathcal{G}}[S]|
        = \mu(xy) + \mu(xz) + \mu(yz) + \mu(wx)+\mu(wy)+\mu(wz)
        \ge 4+4+3+11
        = 22. 
    \end{align*}
    It follows from Lemma~\ref{LEMMA:f5n-basic-property}~\eqref{ITEM:22->xy=5} that $\vec{\mathcal{G}}[S]$ contains a pair (necessarily one of $\{wx, wy, wz\}$) with multiplicity $5$. 

    Suppose that $\mu(wx) = 5$. 
    By symmetry, we may assume that $\mu(wy) \ge \mu(wz)$, and thus, $\mu(wy) \ge (11-5)/2 = 3$. 
    Then the triple $\{x,w,y\}$ is of $\{5,4,\mu(wy)\}$-type with $\mu(wy) \ge 3$, which is impossible by Cases 2, 3, and 6. 

    Suppose that one of $\{wy, wz\}$ has multiplicity $5$. 
    By symmetry, we may assume that $\mu(wy) = 5$. 
    If $\mu(wx) \ge 3$, then the triple $\{y,w,x\}$ is of $\{5,4,\mu(wx)\}$-type with $\mu(wx) \ge 3$, which is impossible by Cases 2, 3, and 6. 
    Thus $\mu(wx) \le 2$, and hence, $\mu(wz) \ge 11 - 5 -2 = 4$. 
    However, this implies that the triple $\{w,y,z\}$ is of either $\{5,5,3\}$-type or $\{5,4,3\}$-type, which is impossible by Cases 5 and 6, respectively. 
    Therefore, $(i,j,k) \neq (4,4,3)$.

    \bigskip

    \noindent\textbf{Case 9}: $(i,j,k) = (4,3,3)$.

    Let $w$ be the vertex guaranteed by Claim~\ref{CLAIM:exist-w-sum-11}.

    \begin{claim}\label{CLAIM:433-wx-wy-3}
        We have $\max\{\mu(wx), \mu(wy)\}\le 3$. 
    \end{claim}
    \begin{proof}[Proof of Claim~\ref{CLAIM:433-wx-wy-3}]
        Suppose to the contrary that $\max\{\mu(wx), \mu(wy)\} \ge 4$. 
        By symmetry, we may assume that $\mu(wx) \ge 4$. 

        Suppose that $\mu(wx) = 4$. 
        Since $\mu(xy) = 4$ and $\{x,y,w\}$ cannot be of $\{4,4,3\}$-type (by Case 8), it follows that $\mu(wy) \le 2$. 
        Therefore, $\mu(wz) \ge 11-4-2 = 5$. 
        However, this implies that the triple $\{w,z,x\}$ is of $\{5,4,3\}$-type, which is impossible by Case 6. 

        Suppose that $\mu(wx) = 5$. 
        Since $\mu(xy) = 4$ and $\{x,w,y\}$ cannot be of $\{5,4,3\}$-type (by Case 6), we have $\mu(wy) \le 2$. 
        Therefore, $\mu(wz) \ge 11-5-2 = 4$. 
        However, this implies that the triple $\{w,x,z\}$ is of either $\{5,5,3\}$-type or $\{5,4,3\}$-type, which is impossible by Cases 5 and 6, respectively. 
        This proves Claim~\ref{CLAIM:433-wx-wy-3}. 
    \end{proof}

    Since $\mu(wx) + \mu(wy) + \mu(wz) \ge 11$, it follows from Claim~\ref{CLAIM:433-wx-wy-3} that $\left(\mu(wx), \mu(wy), \mu(wz)\right) = (3,3,5)$. 
    This implies that the triple $\{w,z,y\}$ is of $\{5,3,3\}$-type, which is impossible by Case 7. 
    Therefore, $(i,j,k) \neq (4,3,3)$.

    \bigskip

    \noindent\textbf{Case 10}: $(i,j,k) = (3,3,3)$.

    Let $w$ be the vertex guaranteed by Claim~\ref{CLAIM:exist-w-sum-11}. 
    By symmetry, we may assume that $\mu(wx) \ge \mu(wy) \ge \mu(wz)$. 
    Since $\mu(wx) + \mu(wy) + \mu(wz) \ge 11$, it follows that $\mu(wx) \ge 4$ and $\mu(wy) \ge 3$. 
    This implies that the triple $\{w,x,y\}$ is of $\{\mu(wx),  \mu(wy), 3\}$-type with $\mu(wx) \ge 4$ and $\mu(wy) \ge 3$, which is impossible by the cases we have already considered. 
    Therefore, $(i,j,k) \neq (3,3,3)$. 
    This completes the proof of Lemma~\ref{LEM:key-lemma-no-333-10-cases}. 
\end{proof}

\subsection{Proof of Theorem~\ref{THM:AES-5-multigraph}}
In this subsection, we present the proof of Theorem~\ref{THM:AES-5-multigraph}.
\begin{proof}[Proof of Theorem~\ref{THM:AES-5-multigraph}]
    Fix a positive integer $n$.
    Let $\vec{\mathcal{G}}=(G_1,\ldots,G_5)$ be an $n$-vertex $\mathbb{K}_{4}$-free $5$-multigraph with $\delta(\vec{\mathcal{G}}) \ge \alpha n$, where we will later consider two cases in the proof: $\alpha = 44/13$ and $\alpha = 58/17$. 
    
    Let $V \coloneqq V(\vec{\mathcal{G}})$. 
    For every vertex $v \in V$ and every $i \in [5]$, let
    \begin{align*}
        N^{(i)}(v)
        \coloneqq \left\{u \in V \colon \mu(uv) = i \right\}
        \quad\text{and}\quad 
        d^{(i)}(v)
        \coloneqq |N^{(i)}(v)|. 
    \end{align*}

    Define a graph
    \begin{align*}
        H
        \coloneqq \left\{\{u,v\} \in \tbinom{V}{2} \colon \mu(uv) \ge 3\right\}. 
    \end{align*}

    \begin{claim}\label{CLAIM:G-ast-bipartite}
        The graph $H$ is bipartite. 
    \end{claim}
    \begin{proof}[Proof of Claim~\ref{CLAIM:G-ast-bipartite}]
        It follows from Lemma~\ref{LEM:key-lemma-no-333-10-cases} that $H$ is $K_{3}$-free. 
        Therefore, by Theorem~\ref{THM:AES-Theorem}, it suffices to show that $\delta(H) > 2n/5$. 

        Fix an arbitrary vertex $v \in V$. 
        Note that 
        \begin{align*}
            d_{\vec{\mathcal{G}}}(v) 
            = \sum_{i=1}^{5} i \cdot d^{(i)}(v) 
            & \le 3\left(d^{(5)}(v)+ d^{(4)}(v)+d^{(3)}(v)\right) + 2 \left( d^{(5)}(v) + \cdots + d^{(1)}(v)\right) \\
            & \le  3 d_{H}(v)+ 2(n-1). 
        \end{align*}
        Combining it with the lower bound $d_{\vec{\mathcal{G}}}(v) \ge \delta(\vec{\mathcal{G}}) \ge \alpha n$, we obtain 
        \begin{align}\label{equ:degree-G-ast}
            d_{H}(v) 
            \ge \tfrac{1}{3} \left(\delta(\vec{\mathcal{G}}) - 2(n-1)\right)
            \ge \tfrac{1}{3} \left(\alpha n - 2(n-1)\right)
            \ge \tfrac{1}{3} \left(\tfrac{44}{13}n - 2(n-1)\right)
            > \tfrac{2}{5}n, 
        \end{align}
        which completes the proof of Claim~\ref{CLAIM:G-ast-bipartite}.
    \end{proof}

    Denote by $V_1$ and $V_2$ the corresponding two parts of $H$. It follows from the definition of $H$ that $\mu(uv) \le 2$ for every pair $\{u,v\} \in \tbinom{V_1}{2} \cup \tbinom{V_2}{2}$. 
    By symmetry, we may assume that $|V_1| \ge |V_2|$. 
    It follows from~\eqref{equ:degree-G-ast} that 
    \begin{align*}
        |V_2|
        = n - |V_1|
        \ge \delta(H)
        \ge \tfrac{1}{3} \left(\delta(\vec{\mathcal{G}}) - 2n\right), 
    \end{align*}
    which implies that 
    \begin{align}\label{EQ:bound-X-Y}
        \tfrac{\delta(\vec{\mathcal{G}})}{3} - \tfrac{2n}{3}
        \le |V_2|
        \le \tfrac{n}{2}
        \le |V_1|
        \le \tfrac{5n}{3} - \tfrac{\delta(\vec{\mathcal{G}})}{3}. 
    \end{align}
    Next, we consider the following graphs 
    \begin{align*}
        Q_{1} 
        \coloneqq \left\{\{u,v\} \in \tbinom{V_{1}}{2} \colon \mu(uv) = 2 \right\}
        \quad\text{and}\quad 
        Q_{2}
        \coloneqq \left\{\{u,v\} \in \tbinom{V_{2}}{2} \colon \mu(uv) = 2 \right\}. 
    \end{align*}

    \begin{claim}\label{CLAIM:X-connected-diameter-2}
        The following statements hold. 
        \begin{enumerate}[(i)]
            \item\label{CLAIM:X-connected-diameter-2-a} Suppose that $\alpha = 44/13$. 
            Then for every pair of distinct vertices $x, y \in V_1$, there exists a vertex $z \in V_1 \setminus \{x,y\}$ such that $\{xz, yz\} \subseteq Q_{1}$.
            \item\label{CLAIM:X-connected-diameter-2-b} Suppose that $\alpha = 58/17$. 
            Then for every pair of distinct vertices $x, y \in V_2$, there exists a vertex $z \in V_2 \setminus \{x,y\}$ such that $\{xz, yz\} \subseteq Q_{2}$. 
        \end{enumerate}
    \end{claim}
    \begin{proof}[Proof of Claim~\ref{CLAIM:X-connected-diameter-2}]
        Let us first prove~\eqref{CLAIM:X-connected-diameter-2-a}. 
        Fix two vertices $x, y \in V_1$. Let 
        \begin{align*}
            W_1 
            \coloneqq \left\{z \in V_1 \setminus \{x,y\} \colon \mu(xz)+\mu(yz)=4\right\}.
        \end{align*}
        %
        Note that 
        \begin{align*}
            d_{\vec{\mathcal{G}}}(x) + d_{\vec{\mathcal{G}}}(y) 
            &= 2 \mu(xy) + \sum_{w \in V \setminus \{x,y\} }\left(\mu(wx)+\mu(wy)\right) \\
            &\le 4  + 4|W_1| + 3 \left(|V_1|-2-|W_1|\right)+10|V_2|
            = |W_1|+7|V_2| + 3n-2
            < |W_1|+\tfrac{13 n}{2}, 
        \end{align*}
        where the last inequality follows from the fact $|V_2| \le n/2$.  
        
        Combining the inequality above with the assumption $\delta(\vec{\mathcal{G}})\ge \alpha n = 44n/13$, we obtain  
        \begin{align*}
            |W_1| 
            > d_{\vec{\mathcal{G}}}(x) + d_{\vec{\mathcal{G}}}(y) - \tfrac{13 n}{2} 
            \ge 2 \cdot \tfrac{44n}{13} - \tfrac{13 n}{2}
            = \tfrac{7n}{26}
            >0. 
        \end{align*}
        Therefore, there exists a vertex $z \in V_1 \setminus \{x,y\}$ such that $\mu(xz)+\mu(yz)=4$.
        Since $\max\{\mu(xz), \mu(yz)\} \le 2$, it must hold that $\mu(xz) = \mu(yz) = 2$. 
        This proves~\eqref{CLAIM:X-connected-diameter-2-a}.

        Next, we prove~\eqref{CLAIM:X-connected-diameter-2-b}. 
        Let 
        \begin{align*}
            W_1'
            \coloneqq \left\{w \in V_2 \colon \mu(wx)+\mu(wy)=4\right\}. 
        \end{align*}
        Similar to the proof of~\eqref{CLAIM:X-connected-diameter-2-a}, it suffices to show that $W_1' \neq \emptyset$. 
        
        It follows from~\eqref{EQ:bound-X-Y} that  
        \begin{align*}
            2  \delta(\vec{\mathcal{G}}) 
            \le d_{\vec{\mathcal{G}}}(x) + d_{\vec{\mathcal{G}}}(y) 
            &= 2\mu(xy) + \sum_{w \in V \setminus \{x,y\} }\left(\mu(wx)+\mu(wy)\right) \\
            &\le 4  + 4|W_1'|+3\left(|V_2|-2-|W_1'|\right)+10|V_1|\\
            & = |W_1'|+7|V_1|+3n-2\\
            & < |W_1'|+7\left(\tfrac{5}{3}n - \tfrac{1}{3}\delta(\vec{\mathcal{G}})\right)+3n
            = |W_1'| + \tfrac{44}{3}n - \tfrac{7}{3}\delta(\vec{\mathcal{G}}), 
        \end{align*}
        which, combined with the assumption $\alpha = 58/17$, implies that
        \begin{align*}
            |W_1'|
            > \tfrac{13}{3}\delta(\vec{\mathcal{G}})- \tfrac{44}{3}n 
            \ge \tfrac{13}{3} \cdot \tfrac{58}{17}n - \tfrac{44}{3}n 
            = \tfrac{2}{17}n
            >0.
        \end{align*}
        This shows that $W_1' \neq \emptyset$, and hence, proves~\eqref{CLAIM:X-connected-diameter-2-b}. 
    \end{proof}
    
    \begin{claim}\label{CLAIM:X-rainbowK_3-free}
        The following statements hold. 
        \begin{enumerate}[(i)]
            \item\label{CLAIM:X-rainbowK_3-free-a} Suppose that $\alpha = 44/13$. Then for every triple of vertices $\{x,y,z\} \subseteq V_1$, there are no distinct indices $i,j,k \in [5]$ such that $xy \in G_{i}$, $yz \in G_{j}$, and $xz \in G_{k}$. 
            \item\label{CLAIM:X-rainbowK_3-free-b} Suppose that $\alpha = 58/17$. Then for every triple of vertices $\{x,y,z\} \subseteq V_2$, there are no distinct indices $i,j,k \in [5]$ such that $xy \in G_{i}$, $yz \in G_{j}$, and $xz \in G_{k}$. 
        \end{enumerate}
    \end{claim}
    \begin{proof}[Proof of Claim~\ref{CLAIM:X-rainbowK_3-free}]
        Let us first prove~\eqref{CLAIM:X-rainbowK_3-free-a}.
        Suppose to the contrary that there exist distinct vertices $\{x,y,z\} \subseteq V_1$ and distinct indices $\{i,j,k\} \subseteq [5]$ such that $xy \in G_{i}$, $yz \in G_{j}$, and $xz \in G_{k}$. 
        Define 
        \begin{align*}
            W_2 
            \coloneqq \left\{w \in V_2 \colon \mu(wx)=\mu(wy)=\mu(wz)=5 \right\}.
        \end{align*}
        It is easy to see from the $\mathbb{K}_{4}$-freeness of $\vec{\mathcal{G}}$ that $W_2 = \emptyset$. 
        Therefore, $\mu(wx)+\mu(wy)+\mu(wz) \le 14$ for every $w \in V_2$. 
        Combining it with~\eqref{EQ:sum-le-2eU+} (applied to the set $\{x,y,z\}$), we obtain 
        \begin{align*}
            d_{\vec{\mathcal{G}}}(x)+d_{\vec{\mathcal{G}}}(y)+d_{\vec{\mathcal{G}}}(z) 
            & = 2 |\vec{\mathcal{G}}[\{x,y,z\}]| + \sum_{w \in V \setminus \{x,y,z\} }\left(\mu(wx)+\mu(wy)+\mu(wz)\right) \\
            & \le 12 + 6 \left(|V_1|-3\right)+ 14 |V_2|
            < 6 n + 8 |V_2|
            \le 10n, 
        \end{align*}
        where the last inequality follows from the fact that $|V_2| \le n/2$. 
        This is a contradiction to the lower bound 
        \begin{align*}
            d_{\vec{\mathcal{G}}}(x)+d_{\vec{\mathcal{G}}}(y)+d_{\vec{\mathcal{G}}}(z)
            \ge 3 \cdot \delta(\vec{\mathcal{G}})
            \ge 3 \alpha n
            = 3 \cdot \tfrac{44}{13} n
            > 10n. 
        \end{align*}

        Next, we prove~\eqref{CLAIM:X-rainbowK_3-free-b}.
        The proof is similar to that of~\eqref{CLAIM:X-rainbowK_3-free-a}. 
        Suppose to the contrary that there exist distinct vertices $\{x,y,z\} \subseteq V_2$ and distinct indices $\{i,j,k\} \subseteq [5]$ such that $xy \in G_{i}$, $yz \in G_{j}$, and $xz \in G_{k}$. 
        Let 
        \begin{align*}
            W_2' 
            \coloneqq \left\{w \in V_1 \colon \mu(wx)=\mu(wy)=\mu(wz)=5 \right\}. 
        \end{align*}
        It follows from the $\mathbb{K}_{4}$-freeness of $\vec{\mathcal{G}}$ that $W_2' = \emptyset$. 
        So it follows from~\eqref{EQ:sum-le-2eU+} (applied to the set $\{x,y,z\}$) that 
        \begin{align*}
            d_{\vec{\mathcal{G}}}(x)+d_{\vec{\mathcal{G}}}(y)+d_{\vec{\mathcal{G}}}(z) 
            & = 2 |\vec{\mathcal{G}}[\{x,y,z\}]| + \sum_{w \in V \setminus \{x,y,z\} }\left(\mu(wx)+\mu(wy)+\mu(wz)\right) \\
            & \le 12 + 6 \left(|V_2|-3\right)+ 14 |V_1|
            < 6 n + 8 |V_1|.
        \end{align*}
        Combining it with~\eqref{EQ:bound-X-Y}, we obtain 
        \begin{align*}
            \tfrac{5}{3} n - \tfrac{1}{3} \delta(\vec{\mathcal{G}})
            \ge |V_1|
            > \tfrac{1}{8} \left(d_{\vec{\mathcal{G}}}(x)+d_{\vec{\mathcal{G}}}(y)+d_{\vec{\mathcal{G}}}(z) - 6n\right)
            \ge \tfrac{1}{8} \left( 3 \delta(\vec{\mathcal{G}}) - 6n\right). 
        \end{align*}
        which implies $\delta(\vec{\mathcal{G}}) < 58n/17 = \alpha n$, a contradiction. 
        This proves~\eqref{CLAIM:X-rainbowK_3-free-b}.
    \end{proof}

    \begin{claim}\label{CLAIM:X-12-no345}
        The following statements hold. 
        \begin{enumerate}[(i)]
            \item\label{CLAIM:X-12-no345-a} Suppose that $\alpha = 44/13$. Let $x, y, z \in V_1$ be three distinct vertices and suppose that there exist distinct indices $i, j \in [5]$ such that $xy \in G_i \cap G_j$. 
            Then there is no index $k \in [5] \setminus \{i,j\}$ such that $yz \in G_{k}$. 
            \item\label{CLAIM:X-12-no345-b} Suppose that $\alpha = 58/17$. Let $x, y, z \in V_2$ be three distinct vertices and suppose that there exist distinct indices $i, j \in [5]$ such that $xy \in G_i \cap G_j$. 
            Then there is no index $k \in [5] \setminus \{i,j\}$ such that $yz \in G_{k}$. 
        \end{enumerate} 
    \end{claim}
    \begin{proof}[Proof of Claim~\ref{CLAIM:X-12-no345}]
        We prove~\eqref{CLAIM:X-12-no345-a} first. 
        Suppose to the contrary that there exist distinct vertices $\{x,y,z\} \subseteq V_1$ and distinct indices $\{i,j,k\} \subseteq [5]$ such that $xy \in G_{i} \cap G_{j}$ and $yz \in G_{k}$. 
        By symmetry, we may assume that $(i,j,k) = (1,2,3)$.
        Let 
        \begin{align*}
            W_3 
            \coloneqq \left\{w\in V_1 \colon \mu(wx) \ge 1\right\},~ 
            W_4 
            \coloneqq \left\{w\in V_1 \colon \mu(wy) = 2\right\},~
            W_5 
            \coloneqq \left\{w\in V_1 \colon \mu(wz) = 2\right\}
        \end{align*}
        Since $\delta(\vec{\mathcal{G}}) \le d_{\vec{\mathcal{G}}}(x) \le 2|W_3|+5|V_2|$, we have 
        \begin{align}\label{EQ:lower-bound-W3}
            |W_3| 
            \ge \tfrac{\delta(\vec{\mathcal{G}})-5|V_2|}{2}.
        \end{align}
        Since 
        \begin{align*}
            \delta(\vec{\mathcal{G}}) 
            \le d_{\vec{\mathcal{G}}}(y) 
            \le 2|W_4|+ \left(|V_1|-1-|W_4| \right)+5|V_2|
            < |W_4|+4|V_2|+n,
        \end{align*}
        we have 
        \begin{align}\label{EQ:lower-bound-W4}
            |W_4|
            > \delta(\vec{\mathcal{G}})-4|V_2|-n.
        \end{align}
        For a similar reason, we have 
        \begin{align}\label{EQ:lower-bound-W5}
            |W_5|
            > \delta(\vec{\mathcal{G}})-4|V_2|-n.
        \end{align}
        Combining~\eqref{EQ:lower-bound-W3},~\eqref{EQ:lower-bound-W4} and~\eqref{EQ:lower-bound-W5}, we obtain 
        \begin{align*}
            \left|W_3 \cap W_4 \cap W_5\right| 
            &\ge |W_3|+|W_4|+|W_5|-2|V_1| \\
            &> \tfrac{\delta(\vec{\mathcal{G}})-5|V_2|}{2} + 2\left(\delta(\vec{\mathcal{G}})-4|V_2|-n\right) - 2\left(n-|V_2| \right)  \\
            &= \tfrac{5}{2} \delta(\vec{\mathcal{G}}) -\tfrac{17}{2}|V_2|-4n 
            \ge \tfrac{5}{2} \cdot \tfrac{44}{13}n - \tfrac{17}{2} \cdot \tfrac{1}{2}n - 4n
            =  \tfrac{11}{52} n
            > 0.
        \end{align*}
        Therefore, there exists a vertex $w \in W_3 \cap W_4 \cap W_5$. 
        By the definitions of $W_3, W_4$, and $W_5$, we have $\mu(wx) \ge 1$ and $\mu(wy) = \mu(wz)=2$.
        
        Suppose, for the sake of contradiction, that $C(wy) \neq \{1,2\}$. 
        Fix $i \in C(wx)$.
        If $i \in \{1,2\}$, then let $j \in \{1,2\} \setminus \{i\}$. Since $C(wy) \neq \{1,2\} = \{i,j\}$, there exists some $k \in C(wy)\setminus \{i,j\}$.
        However, this implies that the triple $\{w,x,y\}$ contradicts Claim~\ref{CLAIM:X-rainbowK_3-free}~\eqref{CLAIM:X-rainbowK_3-free-a}. 
        So we may assume that $i \not\in \{1,2\}$. 
        Then fix some $k \in C(wy)\setminus \{i\}$, and then choose $j \in \{1,2\} \setminus \{k\}$. This also implies that the triple $\{w,x,y\}$ contradicts Claim~\ref{CLAIM:X-rainbowK_3-free}~\eqref{CLAIM:X-rainbowK_3-free-a}. 
        Therefore, $C(wy) = \{1,2\}$.
        
        Since $\mu(wz) = 2$, there exists some $i \in C(wz) \setminus \{3\}$. 
        Since $C(wy) = \{1,2\}$, there exists some $j \in C(wy) \setminus \{3,i\}$.
        However, this implies that the triple $\{w,y,z\}$ contradicts Claim~\ref{CLAIM:X-rainbowK_3-free}~\eqref{CLAIM:X-rainbowK_3-free-a}. This completes the proof of~\eqref{CLAIM:X-12-no345-a}. 

        The proof of~\eqref{CLAIM:X-12-no345-b} is similar to that of~\eqref{CLAIM:X-12-no345-a}.
        Suppose to the contrary that there exist distinct vertices $\{x,y,z\} \subseteq V_2$ and distinct indices $\{i,j,k\} \subseteq [5]$ such that $xy \in G_{i} \cap G_{j}$ and $yz \in G_{k}$. 
        Let 
        \begin{align*}
            W_3' 
            \coloneqq \left\{w \in V_2 \colon \mu(wx) \ge 1 \right\},~
            W_4'
            \coloneqq \left\{w \in V_2 \colon \mu(wy)=2 \right\},~
            W_5'
            \coloneqq \left\{w \in V_2 \colon \mu(wz)=2 \right\}. 
        \end{align*}
        We will show that $W_3' \cap W_4' \cap W_5' \neq \emptyset$.
        By an argument similar to the proof of~\eqref{CLAIM:X-12-no345-a}, this contradicts to Claim~\ref{CLAIM:X-rainbowK_3-free}~\eqref{CLAIM:X-rainbowK_3-free-b}. 

        First, it follows from $\delta(\vec{\mathcal{G}}) \le d_{\vec{\mathcal{G}}}(x) \le 2|W_3'|+5|V_1|$ that 
        \begin{align}\label{EQ:bound-W3'}
            |W_3'| 
            \ge \tfrac{\delta(\vec{\mathcal{G}})-5|V_1|}{2}. 
        \end{align}
        It follows from 
        \begin{align*}
            \delta(\vec{\mathcal{G}}) 
            \le d_{\vec{\mathcal{G}}}(y) 
            \le 2|W_4'|+ \left(|V_2|-1-|W_4'|\right) + 5|V_1|
            < |W_4'|+4|V_1|+n 
        \end{align*}
        that 
        \begin{align}\label{EQ:bound-W4'}
            |W_4'|
            > \delta(\vec{\mathcal{G}})-4|V_1|-n. 
        \end{align}
        For a similar reason, we have 
        \begin{align}\label{EQ:bound-W5'}
            |W_5'|
            > \delta(\vec{\mathcal{G}})-4|V_1|-n.  
        \end{align}
        Combining~\eqref{EQ:bound-W3'},~\eqref{EQ:bound-W4'},~\eqref{EQ:bound-W5'}, and~\eqref{EQ:bound-X-Y}, we obtain  
        \begin{align*}
            \left|W_3' \cap W_4' \cap W_5'\right| 
            & \ge |W_3'|+|W_4'|+|W_5'|-2|V_2| \\
            &> \tfrac{\delta(\vec{\mathcal{G}})-5|V_1|}{2} + 2\left(\delta(\vec{\mathcal{G}})-4|V_1|-n\right) -2\left(n-|V_1|\right) \\
            &= \tfrac{5}{2} \delta(\vec{\mathcal{G}}) -\tfrac{17}{2}|V_1|-4n \\
            &\ge \tfrac{5}{2} \delta(\vec{\mathcal{G}}) -\tfrac{17}{2}\left(\tfrac{5}{3}n - \tfrac{1}{3}\delta(\vec{\mathcal{G}})\right)-4n 
            = \tfrac{16}{3} \delta(\vec{\mathcal{G}}) - \tfrac{109}{6} n
            \ge \tfrac{16}{3} \cdot \tfrac{58}{17}n - \tfrac{109}{6} n
            = \tfrac{n}{34}.
        \end{align*}
        Therefore, $W_3' \cap W_4' \cap W_5' \neq \emptyset$, and hence, completes the proof of~\eqref{CLAIM:X-12-no345-b}.
    \end{proof}

    \begin{claim}\label{CLAIM:Q-in-Gi-G-j}
        The following statements hold. 
        \begin{enumerate}[(i)]
            \item\label{CLAIM:Q-in-Gi-G-j-a}  Suppose that $\alpha = 44/13$. Then there exist two distinct indices $i, j \in [5]$ such that $C(e) = \{i,j\}$ for every edge $e \in Q_{1}$.
            \item\label{CLAIM:Q-in-Gi-G-j-b}  Suppose that $\alpha = 58/17$. Then there exist two distinct indices $i', j' \in [5]$ such that $C(e) = \{i', j'\}$ for every edge $e \in Q_{2}$.
        \end{enumerate}
    \end{claim}
    \begin{proof}[Proof of Claim~\ref{CLAIM:Q-in-Gi-G-j}]
        Since the proofs of the two conclusions are nearly identical, we will present only the proof of~\eqref{CLAIM:Q-in-Gi-G-j-a}. 
        It follows easily from Claim~\ref{CLAIM:X-12-no345}~\eqref{CLAIM:X-12-no345-a} that if two edges $e, e' \in Q_{1}$ have nonempty intersection, then $C(e) = C(e')$. 
        Moreover, by Claim~\ref{CLAIM:X-connected-diameter-2}~\eqref{CLAIM:X-connected-diameter-2-a}, the graph $Q_1$ is connected. 
        Therefore, there exist two distinct indices $i, j \in [5]$ such that $C(e) = \{i,j\}$ for every $e\in Q_1$. 
    \end{proof}

    By symmetry, we may assume that $C(e) = \{1,2\}$ for every edge $e \in Q_{1}$. 
    \begin{claim}\label{CLAIM:edges-in-V1-color}
        For $\alpha = 44/13$, we have $C(e) \subseteq \{1,2\}$ for every $e \in \vec{\mathcal{G}}[V_1]$. 
        In particular, 
        \begin{align*}
            G_{3}[V_1] = G_{4}[V_1] = G_{5}[V_1] = \emptyset. 
        \end{align*}
    \end{claim}
    \begin{proof}[Proof of Claim~\ref{CLAIM:edges-in-V1-color}]
        We are done if $Q_{1} = \vec{\mathcal{G}}[V_1]$. So we may assume that $Q_{1} \neq \vec{\mathcal{G}}[V_1]$.
        Fix an arbitrary edge $xy \in \vec{\mathcal{G}}[V_1] \setminus Q_{1}$. 
        By Claim~\ref{CLAIM:X-connected-diameter-2}~\eqref{CLAIM:X-connected-diameter-2-a}, there exists vertex $z \in V_1 \setminus \{x,y\}$ such that $yz \in Q_{1}$ (and also $xz \in Q_{1}$).
        Applying Claim~\ref{CLAIM:X-12-no345}~\eqref{CLAIM:X-12-no345-a} to the triple $\{x,y,z\}$, we conclude that $C(xy) \subseteq C(yz) = \{1,2\}$.  
    \end{proof}

    Let us first complete the proof of Theorem~\ref{THM:AES-5-multigraph}~\eqref{THM:AES-5-multigraph-a}.
    By Claim~\ref{CLAIM:edges-in-V1-color}, it suffices to show that $G_1[V_2] = G_2[V_2] = \emptyset$, i.e., 
    \begin{align}\label{equ:color-e-in-V2-not-12}
        \text{$e \not\in G_1 \cup G_2$ for every $e\in \vec{\mathcal{G}}[V_2]$.}
    \end{align}
    Suppose to the contrary that $xy \in G_1 \cup G_2$ for some $xy \in \vec{\mathcal{G}}[V_2]$. 
    By symmetry, we may assume that $xy \in G_1$. 
    Let 
    \begin{align*}
        W_6 
        \coloneqq \left\{w \in V_1 \colon \mu(wx)+\mu(wy) \ge 9 \right\}
        \quad\text{and}\quad 
        W_7 
        \coloneqq \left\{w \in V_1 \colon \mu(wx)+\mu(wy) \ge 8 \right\}.
    \end{align*}
    First, we show that $W_6 \neq \emptyset$.
    By~\eqref{EQ:bound-X-Y}, we have 
    \begin{align*}
        2\delta(\vec{\mathcal{G}})
        \le d_{\vec{\mathcal{G}}}(x) + d_{\vec{\mathcal{G}}}(y) 
        & = 2 \mu(xy) + \sum_{w \in V \setminus \{x,y\} }\left(\mu(wx)+\mu(wy)\right) \\
        & \le 4 + 4\left(|V_2|-2\right)+ 10|W_6| + 8\left(|V_1|-|W_6|\right) \\
        & = 2|W_6| + 4|V_1| + 4n-4 \\
        & < 2|W_6|+4\left(\tfrac{5}{3}n - \tfrac{1}{3}\delta(\vec{\mathcal{G}}) \right)+4n
        = 2|W_6| + \tfrac{32}{3} n - \tfrac{4}{3}\delta(\vec{\mathcal{G}}),
    \end{align*}
    which implies that 
    \begin{align*}
        |W_6| 
        > \tfrac{5}{3}\delta(\vec{\mathcal{G}}) -\tfrac{16}{3}n  
        \ge \tfrac{5}{3} \cdot \tfrac{44}{13} n -\tfrac{16}{3}n
        = \tfrac{4}{13}n >0.
    \end{align*}
    Therefore, $W_6 \neq \emptyset$. 

    Fix a vertex  $z \in W_6$. By the definition of $W_6$, we have $\mu(xz)+\mu(yz) \ge 9$.
    Now consider the set $W_7$. 
    Similarly, it follows from~\eqref{EQ:bound-X-Y} that  
    \begin{align}\label{EQ:bound-W7}
        2\delta(\vec{\mathcal{G}})
        \le d_{\vec{\mathcal{G}}}(x)+d_{\vec{\mathcal{G}}}(y) 
        & = 2 \mu(xy) + \sum_{w \in V \setminus \{x,y\} }\left(\mu(wx)+\mu(wy)\right) \notag \\
        & \le 4 + 4\left(|V_2|-2\right)+10|W_7|+7 \left(|V_1|-|W_7|\right) \notag\\
        & = 3|W_7|+3|V_1|+4n-4 \\
        &< 3|W_7|+3\left(\tfrac{5}{3}n - \tfrac{1}{3}\delta(\vec{\mathcal{G}}) \right)+4n
        = 3|W_7| + 9 n - \delta(\vec{\mathcal{G}}), \notag 
    \end{align}
    which implies that 
    \begin{align*}
        |W_7| 
        > \delta(\vec{\mathcal{G}}) -3n 
        \ge \tfrac{44}{13} n -3n 
        =  \tfrac{5}{13}n 
        > 0. 
    \end{align*} 

    \begin{claim}\label{CLAIM:W7-mu-wz-1}
        We have $\mu(wz) \le 1$ for every $w\in W_7 \setminus \{z\}$. 
    \end{claim}
    \begin{proof}[Proof of Claim~\ref{CLAIM:W7-mu-wz-1}]
        Fix a vertex $w\in W_7 \setminus \{z\}$. 
        Recall from Claim~\ref{CLAIM:edges-in-V1-color} that $wz \not\in G_3 \cup G_4 \cup G_5$.
        From the definitions of $W_6$ and $W_7$, we have  
        \begin{align*}
            \mu(xz) + \mu(yz) + \mu(xw) + \mu(yw)
            \ge 8 + 9 
            = 17. 
        \end{align*}
        By Lemma~\ref{LEMMA:f5n-basic-property}~\eqref{ITEM:17->no-same-graph}, it follows that $C(xy) \cap C(wz) = \emptyset$. 
        Since $xy \in G_1$ by assumption, we conclude that $wz \not\in G_1$. Therefore, $C(wz) \subseteq [5]\setminus\{1,3,4,5\}$, which implies that $\mu(wz) \le 1$. 
        This completes the proof of Claim~\ref{CLAIM:W7-mu-wz-1}.
    \end{proof}
    
    It follows from Claim~\ref{CLAIM:W7-mu-wz-1} that 
    \begin{align*}
        \delta(\vec{\mathcal{G}}) 
        \le d_{\vec{\mathcal{G}}}(z) 
        \leq 5|V_2| + \left(|W_7|-1 \right) + 2\left(|V_1|-|W_7|\right) 
        < 3|V_2|+2n-|W_7|. 
    \end{align*}
    Combining it with~\eqref{EQ:bound-W7}, we obtain 
    \begin{align*}
        \delta(\vec{\mathcal{G}})
        \le \tfrac{1}{5} \left(9|V_2| + 3|V_1| + 10n - 4 \right)
        \le \tfrac{1}{5} \left( 13n + 6 |V_2| \right)
        \le \tfrac{1}{5} \left( 13n + 6 \cdot \tfrac{n}{2} \right)
        = \tfrac{16}{5} n
        < \tfrac{44}{13} n,
    \end{align*}
    a contradiction. 
    This proves~\eqref{equ:color-e-in-V2-not-12}, and hence, completes the proof of Theorem~\ref{THM:AES-5-multigraph}~\eqref{THM:AES-5-multigraph-a}. 

    Next, we complete the proof of Theorem~\ref{THM:AES-5-multigraph}~\eqref{THM:AES-5-multigraph-b}.
    Recall from Claim~\ref{CLAIM:Q-in-Gi-G-j}~\eqref{CLAIM:Q-in-Gi-G-j-b} that there exists $\{i',j'\} \subseteq [5]$ such that $C(e) = \{i',j'\}$ for every $e \in Q_2$. 
    It follows from~\eqref{equ:color-e-in-V2-not-12} that $\{i',j'\} \subseteq [5] \setminus \{1,2\}$. 
    By symmetry, we may assume that $\{i',j'\} = \{3,4\}$. 
    
    By repeating the proof of Claim~\ref{CLAIM:edges-in-V1-color} for $V_2$, we obtain the following claim. For completeness, we include its proof below.
    \begin{claim}\label{CLAIM:edges-in-V2-color}
        For $\alpha = 58/17$, we have $C(e) \subseteq \{3,4\}$ for every $e \in \vec{\mathcal{G}}[V_2]$. 
        In particular, 
        \begin{align*}
            G_{1}[V_2] = G_{2}[V_2] = G_{5}[V_2] = \emptyset. 
        \end{align*}
    \end{claim}
    \begin{proof}[Proof of Claim~\ref{CLAIM:edges-in-V2-color}]
        We are done if $Q_{2} = \vec{\mathcal{G}}[V_2]$. So we may assume that $Q_{2} \neq \vec{\mathcal{G}}[V_2]$.
        Fix an arbitrary edge $xy \in \vec{\mathcal{G}}[V_2] \setminus Q_{2}$. 
        By Claim~\ref{CLAIM:X-connected-diameter-2}~\eqref{CLAIM:X-connected-diameter-2-b}, there exists vertex $z \in V_2 \setminus \{x,y\}$ such that $yz \in Q_{2}$ (and also $xz \in Q_{2}$).
        Applying Claim~\ref{CLAIM:X-12-no345}~\eqref{CLAIM:X-12-no345-b} to the triple $\{x,y,z\}$, we conclude that $C(xy) \subseteq C(yz) = \{3,4\}$.  
    \end{proof}
    Claim~\ref{CLAIM:edges-in-V2-color} completes the proof of Theorem~\ref{THM:AES-5-multigraph}~\eqref{THM:AES-5-multigraph-b}. 
\end{proof}

\section{Concluding remarks}\label{SEC:Remarks}
Recall that the maximum size of an $n$-vertex $\mathbb{K}_{4}$-free $5$-multigraph was determined for all $n \in \mathbb{N}$ by Bellmann--Reiher~\cite{BR19} (see~\eqref{equ:Bellmann-Reiher-5-multigraph}). Since no proof was provided there, we present a short proof here for all $n \ge 632$, using Theorem~\ref{THM:AES-5-multigraph}~\eqref{THM:AES-5-multigraph-a}. 

\begin{theorem}\label{THM:Remark-exact-f_5(n)}
    For every integer $n \ge 632$, it holds that
    \begin{align*}
        \mathrm{ex}_{5}(n,\mathbb{K}_{4}) 
        = 2\tbinom{n}{2} + 3 \left\lfloor n^2/4 \right\rfloor. 
    \end{align*}
\end{theorem}
We remark that the assumption $n \ge 632$ can be relaxed to $n \ge 52$ if one uses the upper bound~\eqref{equ:Bellmann-Reiher-5-multigraph-upper-bound} given by Bellmann--Reiher in place of the trivial bound $5\tbinom{n}{2}$ in the following proof. 

We will use the following simple fact. 
\begin{fact}\label{FACT:nice-partition-multigraph}
    Let $\vec{\mathcal{G}}_{n}=(G_1,\ldots,G_5)$ be an $n$-vertex $\mathbb{K}_{4}$-free $5$-multigraph. 
    Suppose that $\vec{\mathcal{G}}_{n}$ admits a nice partition. 
    Then $|\vec{\mathcal{G}}_{n}| \le 2\tbinom{n}{2} + 3 \left\lfloor n^2/4 \right\rfloor$. 
\end{fact}

\begin{proof}[Proof of Theorem~\ref{THM:Remark-exact-f_5(n)}]
    Let $n \ge 632$ be an integer.  
    Let $\vec{\mathcal{G}}_{n}=(G_1,\ldots,G_5)$ be an $n$-vertex $\mathbb{K}_{4}$-free $5$-multigraph with $|\vec{\mathcal{G}}_{n}| = \mathrm{ex}_{5}(n,\mathbb{K}_{4})$, i.e., the maximum size.
    Let $V \coloneqq V(\vec{\mathcal{G}}_{n})$. 
    
    Let $g(m) \coloneqq 2\tbinom{m}{2} + 3 \left\lfloor m^2/4 \right\rfloor$ and $f(m) \coloneqq 44m/13$.
    Note that for every integer $m \ge 30$, 
    \begin{align*}
        g(m)-g(m-1) 
        = 2(m-1)+ 3 \left\lfloor m/2 \right\rfloor 
        > f(m). 
    \end{align*}
    Suppose that $\delta(\vec{\mathcal{G}}_{n}) \ge f(n)$.
    Then by Theorem~\ref{THM:AES-5-multigraph}~\eqref{THM:AES-5-multigraph-a}, $\vec{\mathcal{G}}_{n}$ admits a nice partition, and thus, $|\vec{\mathcal{G}}_{n}| \le g(n)$ by Fact~\ref{FACT:nice-partition-multigraph}.
    So we may assume that $\delta(\vec{\mathcal{G}}_{n}) < f(n)$. 

    Let $v_{n} \in V$ be a vertex of minimum degree in $\vec{\mathcal{G}}_{n}$. 
    Let $V_{n-1} \coloneqq V\setminus \{v_{n}\}$ and $\vec{\mathcal{G}}_{n-1} \coloneqq \vec{\mathcal{G}}_{n} - v_{n}$. 
    Then 
    \begin{align*}
        |\vec{\mathcal{G}}_{n-1}|
        = |\vec{\mathcal{G}}_{n}| - d_{\vec{\mathcal{G}}_{n}}(v_{n})
        = |\vec{\mathcal{G}}_{n}| - \delta(\vec{\mathcal{G}}_{n})
        > g(n) - f(n)
        > g(n-1),
    \end{align*}
    which means that $|\vec{\mathcal{G}}_{n-1}| \ge g(n-1)+1$. 

    Since $|\vec{\mathcal{G}}_{n-1}| \ge g(n-1)+1$, it follows from Theorem~\ref{THM:AES-5-multigraph}~\eqref{THM:AES-5-multigraph-a} and Fact~\ref{FACT:nice-partition-multigraph} that $\delta(\vec{\mathcal{G}}_{n-1}) < f(n-1)$. 
    Fix a vertex $v_{n-1} \in V_{n-1}$ of minimum degree in $\vec{\mathcal{G}}_{n-1}$. 
    Let $V_{n-2} \coloneqq V_{n-1}\setminus \{v_{n-1}\}$ and $\vec{\mathcal{G}}_{n-2} \coloneqq \vec{\mathcal{G}}_{n-1} - v_{n-1}$. 
    Then 
    \begin{align*}
        |\vec{\mathcal{G}}_{n-2}|
        = |\vec{\mathcal{G}}_{n-1}| - d_{\vec{\mathcal{G}}_{n-1}}(v_{n-1})
        & = |\vec{\mathcal{G}}_{n-1}| - \delta(\vec{\mathcal{G}}_{n-1}) 
        > g(n-1) + 1 - f(n-1) 
        > g(n-2) + 1,
    \end{align*}
    which means that $|\vec{\mathcal{G}}_{n-2}| \ge g(n-2)+2$. 

    Let $n_0 \coloneqq \lfloor 2\sqrt{3n-2}/3 \rfloor$, noting that $n_0 + 1 \ge \left\lfloor 2\sqrt{3 \times 632-2}/3 \right\rfloor + 1 = 30$ and $n \ge (9 n_0^2 + 8)/12$. 
    Let us repeat the above deleting minimum degree vertex process until we get a $5$-multigraph $\vec{\mathcal{G}}_{n_0}$ on $n_0$ vertices. 
    Then simple calculations show that   
    \begin{align*}
        |\vec{\mathcal{G}}_{n_0}|
        \ge g(n_0) + n-n_0
        & = 2\tbinom{n_0}{2} + 3 \left\lfloor \tfrac{n_0^2}{4} \right\rfloor +n - n_0 
        \ge 2\tbinom{n_0}{2} + 3 \left\lfloor \tfrac{n_0^2}{4} \right\rfloor + \tfrac{9 n_0^2 + 8}{12} - n_0 
        > 5\tbinom{n_0}{2}, 
    \end{align*}
    which is a contradiction to the trivial upper bound $|\vec{\mathcal{G}}_{n_0}| \le 5\tbinom{n_0}{2}$. 
\end{proof}

A minor modification of the proof in the present paper yields the following result on the hypergraph generalized Tur\'{a}n problem. 
\begin{theorem}\label{THM:generalized-Turan-K112}
    There exists $N_0$ such that for every $\mathbb{F}$-free $3$-graph $\mathcal{H}$ on $n \ge N_0$ vertices, 
    \begin{align*}
        \mathrm{N}(\mathbb{S}_2,\mathcal{H}) 
        \le \mathrm{N}(\mathbb{S}_2,\mathbb{B}_n), 
    \end{align*}
    and equality holds if and only if $\mathcal{H} \cong \mathbb{B}_n$. 
\end{theorem}

Several natural questions arise from the proofs and results presented in this paper. 
One such question is to extend Theorem~\ref{THM:AES-Fano-L2norm} to all positive integers $k$.
\begin{problem}\label{PROB:general-p}
    Let $k \ge 3$ be an integer. 
    Determine $\mathrm{ex}_{\ell_k}(n,\mathbb{F})$ for all sufficiently large $n$. 
\end{problem}

It would also be interesting to investigate whether the constant $58/17$ in Theorem~\ref{THM:AES-5-multigraph}~\eqref{THM:AES-5-multigraph-b} can be improved.
\begin{problem}\label{PROB:AES-multigraph-optimal}
    Determine the optimal value of $\alpha$ such that, for large $n$, every $n$-vertex $\mathbb{K}_{4}$-free $5$-multigraph $\vec{\mathcal{G}}=(G_1,\ldots, G_5)$ with $\delta(\vec{\mathcal{G}}) > \alpha n$ admits a good partition. 
\end{problem}

Extending Theorem~\ref{THM:Remark-exact-f_5(n)} to $m$-multigraphs for $m \ge 6$ also appears to be an interesting direction. 
\begin{problem}\label{PROB:multigraph-Turan-K4}
    Let $m \ge 6$ be an integer. 
    Determine $\mathrm{ex}_{m}(n,\mathbb{K}_{4})$ for all sufficiently large $n$. 
\end{problem}
\section*{Acknowledgments}
 
J.H. was supported by the National Key R\&D Program of China (No.~2023YFA1010202) and by the Central Guidance on Local Science and Technology Development Fund of Fujian Province (No.~2023L3003). 
X.L. was supported by the Excellent Young Talents Program (Overseas) of the National Natural Science Foundation of China. 

\newcommand{\etalchar}[1]{$^{#1}$}

\end{document}